\documentclass[12pt]{article}
\usepackage{amssymb,amsmath,amsthm,tikz,multirow,mathdots}
\usepackage{hyperref}
\usepackage{authblk}

\usepackage{xcolor}
\hypersetup{
    colorlinks,
    linkcolor={red!80!black},
    citecolor={green!60!black},
    urlcolor={blue!80!black}
}

\usepackage{algorithm}
\usepackage[noend]{algpseudocode}
\usepackage{caption}
\usepackage{subcaption}
\usepackage{comment} 
\usepackage{hhline} 
\usepackage{array} 

\usetikzlibrary{calc,arrows, arrows.meta, math} 

\newtheorem*{theorem*}{Theorem}

\theoremstyle{plain}
\newtheorem{thm}{Theorem}[section]
\newtheorem{lem}[thm]{Lemma}
\newtheorem{prop}[thm]{Proposition}

\newcommand{\bvert}{\raisebox{-2.5pt}{\vrule width 2pt height 12pt}}

\def\GetOne(#1,#2,#3,#4,#5,#6,#7,#8)#9{\def#9{#1}}
\def\GetTwo(#1,#2,#3,#4,#5,#6,#7,#8)#9{\def#9{#2}}
\def\GetThree(#1,#2,#3,#4,#5,#6,#7,#8)#9{\def#9{#3}}
\def\GetFour(#1,#2,#3,#4,#5,#6,#7,#8)#9{\def#9{#4}}
\def\GetFive(#1,#2,#3,#4,#5,#6,#7,#8)#9{\def#9{#5}}
\def\GetSix(#1,#2,#3,#4,#5,#6,#7,#8)#9{\def#9{#6}}
\def\GetSeven(#1,#2,#3,#4,#5,#6,#7,#8)#9{\def#9{#7}}
\def\GetEight(#1,#2,#3,#4,#5,#6,#7,#8)#9{\def#9{#8}}

\colorlet{HBlue}{teal!75!blue!95!white}
\colorlet{HOrange}{orange!80!black}

\hypersetup{colorlinks=true, linktocpage}

\makeatletter
\renewcommand*\env@matrix[1][c]{\hskip -\arraycolsep
  \let\@ifnextchar\new@ifnextchar
  \array{*\c@MaxMatrixCols #1}}
\makeatother

\providecommand{\keywords}[1]{\noindent \textit{Keywords:} #1}
\providecommand{\subject}[1]{\noindent \textit{Mathematics Subject Classification:} #1}

\title{Dihedral f-Tilings of the Sphere Induced by the M\"obius Triangle $(2,3,4)$}

\newcommand\CoAuthorMark{\footnotemark[\arabic{footnote}]} 

\author[1]{Catarina Avelino\thanks{This research was partially financed by Portuguese Funds through FCT (Fundação para a Ciência e a Tecnologia) within the projects UIDB/00013/2020 and UIDP/00013/2020 of CMAT-UTAD, Center of Mathematics of University of Minho, and projects UIDB/04621/2020 and UIDP/04621/2020 of CEMAT/IST-ID, Center for Computational and Stochastic Mathematics, Instituto Superior Técnico, University of Lisbon.}}
\author[2]{Hoi Ping Luk}
\author[3]{Altino Santos\protect\CoAuthorMark}

\affil[1,3]{\normalsize University of Tr\'as-os-Montes and Alto Douro, Vila Real, Portugal}
\affil[2]{\normalsize Z\'{a}pado\v{c}esk\'a univerzita v Plzni, Pilsen, Czech Republic}

\begin{document}

\date{\vspace{-1.4cm}}
\maketitle
\let\thefootnote\relax\footnotetext{E-mail addresses: cavelino@utad.pt (C. Avelino), hoi@connect.ust.hk (H. Luk), afolgado@utad.pt (A. Santos)}

\begin{abstract} We classify the special families of dihedral folding tilings of the sphere derived from the M\"obius triangle $(2,3,4)$. Our study emerges from the study of isometric foldings in the Riemann sphere and meets at the juncture of the triangle group $\Delta(2,3,4)$. The juxtaposition enables us to apply the classification theorem of edge-to-edge tilings of the sphere by congruent triangles and introduces a group theoretical method. The two prototiles of each family consist of the M\"obius triangle and another polygon induced by a reflection of the triangle group acting on the M\"obius triangle. To enumerate the tilings, we give two solutions to solve the associated constraint satisfaction problem. The methods are not exclusive to this problem and therefore applicable to similar problems of more general settings. 
\end{abstract}

 \\

\subject{05B45, 52C20, 51M09, 51M20} \\


\section{Introduction}

Tilings have been continuously studied and have continued to fascinate for centuries. Milestones include the classification of the Wallpaper groups \cite{fed,pol}, the solution to Hilbert's 18$^{\text{th}}$ problem \cite{bie, hee, rei}, the classification of the isohedral tilings of the plane \cite{gs}, Penrose tilings \cite{pen}, and most recently the classification of edge-to-edge monohedral tilings of the sphere \cite{awy, ay, cl, cly, cly2, gsy, ua, wy, wy2}, as well as the discovery of aperiodic monotiles \cite{smkg} of the plane. 

The essence of many tiling problems is about the existence of tilings under a set of constraints.  Hence they can be simply formulated as constraint satisfaction problems. However, they are often difficult to solve. Powerful tools are required. One of them is symmetry. Notably, symmetry is at the heart of the tiling problem of the plane mentioned before. Despite the immense challenges in the spherical counterparts, symmetry has made cameo appearances only as a consequence of the solutions. 

In this paper, our interest centres on the classification of spherical tilings by deploying symmetry as a tool on a particular set of constraints (to be stated below). Facing similar problems, our specific venture helps to shed light on the methodology: applying symmetry as a tool is not only advantageous but also out of necessity. 

The spherical tilings we study are {\em dihedral}, which means that in each tiling some tiles are congruent to one polygon and the rest of the tiles are congruent to a different polygon. The polygons are called the {\em prototiles} of a tiling. If a tiling has exactly one prototile, then it is called {\em monohedral}.

An {\em edge-to-edge} tiling means that no vertex lies in the interior of an edge. 

We focus on spherical dihedral tilings of {\em folding-type} (or {\em f-tilings} for short), which means that they are edge-to-edge,  all 
vertices have even degree $\ge4$ and the sums of alternate angles at each vertex are $\pi$.  The sums of alternate angles at a vertex of even degree means that for the $2k$ (for some integer $k\ge2$) angles labeled in cyclic order $\alpha_1, \alpha_2, ..., \alpha_{2k}$, we have
\begin{align}\label{Eq-alt-ang-sums}
\sum_{i=1}^k \alpha_{2i-1} = \sum_{i=1}^k \alpha_{2i} = \pi.
\end{align}
Obviously, the degree of a vertex is necessarily even if the condition on the sums of alternate angles is satisfied. 

The subject of this paper is the dihedral f-tilings having one prototile being the M\"obius triangle $(2,3,4)$ and another prototile being {\em induced} by the M\"obius triangle via the reflections of the triangle group $\Delta(2,3,4)$. The M\"obius triangle with edge labels $a,b,c$ (also denoted by $\Vert$, \textcolor{blue!80!black}{$\bvert$} and \textcolor{red}{$\vert$} respectively) and their opposite angles $\alpha, \beta, \gamma$ is illustrated the first picture of Figure \ref{Fig:prototiles-kite-Mobius-tri}. The angle values are $\alpha=\tfrac{1}{4}\pi$, $\beta=\tfrac{1}{3}\pi$ and $\gamma=\tfrac{1}{2}\pi$. The {\em induced prototile} is one of the three cases of gluing two mirror images of the M\"obius triangle along an edge: a kite (the second picture) along the $c$-edge, an isosceles triangle (the third picture) along the $b$-edge, or an isosceles triangle (the fourth picture) along the $a$-edge. The angles of the kite are denoted by $\alpha^2, \beta^2, \gamma, \gamma$, where $\alpha^2$ (resp. $\beta^2$) denotes $2$ copies of $\alpha$ (resp. $\beta$). The angle notations in the isosceles triangles are defined similarly. Let $\bar{x}=2x$ for $x=a,b$. Then the prototiles are referred to as {\em the M\"obius triangle}, {\em the kite} and {\em the isosceles triangles $\triangle \bar{a}c^2, \triangle \bar{b}c^2$}. The dihedral f-tilings with the M\"obius triangle and one of the three induced prototiles are referred to as the dihedral f-tilings {\em induced by} the M\"obius triangle.  Since two prototiles are often specified in tandem, the term ``prototiles" is dropped whenever it is clear and obvious in the context.

The other ways to generate the second prototile from the M\"obius triangle are through glided reflections along an edge. The dihedral f-tilings having the M\"obius triangle and such a second prototile have been classified in \cite{brealt}.

\begin{figure}[h!] 
\centering
\begin{tikzpicture}

\tikzmath{
\s=2;
\r=1;
\h=1;
}

\begin{scope}[] 

\foreach \aa in {-1} {

\tikzset{xscale=\aa}

\draw[double, line width=0.5]
	(90:\r) -- (0:\r)
;

\draw[blue!80!black, thick]
	(270:1.25*\r) -- (0:\r)
;

\draw[red]
	(90:\r) -- (270:1.25*\r)
;

\node at (0.7*\r,0.7*\r) {\small $a$};
\node at (0.7*\r,-0.7*\r) {\small $b$};
\node at (-0.25*\r,-0.05*\r) {\small $c$};

\node at (0.15*\r,0.5*\r) {\small $\beta$};
\node at (0.175*\r,-0.75*\r) {\small $\alpha$};
\node at (0.65*\r,-0.05*\r) {\small $\gamma$};

}

\end{scope}

\begin{scope}[xshift=\s cm] 

\foreach \aa in {-1,1} {

\tikzset{xscale=\aa}

\draw[red!20]
	(0,0.8*\r) -- (0,\r)
	(0,0.4*\r) -- (0,-0.6*\r)
	(0,-\r) -- (0,-1.25*\r)
;

\draw[double, line width=0.5]
	(90:\r) -- (0:\r)
;

\draw[blue!80!black, thick]
	(270:1.25*\r) -- (0:\r)
;

\node at (0.7*\r,0.7*\r) {\small $a$};
\node at (0.7*\r,-0.7*\r) {\small $b$};

}

\node at (90:0.55*\r) {\small $\beta^2$};
\node at (270:0.8*\r) {\small $\alpha^2$};
\node at (0.65*\r,-0.05*\r) {\small $\gamma$};
\node at (-0.65*\r,-0.05*\r) {\small $\gamma$};

\end{scope}

\begin{scope}[xshift=2.5*\s cm] 

\draw[blue!80!black!25, thick]
	(0,0.625*\r) -- (0,0.35*\r)
	(0,-0.625*\r) -- (0,-0.05*\r)
;

\foreach \aa in {-1,1} {

\tikzset{xscale=\aa}

\draw[double, line width=0.5]
	(0,-0.625*\r) -- (\r,-0.625*\r) 
;

\draw[red]
	(0,0.625*\r) -- (\r,-0.625*\r) 
;

\node at (0.75*\r,0*\r) {\small $c$};

\node at (0.6*\r,-0.45*\r) {\small $\beta$};

}

\node at (0, -0.95*\r) {\small $\bar{a}=2a$};

\node at (0,0.2*\r) {\small $\alpha^2$};

\end{scope}

\begin{scope}[xshift=4.15*\s cm] 

\draw[double, line width=0.5, black!25]
	(0,0.5*\r) --  (0,0.3*\r)
	(0,-0.5*\r) -- (0,-0.05*\r)
;

\foreach \aa in {-1,1} {

\tikzset{xscale=\aa}

\draw[red]
	(0,0.5*\r) -- (1.25*\r,-0.5*\r) 
;

\draw[blue!80!black, thick]
	(0,-0.5*\r) -- (1.25*\r,-0.5*\r) 
;

\node at (0.75*\r,0.2*\r) {\small $c$};

\node at (0.75*\r,-0.35*\r) {\small $\alpha$};

}

\node at (0, -0.85*\r) {\small $\bar{b}=2b$};

\node at (0,0.1*\r) {\small $\beta^2$};

\end{scope}

\end{tikzpicture}
\caption{The M\"obius triangle $\triangle abc$, the kite $\square a^2b^2$, and $\alpha=\tfrac{1}{4}\pi$, $\beta=\tfrac{1}{3}\pi$ and $\gamma=\tfrac{1}{2}\pi$ and the isosceles triangles $\triangle \bar{a}c^2$ and $\triangle\bar{b}c^2$}
\label{Fig:prototiles-kite-Mobius-tri}
\end{figure}
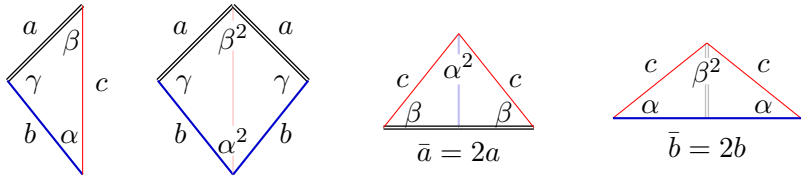

The hypothesis of the subject can be traced back to the juxtaposition of two well-studied mathematical topics with a wealth of impacts. The obvious one is the geometric realisation of the triangle groups, $\Delta (2,3,4)$ in our case, via a sequence of reflections across the edges of the M\"obius triangle. Thus the second prototile arises naturally from the realisation. The other is the study of isometric foldings in Riemannian manifolds \cite{rob} in which the singularities are realised by the embedded f-tilings. Hence the classifications of f-tilings lay the foundation for the theory thereof.

Historically, the monohedral f-tilings have been classified \cite{bre}. The dihedral f-tilings have been classified for the prototiles being a kite and a regular/isosceles triangle \cite{as, as2, as3, as4}. Our choice of prototiles is one step of a natural progression into an unchartered territory, at least to our best knowledge.  On one hand,  the challenges from a scalene triangle prototile are well documented in the classification of monohedral tilings \cite{cly, ua} and in dihedral f-tilings \cite{as5}. On the other hand, quadrilateral prototiles with few distinct edge lengths are typically difficult to handle \cite{cly}. We testify this conviction in the subject of this paper and demonstrate two new methods, one brings symmetry and group theory to the forefront and the other applies graph theory. This paper also generalises the method in \cite{as2015,as6}.

From now on, by the {\em dihedral f-tilings} we refer to the subject of the paper, which is organised as follows. The main result will be explained in Section \ref{Sec:Results}, and followed by a proof in Section \ref{Sec:Sym} using symmetry as well as an alternative proof in Section \ref{Sec:Graph} via graph isomorphism. The planar representations and the links for the corresponding 3D models are included in the Appendix (Section \ref{Sec:Appendix}).  In this section we also present the geometric and combinatorial structure of the dihedral f-tilings induced by the M\"obius triangle (2, 3, 4). 

\section{Main Results}
\label{Sec:Results}


\begin{theorem*}\label{main_th} There are a total of $123$ dihedral f-tilings induced by the M\"obius triangle $(2,3,4)$.  Among them, there are 
\begin{enumerate}
\item $104$ dihedral f-tilings if the prototiles are the M\"obius triangle and the kite; and
\item $12$ dihedral f-tilings if the prototiles are the M\"obius triangle and the isosceles triangle $\triangle \bar{a}c^2$; and
\item $7$ dihedral f-tilings if the prototiles are the M\"obius triangle and the isosceles triangle $\triangle \bar{b}c^2$.
\end{enumerate}
\end{theorem*}

Under this unified framework, we efficiently recover the results in \cite{as2015,as6} in items 2  and 3.  We also take the opportunity to point out a small error in Table 1 in \cite{as6}. The symmetry group of the f-tiling $F_6$ is $C_2 \times C_2$ and $F_4$ and $F_6$ are isomorphic.  

The structure of each tiling is best represented in a plane drawing. The plane drawings of the tilings are given in the Appendix. The open-ended edges (with or without the arrows) in the plane drawing of each tiling converge to a single vertex.  

Since the second prototile is induced by the M\"obius triangle, reversing the gluing in the second prototile results in two identical copies of the M\"obius triangle and doing so in each copy of the second prototile in a dihedral tiling results in monohedral tiling by the M\"obius triangles. There are two such tilings \cite{cly, ua}, the barycentric subdivision $B\mathcal{O}$ of the octahedron $\mathcal{O}$ and its flip modification $FB\mathcal{O}$ (plane drawings in Figure \ref{Fig:barycentric-octahedron-flip}). Hence, every dihedral f-tiling can be reduced to $B\mathcal{O}$ or $FB\mathcal{O}$ via such subdivision.



\begin{figure}[h!]
\centering
\begin{tikzpicture}[>=latex]

\tikzmath{
\xs=4;
}

\begin{scope}

\begin{scope}[rotate=45]

\foreach \a in {0,...,3} {

\tikzset{rotate=90*\a}

\draw[red]
	(45:0.8) -- (0:1.1) -- (-45:0.8)
;

\draw[->, red]
	(0:1.1) -- (0:1.5)
;

\draw[red]
	(0,0) -- (0.45,0)
	(45:0.8) -- (0.45,0) -- (-45:0.8);

}

\foreach \a in {0,...,3} {

\tikzset{rotate=90*\a}

\draw[blue!80!black, thick]
	(0:0.8) -- (45:0.8) -- (90:0.8)
;

\draw[blue!80!black, thick]
	(0,0) -- (45:0.8);

\draw[->,blue!80!black, thick]
	(45:0.8) -- (45:1.5)
;

}

\foreach \a in {0,...,3} {

\tikzset{rotate=90*\a}

\draw[double, line width=0.5]
	(0:0.8) -- (0:1.1)
	(45:1.1) -- (0:1.1) -- (-45:1.1)
;

\draw[double, line width=0.5]
	(0:0.8) -- (0:0.45) -- (45:0.45) -- (90:0.45);

}

\end{scope}

\node at (0,-2) {\small $B\mathcal{O}$};

\end{scope} 
	
\begin{scope}[xshift=\xs cm]

\begin{scope}[rotate=45]

\foreach \a in {0,...,3} {
{
\tikzset{rotate=90*\a}

\draw[red]
	(45:0.8) -- (0:1.1) -- (-45:0.8)
;

\draw[->, red]
	(0:1.1) -- (0:1.5)
;

}

{
\tikzset{rotate=45+90*\a} 

\draw[red]
	(0,0) -- (0.45,0)
	(45:0.8) -- (0.45,0) -- (-45:0.8);

\draw[blue!80!black, thick]
	(0,0) -- (45:0.8);

\draw[double, line width=0.5]
	(0:0.8) -- (0:0.45) -- (45:0.45) -- (90:0.45);

}
}

\foreach \a in {0,...,3} {
{
\tikzset{rotate=90*\a}

\draw[blue!80!black, thick]
	(0:0.8) -- (45:0.8) -- (90:0.8);

\draw[->, blue!80!black, thick]
	(45:0.8) -- (45:1.5)
;

\draw[double, line width=0.5]
	(0:0.8) -- (0:1.1)
	(45:1.1) -- (0:1.1) -- (-45:1.1);
}

{
\tikzset{rotate=45+90*\a}

\draw[blue!80!black, thick]
	(0,0) -- (45:0.8);

\draw[double, line width=0.5]
	(0:0.8) -- (0:0.45) -- (45:0.45) -- (90:0.45);

}
}

\foreach \a in {0,...,3} {
{
\tikzset{rotate=90*\a}

\draw[double, line width=0.5]
	(0:0.8) -- (0:1.1)
	(45:1.1) -- (0:1.1) -- (-45:1.1);
}

{
\tikzset{rotate=45+90*\a} 

\draw[double, line width=0.5]
	(0:0.8) -- (0:0.45) -- (45:0.45) -- (90:0.45);

}
}

\end{scope} 

\node at (0,-2) {\small $FB\mathcal{O}$};

\end{scope}

\end{tikzpicture}
\caption{The plane drawings of the barycentric subdivision of the octahedron $B\mathcal{O}$ and its flip modification $FB\mathcal{O}$; the arrows in each picture converge to a single vertex}
\label{Fig:barycentric-octahedron-flip}
\end{figure}
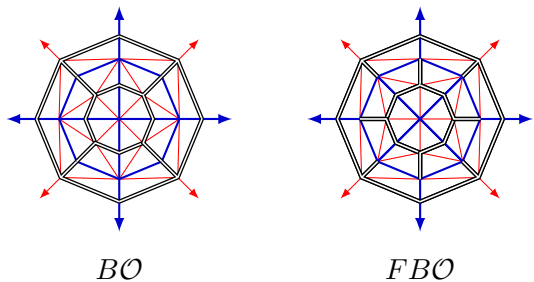


In view of the ``universality" of $B\mathcal{O}, FB\mathcal{O}$ in our classification, the dihedral f-tilings are determined by the presence (or the absence) of each $x$-edge in one of $B\mathcal{O}, FB\mathcal{O}$, subject to the folding conditions for a fixed $x=a, b$ or $c$. We call a result from this process an {\em $x$-edge assignment} (in $B\mathcal{O}$ or $FB\mathcal{O}$) or simply an {\em edge assignment}. For example, when $x=c$, the dashed lines in Figure \ref{Fig:edge-assign-BP8-FBP8} indicate the locations of the presence or absence of $c$-edges in $B\mathcal{O}$ and $FB\mathcal{O}$ respectively. 
In $B\mathcal{O}$, the labels, $T_i$ for $i \in I:=\{1,...,8\}$ and $S_j$ for $j \in J:=\{1,...,6\}$, represent the vertices corresponding to the cube (dual to the octahedron) and the octahedron, respectively.  In $FB\mathcal{O}$,  the vertices after a $\frac{1}{4}\pi$-rotation in the inner hemisphere of $B\mathcal{O}$ are denoted as $T_1', T_2', T_3', T_4'$ and $S_2', S_3', S_4', S_5'$. The above, which outlines the proof for the main theorem, is effectively a constraint satisfaction problem that can be resolved by computer. It is worth-noting that the same method works for dihedral or multihedral tilings without the folding condition.

\begin{figure}[h!]
\centering
\begin{tikzpicture}[>=latex]

\tikzmath{
\s=3;
\r=1.5; 
\x=0.4*\r;
\l=3;
\th=360/4;
}

\begin{scope}[] 

\begin{scope}

\foreach \a in {0,1,2,3} {

\tikzset{rotate=\a*\th}
	
\draw[double, line width=0.5]
	(\x,\x) -- (-\x,\x)
	(\l*\x,\l*\x) -- (-\l*\x,\l*\x)
	(\x,\x) -- (\l*\x,\l*\x) 
;

\draw[blue!80!black]
	(0,0) -- (2*\x,0)
;

\draw[->, blue!80!black]
	(2*\x,0) -- (4*\x,0)
;

}

\foreach \a in {0,...,3} {

\tikzset{rotate=\a*\th}

\draw[red, dashed]
	(0,0) -- (\x,\x)
	(2*\x,0) -- (\x,\x)
	(0,2*\x) -- (\x,\x)
	(0,2*\x) -- (\l*\x,\l*\x)
	(2*\x,0) -- (\l*\x,\l*\x)
	(4*\x,0) -- (\l*\x,\l*\x)
	(0,4*\x) -- (\l*\x,\l*\x)
;

}

\draw[blue!80!black] (0,0) circle (2*\x);

\node[inner sep=0.3,draw=blue!80!black,fill=white,shape=circle] at (0,0) {\tiny \textcolor{blue!80!black}{$S_1$}};
\node[inner sep=0.3,draw=blue!80!black,fill=white,shape=circle] at (2*\x,0) {\tiny \textcolor{blue!80!black}{$S_2$}};
\node[inner sep=0.3,draw=blue!80!black,fill=white,shape=circle] at (0,2*\x) {\tiny \textcolor{blue!80!black}{$S_3$}};
\node[inner sep=0.3,draw=blue!80!black,fill=white,shape=circle] at (-2*\x,0) {\tiny \textcolor{blue!80!black}{$S_4$}};
\node[inner sep=0.3,draw=blue!80!black,fill=white,shape=circle] at (0,-2*\x) {\tiny \textcolor{blue!80!black}{$S_5$}};

\node[inner sep=0.3,draw=blue!80!black,fill=white,shape=circle] at (4*\x,0) {\tiny \textcolor{blue!80!black}{$S_6$}};
\node[inner sep=0.3,draw=blue!80!black,fill=white,shape=circle] at (-4*\x,0) {\tiny \textcolor{blue!80!black}{$S_6$}};
\node[inner sep=0.3,draw=blue!80!black,fill=white,shape=circle] at (0,4*\x) {\tiny \textcolor{blue!80!black}{$S_6$}};
\node[inner sep=0.3,draw=blue!80!black,fill=white,shape=circle] at (0,-4*\x) {\tiny \textcolor{blue!80!black}{$S_6$}};

\node[inner sep=0.3,draw=black,fill=white,shape=circle] at (\x,\x) {\tiny \textcolor{red}{$T_1$}};
\node[inner sep=0.3,draw=black,fill=white,shape=circle] at (-\x,\x) {\tiny \textcolor{red}{$T_2$}};
\node[inner sep=0.3,draw=black,fill=white,shape=circle] at (-\x,-\x) {\tiny \textcolor{red}{$T_3$}};
\node[inner sep=0.3,draw=black,fill=white,shape=circle] at (\x,-\x) {\tiny \textcolor{red}{$T_4$}};

\node[inner sep=0.3,draw=black,fill=white,shape=circle] at (\l*\x,\l*\x) {\tiny \textcolor{red}{$T_5$}};
\node[inner sep=0.3,draw=black,fill=white,shape=circle] at (-\l*\x,\l*\x) {\tiny \textcolor{red}{$T_6$}};
\node[inner sep=0.3,draw=black,fill=white,shape=circle] at (-\l*\x,-\l*\x) {\tiny \textcolor{red}{$T_7$}};
\node[inner sep=0.3,draw=black,fill=white,shape=circle] at (\l*\x,-\l*\x) {\tiny \textcolor{red}{$T_8$}};

\end{scope} 

\begin{scope}[xshift=2*\s cm] 

\foreach \a in {0,1,2,3} {

\tikzset{rotate=\a*\th}
	
\draw[double, line width=0.5]
	(90:\x) -- (90+\th:\x)
	(\x,0) -- (2*\x,0)
	(\l*\x,\l*\x) -- (-\l*\x,\l*\x)
	(45:2*\x) -- (\l*\x,\l*\x) 
;

\draw[blue!80!black]
	(0,0) -- (45:2*\x) 
;

\draw[->, blue!80!black]
	(2*\x,0) -- (4*\x,0)
;

}

\foreach \a in {0,...,3} {

\tikzset{rotate=\a*\th}

\draw[red, densely dashed]
	(0,0) -- (\x,0)
;

\draw[red, dashed]
	(\x,0) -- (45:2*\x)
	(0,\x) -- (45:2*\x)
	(0,2*\x) -- (\l*\x,\l*\x)
	(2*\x,0) -- (\l*\x,\l*\x)
	(4*\x,0) -- (\l*\x,\l*\x)
	(0,4*\x) -- (\l*\x,\l*\x)
;

}

\draw[blue!80!black] (0,0) circle (2*\x);

\node[inner sep=0.3,draw=blue!80!black,fill=white,shape=circle] at (0,0) {\tiny \textcolor{blue!80!black}{$S_1$}};

\node[inner sep=0.3,draw=blue!80!black,fill=white,shape=circle] at (2*\x,0) {\tiny \textcolor{blue!80!black}{$S_2$}};
\node[inner sep=0.3,draw=blue!80!black,fill=white,shape=circle] at (0,2*\x) {\tiny \textcolor{blue!80!black}{$S_3$}};
\node[inner sep=0.3,draw=blue!80!black,fill=white,shape=circle] at (-2*\x,0) {\tiny \textcolor{blue!80!black}{$S_4$}};
\node[inner sep=0.3,draw=blue!80!black,fill=white,shape=circle] at (0,-2*\x) {\tiny \textcolor{blue!80!black}{$S_5$}};

\node[inner sep=0.05,draw=blue!80!black,fill=white,shape=circle] at (45:2*\x) {\tiny \textcolor{blue!80!black}{$S_3'$}};
\node[inner sep=0.05,draw=blue!80!black,fill=white,shape=circle] at (135:2*\x) {\tiny \textcolor{blue!80!black}{$S_4'$}};
\node[inner sep=0.05,draw=blue!80!black,fill=white,shape=circle] at (225:2*\x) {\tiny \textcolor{blue!80!black}{$S_5'$}};
\node[inner sep=0.05,draw=blue!80!black,fill=white,shape=circle] at (315:2*\x) {\tiny \textcolor{blue!80!black}{$S_2'$}};

\node[inner sep=0.3,draw=blue!80!black,fill=white,shape=circle] at (4*\x,0) {\tiny \textcolor{blue!80!black}{$S_6$}};
\node[inner sep=0.3,draw=blue!80!black,fill=white,shape=circle] at (-4*\x,0) {\tiny \textcolor{blue!80!black}{$S_6$}};
\node[inner sep=0.3,draw=blue!80!black,fill=white,shape=circle] at (0,4*\x) {\tiny \textcolor{blue!80!black}{$S_6$}};
\node[inner sep=0.3,draw=blue!80!black,fill=white,shape=circle] at (0,-4*\x) {\tiny \textcolor{blue!80!black}{$S_6$}};

\node[inner sep=0.05,draw=black,fill=white,shape=circle] at (\x,0) {\tiny \textcolor{red}{$T_1'$}};
\node[inner sep=0.05,draw=black,fill=white,shape=circle] at (0,\x) {\tiny \textcolor{red}{$T_2'$}};
\node[inner sep=0.05,draw=black,fill=white,shape=circle] at (-\x,0) {\tiny \textcolor{red}{$T_3'$}};
\node[inner sep=0.05,draw=black,fill=white,shape=circle] at (0,-\x) {\tiny \textcolor{red}{$T_4'$}};


\node[inner sep=0.3,draw=black,fill=white,shape=circle] at (\l*\x,\l*\x) {\tiny \textcolor{red}{$T_5$}};
\node[inner sep=0.3,draw=black,fill=white,shape=circle] at (-\l*\x,\l*\x) {\tiny \textcolor{red}{$T_6$}};
\node[inner sep=0.3,draw=black,fill=white,shape=circle] at (-\l*\x,-\l*\x) {\tiny \textcolor{red}{$T_7$}};
\node[inner sep=0.3,draw=black,fill=white,shape=circle] at (\l*\x,-\l*\x) {\tiny \textcolor{red}{$T_8$}};


\end{scope} 

\end{scope} 

\end{tikzpicture}
\caption{Locations of $c$-assignments in $B\mathcal{O}$ and $FB\mathcal{O}$}
\label{Fig:edge-assign-BP8-FBP8}
\end{figure}
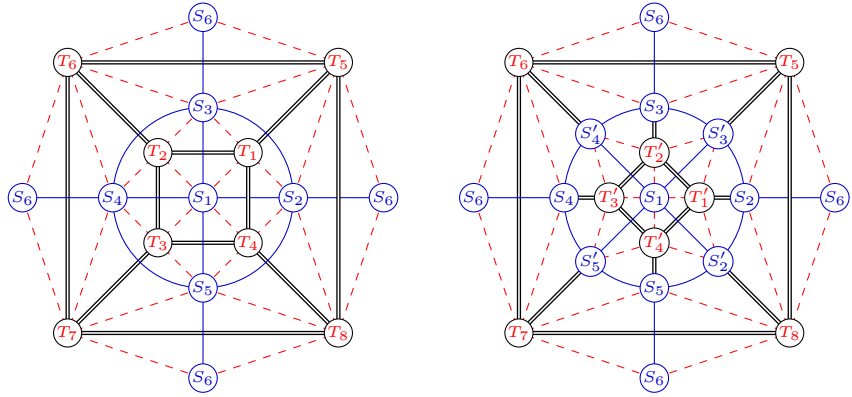





We establish two useful facts for the proof below.

\begin{lem}\label{Lem-suff-alt-ang-sum} If a vertex has a partition into adjacent angle pairs and the adjacent angles share the same value in each pair, then the vertex satisfies the sums of alternate angles \eqref{Eq-alt-ang-sums}.
\end{lem}

\begin{proof} A partition in the hypothesis has angle pairs $(\alpha_1, \alpha_2), ..., (\alpha_{2k-1}, \alpha_{2k})$, where $\alpha_1 = \alpha_2 = \lambda_1$, and $\alpha_3=\alpha_4=\lambda_2$, ..., and $\alpha_{2k-1} = \alpha_{2k}=\lambda_k$. Then $\sum_{i=1}^k \alpha_{2i-1} = \sum_{i=1}^k \alpha_{2i} = \sum_{i=1}^{k} \lambda_i$, which implies the assertion.
\end{proof}

\begin{lem}\label{Lem-even-deg-iff-alt-sum} In a dihedral tiling induced by the M\"obius triangle $(2,3,4)$, a vertex has an even degree if and only if the sums of alternate angles are $\pi$.
\end{lem}

\begin{proof} The backward implication is obvious. It suffices to consider a vertex of even degree. The presence of a copy of an induced prototile is equivalent to merging two adjacent M\"obius triangles, i.e., removing their common edge, depending on $a,b$ or $c$. All the induced dihedral tilings come from one of $B\mathcal{O}, FB\mathcal{O}$. Both have all vertices of even degrees.

For $B\mathcal{O}$, the vertices are either $\alpha^8$ (Figure \ref{Fig:S-even-edges}, first picture) partitioned by alternating $b,c$ or $\beta^6$ partitioned by alternating $a,c$. By Lemma \ref{Lem-suff-alt-ang-sum}, the vertices $\alpha^8$ satisfy \eqref{Eq-alt-ang-sums} and the same is still true after removing even number of $c$-edges (Figure \ref{Fig:S-even-edges}, first four pictures). The same argument prevails for the case of $\triangle \bar{a}c^2$ where the $b$-edges are removed. The arguments for $\triangle \bar{b}c^2$ and $\beta^6$ are analogous.


\begin{figure}[h!] 
\centering
\begin{tikzpicture}

\tikzmath{
\s=2.25;
\r=1.2;
\rr=0.05*\r;
\th=360/4;
\x=\r*cos(0.5*\th);
}

\foreach \xs in {0,1,2,3} {

\tikzset{xshift=\xs*\s cm}

\foreach \a in {0,1,2,3} {

\tikzset{rotate=\a*\th}

\draw[blue!80!black, thick]
	(0,0) -- (\x,0)
;

	(\x, \x) -- (-\x, \x)
;

}
}

\begin{scope}[xshift=0*\s cm]

\foreach \a in {0,...,3} {

\tikzset{rotate=\a*\th}

\draw[red]
	(0,0) -- (\x,\x)
;

\node at (0.45*\x, 0.2*\x) {\small $\alpha$};
\node at (0.2*\x, 0.45*\x) {\small $\alpha$};

}



\end{scope}

\begin{scope}[xshift=1*\s cm]

\foreach \aa in {-1,1} {

\tikzset{xscale=\aa}

\draw[red]
	(0,0) -- (\x,\x)
;

\node at (0.45*\x, 0.2*\x) {\small $\alpha$};
\node at (0.2*\x, 0.45*\x) {\small $\alpha$};
\node at (0.3*\x, -0.3*\x) {\small $\alpha^2$};

}



\end{scope}

\begin{scope}[xshift=2*\s cm]

\foreach \a in {0,1} {

\tikzset{rotate=\a*180}

\draw[red]
	(0,0) -- (\x,\x)
;

\node at (0.45*\x, 0.2*\x) {\small $\alpha$};
\node at (0.2*\x, 0.45*\x) {\small $\alpha$};
\node at (0.3*\x, -0.3*\x) {\small $\alpha^2$};

}



\end{scope}

\begin{scope}[xshift=3*\s cm]

\foreach \aa in {-1,1} {

\tikzset{xscale=\aa}

\node at (0.3*\x, 0.3*\x) {\small $\alpha^2$};
\node at (0.3*\x, -0.3*\x) {\small $\alpha^2$};

}



\end{scope}

\begin{scope}[xshift=4*\s cm]

\foreach \a in {0,2,3} {

\tikzset{rotate=\a*\th}

\draw[blue!80!black, thick]
	(0,0) -- (\x,0)
;

}

\draw[double, line width=0.5]
	 (0, 0) -- (0, \x)
;

\draw[red]
	(0,0) -- (\x,-\x)
	(0,0) -- (-\x,-\x)
;

\node at (-0.3*\x, 0.3*\x) {\small $\gamma$};
\node at (0.3*\x, 0.3*\x) {\small $\gamma$};
\node at (-0.45*\x, -0.2*\x) {\small $\alpha$};
\node at (-0.2*\x, -0.45*\x) {\small $\alpha$};
\node at (0.45*\x, -0.2*\x) {\small $\alpha$};
\node at (0.2*\x, -0.45*\x) {\small $\alpha$};

\end{scope}

\begin{scope}[xshift=5*\s cm]

\foreach \a in {0,2,3} {

\tikzset{rotate=\a*\th}

\draw[blue!80!black, thick]
	(0,0) -- (\x,0)
;

}

\draw[double, line width=0.5]
	 (0, 0) -- (0, \x)
;

\node at (-0.3*\x, 0.3*\x) {\small $\gamma$};
\node at (0.3*\x, 0.3*\x) {\small $\gamma$};
\node at (-0.3*\x, -0.3*\x) {\small $\alpha^2$};
\node at (0.3*\x, -0.3*\x) {\small $\alpha^2$};

\end{scope}

\end{tikzpicture}
\caption{The even $c$-assignments at a vertex}
\label{Fig:S-even-edges}
\end{figure}
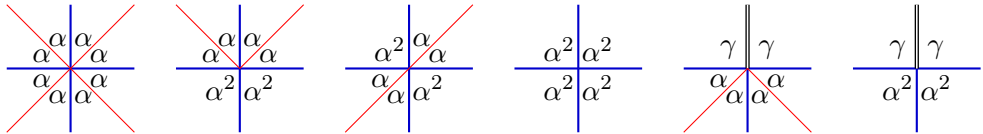

For $FB\mathcal{O}$, the only exception is the vertex $\alpha^4\gamma^2$ (Figure \ref{Fig:S-even-edges}, fifth picture). The removal of even number of $c$-edges will result in the sixth picture and obviously the sums of alternating angles are $\pi$. The removal of even number (two) of $b$-edges will result in a quadrilateral with edge configuration $acac$ (Figure \ref{Fig:barycentric-octahedron-flip}), a contradiction. Hence no $b$-edges are to be removed at $\alpha^4\gamma^2$.
\end{proof}

\section{The Symmetry Approach}
\label{Sec:Sym}

As discussed in the paragraph after the main theorem, the dihedral f-tilings induced by the M\"obius triangle are obtained by $x$-edge assignments in $B\mathcal{O}$ and $FB\mathcal{O}$ for fixed $x=a,b$ or $c$. It suffices to determine which assignments are unique up to isomorphism. To achieve it, we use the geometric models having equivalent locations of the vertices in $\mathbb{R}^3$ for the edge assignments and apply group actions by their corresponding automorphism group to check for isomorphism between two edge assignments. The models for $B\mathcal{O}$ and $FB\mathcal{O}$ are the deltoidal icositetrahedron and the pseudo-deltoidal icositetrahedron respectively. The spherical deltoidal icositetrahedron denoted by $o\,\mathcal{C}$ (Conway's notation) is illustrated in the first picture of Figure \ref{Fig:E2E-NE2E-3D} whereas the spherical pseudo-deltoidal icositetrahedron denoted by $Fo\,\mathcal{C}$, is obtained from the deltoidal icositetrahedron by twisting a half (upper or lower hemisphere) along the equator, the same modification as applied to $FB\mathcal{O}$. The automorphism group for $o\,\mathcal{C}$ is the triangle group $G=\Delta (2,3,4)$, which is also the octahedral symmetry, whereas the automorphism group $G'$ 
for $Fo\,\mathcal{C}$ is a simple consequence of the Orbit-Stabiliser Theorem and the classification of finite subgroups of $SO(3)$. Using this strategy, we prove the next two propositions.


\begin{figure}[h!]
\centering
\begin{subfigure}{0.35\textwidth} \centering
\includegraphics[height=3cm]{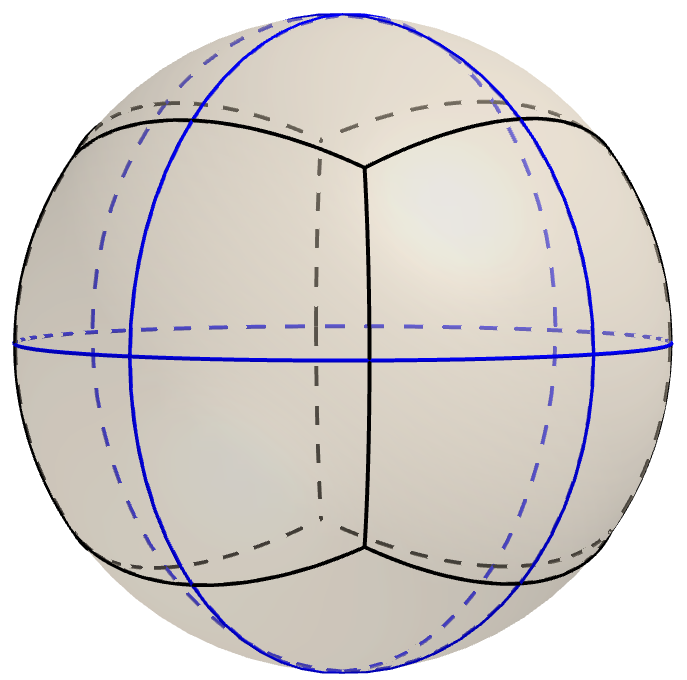}
\end{subfigure} 
\begin{subfigure}{0.35\textwidth} \centering
\includegraphics[height=3cm]{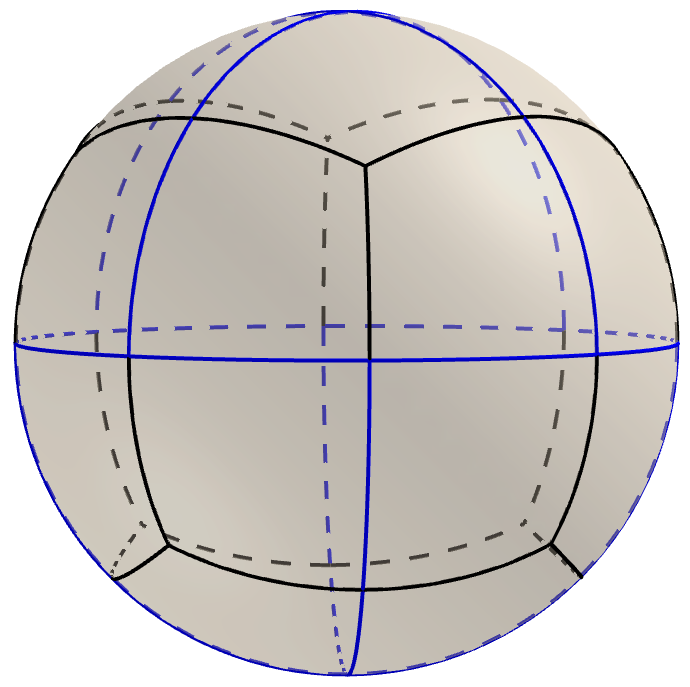}
\end{subfigure} 
\caption{The spherical deltoidal icositetrahedron and the spherical pseudo-deltoidal icositetrahedron}
\label{Fig:E2E-NE2E-3D}
\end{figure}

\begin{prop}\label{Prop-BP8} Up to isomorphism under $G$, in each of the edge assignments below, with the exception of one corresponding to $B\mathcal{O}$, there are
\begin{enumerate}
\item $75$ $c$-edge assignments corresponding to dihedral f-tilings;
\item $12$ $b$-edge assignments corresponding to dihedral f-tilings;
\item $5$ $a$-edge assignments corresponding to dihedral f-tilings.
\end{enumerate}
\end{prop}

These edge assignments can be seen in Figures \ref{Fig:BP8-c-I}, 
\ref{Fig:BP8-b},  \ref{Fig:BP8-a} in the Appendix. The symmetry group of each tiling is also given.

\begin{proof} We will prove the most complicated case, the $c$-edge assignments, and the other cases are analogous and we leave them to the readers as an exercise. 

The monohedral tiling $B\mathcal{O}$ consists of eight barycentrically subdivided octants $X$ in the first picture of Figure \ref{Fig:x-y-octants}. The $c$-edge assignments will result in octants in form of $X$ or $Y$ (Figure \ref{Fig:x-y-octants}, second picture). This means that each resulting octant has either $3$ $c$-edges or $1$ $c$-edge. To reduce the redundancy, we will first determine the assignments (up to isomorphism) of the $c$-degree of $3$ or $1$ in each octant, and we call the procedure {\em $c$-degree assignment}. Redundancy will be ruled out via the automorphism group $G$.

\begin{figure}[h!] 
\centering
\begin{tikzpicture}

\tikzmath{
\s=3;
\r=1;
\th=360/3;
\x=\r*cos(\th/2);
}

\foreach \xs in {0,1} {

\tikzset{xshift=\xs*\s cm}

\foreach \a in {0,1,2} {

\tikzset{rotate=\a*\th}

\draw[blue!80!black, thick]
	(90:\r) -- (90+\th:\r)
;

\draw[double, line width=0.5]
	(0,0) -- (90-0.5*\th:\x)
;

}

}

\begin{scope}[] 

\foreach \a in {0,1,2} {

\tikzset{rotate=\a*\th}

\draw[red]
	(0,0) -- (90:\r)
;

}

\node at (0,-0.8*\r) {\small $X$};

\end{scope} 

\begin{scope}[xshift=\s cm] 

\draw[red]
	(0,0) -- (90:\r)
;

\node at (0,-0.8*\r) {\small $Y$};

\end{scope}

\end{tikzpicture}
\caption{The $X, Y$ octants of the sphere}
\label{Fig:x-y-octants}
\end{figure}
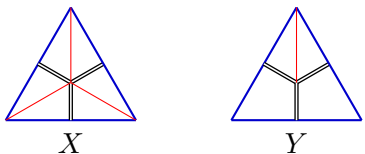

To formalise the $c$-degree assignment, we make use of the notations in the first picture of Figure \ref{Fig:edge-assign-BP8-FBP8}. Let $d: \{ T_i: i=1,...,8 \} \to \{ 1, 3 \}$ be a function where $\{ T_i: i=1,...,8 \}$ is the set of the centres of the octants and $\{ 1, 3 \}$ is obviously the intended set of degrees. We can represent $d$ as an ordered $8$-tuple $D=[d(T_1), d(T_2), ..., d(T_8)]$. 

The automorphism group $G$ has an order $48$. Two edge (resp. degree) assignments are called {\em isomorphic} if there is a group action $\sigma \in G$ mapping one assignment into the other. We will elaborate in details as follows. 

The representation of each $\sigma \in G$ is given by one of the following matrices with the assignments of $+$ or $-$ to the entries with $1$'s,
\begin{align*}
&\begin{bmatrix}
\pm 1 & 0 & 0 \\
0 & \pm 1 & 0 \\
0 & 0 & \pm 1
\end{bmatrix}, \quad
\begin{bmatrix}
0 & 0 & \pm 1 \\
0 & \pm 1 & 0 \\
\pm 1 & 0 & 0
\end{bmatrix}, \quad
\begin{bmatrix}
0 & \pm 1 & 0 \\
\pm 1 & 0 & 0 \\
0 & 0 & \pm 1
\end{bmatrix}, \\
&\begin{bmatrix}
0 & \pm 1 & 0 \\
0 & 0 & \pm 1 \\
\pm 1 & 0 & 0
\end{bmatrix}, \quad
\begin{bmatrix}
0 & 0 & \pm 1 \\
\pm 1 & 0 & 0 \\
0 & \pm 1 & 0
\end{bmatrix}, \quad
\begin{bmatrix}
\pm 1 & 0 & 0 \\
0 & 0 & \pm 1 \\
0 & \pm 1 & 0
\end{bmatrix}.
\end{align*}

For example, element {$-\sigma_7 \in G$} from \eqref{Eq-Oh} has the matrix representation 
\begin{align} \label{Eq-sigma-eg}
{-\sigma_7} = \begin{bmatrix}[r]
-1 & 0 & 0 \\
0 & 0 & ~\,\,1 \\
0 & -1 & 0
\end{bmatrix}.
\end{align}
The vertices $T_i$'s and $S_j$'s in $o\,\mathcal{C}$ (Figure \ref{Fig:edge-assign-BP8-FBP8}) are geometrically represented by the following position vectors in $\mathbb{R}^3$, where $\tau=\tfrac{2\sqrt{2}+1}{7}$,
\begin{alignat*}{4}
&T_1 =\tau\begin{bmatrix}[r]
 1 \\
 1 \\
 1
\end{bmatrix}, \quad
& T_2 =\tau\begin{bmatrix}[r]
 -1 \\
 1 \\
 1
\end{bmatrix}, \quad
& T_3 = \tau\begin{bmatrix}[r]
 -1 \\
 -1 \\
 1
\end{bmatrix}, \quad
& T_4 = \tau\begin{bmatrix}[r]
 1 \\
 -1 \\
 1
\end{bmatrix}, \\[1.5ex]
&T_5 = \tau\begin{bmatrix}[r]
 1 \\
 1 \\
 -1
\end{bmatrix}, \quad
&T_6 = \tau\begin{bmatrix}[r]
 -1 \\
 1 \\
 -1
\end{bmatrix}, \quad
&T_7 =\tau\begin{bmatrix}[r]
 -1 \\
 -1 \\
 -1
\end{bmatrix}, \quad
&T_8 = \tau\begin{bmatrix}[r]
 1 \\
 -1 \\
 -1
\end{bmatrix},
 \end{alignat*}
\begin{align*}
S_1 = \begin{bmatrix}[r]
 0 \\
 0 \\
 1
\end{bmatrix}, \
S_2 = \begin{bmatrix}[r]
 1 \\
 0 \\
 0
\end{bmatrix}, \
S_3 = \begin{bmatrix}[r]
 0 \\
 1 \\
 0
\end{bmatrix}, \
S_4 = \begin{bmatrix}[r]
 -1 \\
 0 \\
 0
\end{bmatrix}, \
S_5 = \begin{bmatrix}[r]
 0 \\
 -1 \\
 0
\end{bmatrix}, \
S_6 = \begin{bmatrix}[r]
 0 \\
 0 \\
 -1
\end{bmatrix}.
\end{align*}

The one-one correspondence between the vertices and the position vectors allows us to abuse the same notations $T_i$'s and $S_j$'s for both the vertices and the position vectors.

Each $\sigma\in G$ induces a permutation on the entries of the ordered $8$-tuple, denoted by $\sigma D$. For example,  {$-\sigma_7$}  in \eqref{Eq-sigma-eg} permutes $T_1,T_2,T_3,T_4,T_5,T_6$, $T_7,T_8$ into $T_6,T_5,T_1,T_2,T_7,T_8,T_4,T_3$ respectively. In other words,
\begin{align*}
{-\sigma_7} = \begin{bmatrix}
1 &  2 & 3 & 4 & 5 & 6 & 7 & 8 \\
6 & 5 & 1 & 2 & 7 & 8 & 4 & 3 
\end{bmatrix}.
\end{align*}

Two degree assignments, $D_1=[ d_1(T_1),$ $d_1(T_2), ..., d_1(T_8) ]$ and $D_2= [ d_2(T_1), d_2(T_2), ..., d_2(T_8)  ]$, where $d_l: \{T_i \}_{i \in I} \to \{ 1,3 \}$ for $l=1,2$, are called {\em isomorphic} if there is a $\sigma \in G$ such that $\sigma D_1 = D_2$, where $\sigma D_1 := [d_1(T_{\sigma(i)})]_{i \in I}$. For example, $D_1=[1,1,3,3,3,3,1,1]$ and $D_2=[3,3,1,1,1,1,3,3]$ are isomorphic as {$-\sigma_7 D_1 = D_2$} (Figure \ref{Fig:isom-deg-assign}) for {$-\sigma_7$} in \eqref{Eq-sigma-eg}.

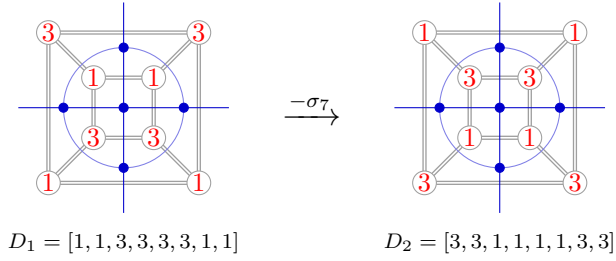
\begin{figure}[h!] 
\centering
\begin{tikzpicture}[>=latex]

\tikzmath{
\s=5;
\x=0.4;
\l=2.5;
\th=360/4;
\rr=0.15*\x;
}

\foreach \xs in {0,...,1} {

\tikzset{xshift=\xs*\s cm}

\foreach \a in {0,1,2,3} {

\tikzset{rotate=\a*\th}
	
\draw[gray!75, double, line width=0.5]
	(\x,\x) -- (-\x,\x)
	(\l*\x,\l*\x) -- (-\l*\x,\l*\x)
	(\x,\x) -- (\l*\x,\l*\x) 
;
}

\foreach \a in {0,1,2,3} {

\tikzset{rotate=\a*\th}

\draw[blue!80!black]
	(0,0) -- (2*\x,0) -- (3.5*\x,0)
;

\draw[->, blue!80!black]
;

}

\draw[blue!80!black!50] (0,0) circle (2*\x);

\filldraw[blue!80!black] 
	(0,0) circle (\rr)
	(2*\x,0) circle (\rr)
	(0,2*\x) circle (\rr)
	(-2*\x,0) circle (\rr)
	(0,-2*\x) circle (\rr)
;

}

\node at (0.5*\s, 0) {$\xrightarrow[]{{-\sigma_7}}$};

\begin{scope}[]

\node[inner sep=0.3,draw=gray!75,fill=white,shape=circle] at (\x,\x) {\footnotesize \textcolor{red}{$1$}};
\node[inner sep=0.3,draw=gray!75,fill=white,shape=circle] at (-\x,\x) {\footnotesize \textcolor{red}{$1$}};
\node[inner sep=0.3,draw=gray!75,fill=white,shape=circle] at (-\x,-\x) {\footnotesize \textcolor{red}{$3$}};
\node[inner sep=0.3,draw=gray!75,fill=white,shape=circle] at (\x,-\x) {\footnotesize \textcolor{red}{$3$}};


\node[inner sep=0.3,draw=gray!75,fill=white,shape=circle] at (\l*\x,\l*\x) {\footnotesize \textcolor{red}{$3$}};
\node[inner sep=0.3,draw=gray!75,fill=white,shape=circle] at (-\l*\x,\l*\x) {\footnotesize \textcolor{red}{$3$}};
\node[inner sep=0.3,draw=gray!75,fill=white,shape=circle] at (-\l*\x,-\l*\x) {\footnotesize \textcolor{red}{$1$}};
\node[inner sep=0.3,draw=gray!75,fill=white,shape=circle] at (\l*\x,-\l*\x) {\footnotesize \textcolor{red}{$1$}};


\node at (0,-1.8*\l*\x) {\scriptsize $D_1=[1,1,3,3,3,3,1,1]$};

\end{scope}

\begin{scope}[xshift=\s cm]

\node[inner sep=0.3,draw=gray!75,fill=white,shape=circle] at (\x,\x) {\footnotesize \textcolor{red}{$3$}};
\node[inner sep=0.3,draw=gray!75,fill=white,shape=circle] at (-\x,\x) {\footnotesize \textcolor{red}{$3$}};
\node[inner sep=0.3,draw=gray!75,fill=white,shape=circle] at (-\x,-\x) {\footnotesize \textcolor{red}{$1$}};
\node[inner sep=0.3,draw=gray!75,fill=white,shape=circle] at (\x,-\x) {\footnotesize \textcolor{red}{$1$}};


\node[inner sep=0.3,draw=gray!75,fill=white,shape=circle] at (\l*\x,\l*\x) {\footnotesize \textcolor{red}{$1$}};
\node[inner sep=0.3,draw=gray!75,fill=white,shape=circle] at (-\l*\x,\l*\x) {\footnotesize \textcolor{red}{$1$}};
\node[inner sep=0.3,draw=gray!75,fill=white,shape=circle] at (-\l*\x,-\l*\x) {\footnotesize \textcolor{red}{$3$}};
\node[inner sep=0.3,draw=gray!75,fill=white,shape=circle] at (\l*\x,-\l*\x) {\footnotesize \textcolor{red}{$3$}};


\node at (0,-1.8*\l*\x) {\scriptsize $D_2=[3,3,1,1,1,1,3,3]$};

\end{scope}

\end{tikzpicture}
\caption{Isomorphic degree assignments $D_1, D_2$}
\label{Fig:isom-deg-assign}
\end{figure}
Up to isomorphism under $G$, the argument yields $22$ $c$-degree assignments. 

We obtain an edge assignment based on a degree assignment and an assigned edge at $T_i$ is $T_iS_{i_k}$ for a neighbouring $S_{i_k}$ (where $i_k \in J$, with respect to the red dashed lines in the first picture of Figure \ref{Fig:edge-assign-BP8-FBP8}). Then two edge assignments are {\em isomorphic} means that there is a $\sigma \in G$ such that the assigned edge sets $\{ T_iS_{i_k} \}_{i \in I, i_k \in J}$ and $\{ T_{\sigma( i )} S_{\sigma(i_k)} \}_{i \in I, i_k \in J}$ are the same, otherwise they are called {\em non-isomorphic}. For example, we have 
\begin{align*}
E=\{ &T_1S_3, T_2S_3, T_3S_1, T_3S_4, T_3S_5, T_4S_1, T_4S_2, T_4S_5, \\ 
&T_5S_2, T_5S_3, T_5S_6, T_6S_3, T_6S_4, T_6S_6, T_7S_5, T_8S_5 \}, \\ 
E'= \{ &T_1S_3, T_2S_3, T_3S_1, T_3S_4, T_3S_5, T_4S_1, T_4S_2, T_4S_5, \\
&T_5S_2, T_5S_3, T_5S_6, T_6S_3, T_6S_4, T_6S_6, T_7S_6, T_8S_6 \},
\end{align*}
which are non-isomorphic edge assignments obtained from $D_1$ (Figure \ref{Fig:non-isom-edge-assign}).

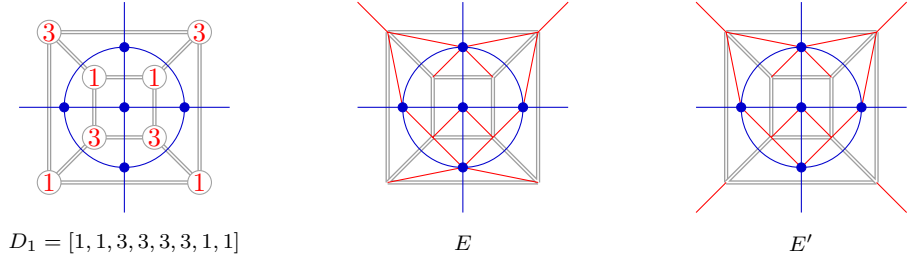
\begin{figure}[h!] 
\centering
\begin{tikzpicture}

\tikzmath{
\s=4.5;
\x=0.4;
\l=2.5;
\th=360/4;
\rr=0.15*\x;
}

\foreach \xs in {0,...,2} {

\tikzset{xshift=\xs*\s cm}

\foreach \a in {0,1,2,3} {

\tikzset{rotate=\a*\th}
	
\draw[gray!75, double, line width=0.5]
	(\x,\x) -- (-\x,\x)
	(\l*\x,\l*\x) -- (-\l*\x,\l*\x)
	(\x,\x) -- (\l*\x,\l*\x) 
;}}

\begin{scope}[]

\node[inner sep=0.3,draw=gray!75,fill=white,shape=circle] at (\x,\x) {\footnotesize \textcolor{red}{$1$}};
\node[inner sep=0.3,draw=gray!75,fill=white,shape=circle] at (-\x,\x) {\footnotesize \textcolor{red}{$1$}};
\node[inner sep=0.3,draw=gray!75,fill=white,shape=circle] at (-\x,-\x) {\footnotesize \textcolor{red}{$3$}};
\node[inner sep=0.3,draw=gray!75,fill=white,shape=circle] at (\x,-\x) {\footnotesize \textcolor{red}{$3$}};


\node[inner sep=0.3,draw=gray!75,fill=white,shape=circle] at (\l*\x,\l*\x) {\footnotesize \textcolor{red}{$3$}};
\node[inner sep=0.3,draw=gray!75,fill=white,shape=circle] at (-\l*\x,\l*\x) {\footnotesize \textcolor{red}{$3$}};
\node[inner sep=0.3,draw=gray!75,fill=white,shape=circle] at (-\l*\x,-\l*\x) {\footnotesize \textcolor{red}{$1$}};
\node[inner sep=0.3,draw=gray!75,fill=white,shape=circle] at (\l*\x,-\l*\x) {\footnotesize \textcolor{red}{$1$}};


\node at (0,-1.8*\l*\x) {\scriptsize $D_1=[1,1,3,3,3,3,1,1]$};

\draw[blue!80!black] (0,0) circle (2*\x);

\foreach \a in {0,1,2,3} {

\tikzset{rotate=\a*\th}

\draw[blue!80!black]
	(0,0) -- (2*\x,0) -- (3.5*\x,0)
;
}
\filldraw[blue!80!black] 
	(0,0) circle (\rr)
	(2*\x,0) circle (\rr)
	(0,2*\x) circle (\rr)
	(-2*\x,0) circle (\rr)
	(0,-2*\x) circle (\rr)
;
\end{scope}

\begin{scope}[xshift=\s cm]

\foreach \aa in {-1, 1} {

\tikzset{xscale=\aa}

\draw[red]
	(\x,\x) -- (0,2*\x) -- (\l*\x, \l*\x)
	(\x,-\x) -- (0,0)
	(\x,-\x) -- (2*\x,0)
	(\x,-\x) -- (0,-2*\x)
	(\l*\x, \l*\x) -- (2*\x,0)
	(\l*\x, \l*\x) -- (3.5*\x,3.5*\x)
	(\l*\x, -\l*\x) -- (0,-2*\x)
;}

\draw[blue!80!black] (0,0) circle (2*\x);

\foreach \a in {0,1,2,3} {

\tikzset{rotate=\a*\th}

\draw[blue!80!black]
	(0,0) -- (2*\x,0) -- (3.5*\x,0)
;}
\filldraw[blue!80!black] 
	(0,0) circle (\rr)
	(2*\x,0) circle (\rr)
	(0,2*\x) circle (\rr)
	(-2*\x,0) circle (\rr)
	(0,-2*\x) circle (\rr)
;
\node at (0,-1.8*\l*\x) {\scriptsize $E$};

\end{scope}

\begin{scope}[xshift=2*\s cm]

\foreach \aa in {-1, 1} {

\tikzset{xscale=\aa}

\draw[red]
	(\x,\x) -- (0,2*\x) -- (\l*\x, \l*\x)
	(\x,-\x) -- (0,0)
	(\x,-\x) -- (2*\x,0)
	(\x,-\x) -- (0,-2*\x)
	(\l*\x, \l*\x) -- (2*\x,0)
	(\l*\x, \l*\x) -- (3.5*\x,3.5*\x)
	(\l*\x, -\l*\x) -- (3.5*\x,-3.5*\x)
;}
\draw[blue!80!black] (0,0) circle (2*\x);

\foreach \a in {0,1,2,3} {

\tikzset{rotate=\a*\th}

\draw[blue!80!black]
	(0,0) -- (2*\x,0) -- (3.5*\x,0)
;}
\filldraw[blue!80!black] 
	(0,0) circle (\rr)
	(2*\x,0) circle (\rr)
	(0,2*\x) circle (\rr)
	(-2*\x,0) circle (\rr)
	(0,-2*\x) circle (\rr)
;
\node at (0,-1.8*\l*\x) {\scriptsize $E'$};

\end{scope}

\end{tikzpicture}
\caption{Edge assignments based on $D_1$}
\label{Fig:non-isom-edge-assign}
\end{figure}
Up to isomorphism under $G$, the argument yields $76$ edge assignments,  where one of them is $B\mathcal{O}$.
The arguments for $a$-edge assignments and $b$-edge assignments are analogous.
\end{proof}


\begin{prop}\label{Prop-FBP8} Up to isomorphism under $G'$, in each of the edge assignments below, with the exception of one corresponding to $FB\mathcal{O}$, there are
\begin{enumerate}
\item $29$ $c$-edge assignments corresponding to dihedral f-tilings;
\item no $b$-edge assignments corresponding to dihedral f-tilings;
\item $2$ $a$-edge assignments corresponding to dihedral f-tilings.
\end{enumerate}
\end{prop}


These edge assignments can be seen in Figures \ref{Fig:X-Y-octants-ne2e} and \ref{Fig:FBP8-a}  in the Appendix. The symmetry group of each tiling is also given.

\begin{proof} We follow the same argument in the proof of Proposition \ref{Prop-BP8} by replacing $G$, $T_1, ..., T_8$, and $S_1, ..., S_6$ by $G'$ and $T_1', T_2', T_3', T_4', T_5, T_6, T_7, T_8$, $S_2', S_3', S_4', S_5'$ and $S_1, ..., S_6$ respectively.  Recall that the pseudo-deltoidal icositetrahedron $Fo\,\mathcal{C}$ is obtained from $o\,\mathcal{C}$ by rotating a hemisphere $\tfrac{1}{4}\pi$ clockwise along the equator (Figure \ref{Fig:edge-assign-BP8-FBP8}). Let $R$ denote such rotation that
\begin{align*}
R 
:=
\begin{bmatrix}[r]
\tfrac{1}{\sqrt{2}} & \tfrac{1}{\sqrt{2}}  & 0 \\
-\tfrac{1}{\sqrt{2}}  & \tfrac{1}{\sqrt{2}} & 0 \\
0~ & 0~ & ~1
\end{bmatrix},
\end{align*}
and again $\tau=\tfrac{2\sqrt{2}+1}{7}$.
Then the vertices $T_i$'s and $T_k'$'s and $S_j$'s and $S_l'$'s are vertices of $Fo\,\mathcal{C}$ represented by the following position vectors, 
\begin{alignat*}{4}
&T_1' = R \tau \begin{bmatrix}[r]
 1 \\
 1 \\
 1
\end{bmatrix}, \quad
&T_2' = R \tau\begin{bmatrix}[r]
 -1 \\
 1 \\
 1
\end{bmatrix}, \quad
&T_3' = R \tau\begin{bmatrix}[r]
 -1 \\
 -1 \\
 1
\end{bmatrix}, \quad
&T_4' = R \tau  \begin{bmatrix}[r]
 1 \\
 -1 \\
 1
\end{bmatrix}, \\[1.5ex]
&T_5 = \tau\begin{bmatrix}[r]
 1 \\
 1 \\
 -1
\end{bmatrix}, \quad
&T_6 =  \tau\begin{bmatrix}[r]
 -1 \\
 1 \\
 -1
\end{bmatrix}, \quad
&T_7 =  \tau\begin{bmatrix}[r]
 -1 \\
 -1 \\
 -1
\end{bmatrix}, \quad
&T_8 =  \tau\begin{bmatrix}[r]
 1 \\
 -1 \\
 -1
\end{bmatrix},
 \end{alignat*}

\begin{align*}
&S_1 = \begin{bmatrix}[r]
 0 \\
 0 \\
 1
\end{bmatrix}, \
S_2 = \begin{bmatrix}[r]
 1 \\
 0 \\
 0
\end{bmatrix}, \
S_3 = \begin{bmatrix}[r]
 0 \\
 1 \\
 0
\end{bmatrix}, \
S_4 = \begin{bmatrix}[r]
 -1 \\
 0 \\
 0
\end{bmatrix}, \
S_5 = \begin{bmatrix}[r]
 0 \\
 -1 \\
 0
\end{bmatrix}, \
S_6 = \begin{bmatrix}[r]
 0 \\
 0 \\
 -1
\end{bmatrix}, \\[1.5ex]
& ~~~~~~S_2' = \tfrac{1}{\sqrt{2}}\begin{bmatrix}[r]
 1 \\
 -1 \\
 0
\end{bmatrix}, \
S_3' = \tfrac{1}{\sqrt{2}}\begin{bmatrix}[r]
 1 \\
 1 \\
 0
\end{bmatrix}, \
S_4' = \tfrac{1}{\sqrt{2}}\begin{bmatrix}[r]
 -1 \\
 1 \\
 0
\end{bmatrix}, \
S_5' = \tfrac{1}{\sqrt{2}}\begin{bmatrix}[r]
 -1 \\
 -1 \\
 0
\end{bmatrix}.
\end{align*}

By the Orbit-Stabiliser Theorem, we determine that $G'$ is a group of order $16$. In fact,  the orbit of $S_1$ (equivalently $S_6$) is $\{ S_1, S_6 \}$ and the stabiliser of $G'_{S_1}$ fixes both $S_1, S_6$ and has $4$ rotational symmetries and $4$ mirror symmetries. Hence, we get $\vert G' \vert  =  \vert G' \cdot S_1 \vert \vert G'_{S_1} \vert = 2 \cdot 8$. Then the Classification Theorem of finite subgroups of $SO(3)$ determines the group elements in terms of the following matrices
\begin{equation*}
\begin{alignedat}{4}  \label{List-FT-Matrices}
&\begin{bmatrix}[r]
\,\,\,1 & 0 & 0 \\
0 & \,\,\,\,\,1 & 0 \\
0 & 0 & \,\,\,\,1
\end{bmatrix}, \
&\begin{bmatrix}[r]
0 & \,\,\,\,\,1 & 0 \\
\,\,\,1 & 0 & 0 \\
0 & 0 & \,\,\,\,1
\end{bmatrix}, \
&\begin{bmatrix}[r]
\,\,\,1 & 0 & 0 \\
0 & -1 & 0 \\
0 & 0 & \,\,\,\,1
\end{bmatrix},\
&\begin{bmatrix}[r]
0 & \,\,\,\,\,1 & 0 \\
-1 & 0 & 0 \\
0 & 0 & \,\,\,\,1
\end{bmatrix}, \\[1.5ex] 
&\begin{bmatrix}[r]
-1 & 0 & 0 \\
0 & -1 & 0 \\
0 & 0 & \,\,\,1
\end{bmatrix}, \
&\begin{bmatrix}[r]
0 & -1 & 0 \\
-1 & 0 & 0 \\
0 & 0 & \,\,\,1
\end{bmatrix}, \
&\begin{bmatrix}[r]
-1 & 0 & 0 \\
0 &  \,\,\,\,1 & 0 \\
0 & 0 & \,\,\,1
\end{bmatrix}, \
&\begin{bmatrix}[r]
0 & -1 & 0 \\
\,\,\,\,1 & 0 & 0 \\
0 & 0 & \,\,\,\,1
\end{bmatrix},\\[1.5ex] 
&\begin{bmatrix}[r]
\tfrac{-1}{\sqrt{2}} & \tfrac{-1}{\sqrt{2}}  & 0 \\
\tfrac{-1}{\sqrt{2}}  & \,\,\,\tfrac{1}{\sqrt{2}} & 0 \\
0 & 0 & -1
\end{bmatrix}, \
 &\begin{bmatrix}[r]
\tfrac{-1}{\sqrt{2}} & \tfrac{-1}{\sqrt{2}}  & 0 \\
\tfrac{1}{\sqrt{2}}  & \,\,\,\tfrac{-1}{\sqrt{2}} & 0 \\
0 & 0 & -1
\end{bmatrix}, \
&\begin{bmatrix}[r]
\tfrac{-1}{\sqrt{2}} & \tfrac{1}{\sqrt{2}}  & 0 \\
\tfrac{-1}{\sqrt{2}}  & \,\,\,\tfrac{-1}{\sqrt{2}} & 0 \\
0 & 0 & -1
\end{bmatrix}, \
&\begin{bmatrix}[r]
\tfrac{-1}{\sqrt{2}} & \tfrac{1}{\sqrt{2}}  & 0 \\
\tfrac{1}{\sqrt{2}}  & \,\,\,\tfrac{1}{\sqrt{2}} & 0 \\
0 & 0 & -1
\end{bmatrix}, \\[1.5ex]
&\begin{bmatrix}[r]
\tfrac{1}{\sqrt{2}} & \tfrac{-1}{\sqrt{2}}  & 0 \\
\tfrac{-1}{\sqrt{2}}  & \,\,\,\tfrac{-1}{\sqrt{2}} & 0 \\
0 & 0 & -1
\end{bmatrix}, \
&\begin{bmatrix}[r]
\tfrac{1}{\sqrt{2}} & \tfrac{-1}{\sqrt{2}}  & 0 \\
\tfrac{1}{\sqrt{2}}  &\,\,\, \tfrac{1}{\sqrt{2}} & 0 \\
0 & 0 & -1
\end{bmatrix}, \
&\begin{bmatrix}[r]
\tfrac{1}{\sqrt{2}} & \tfrac{1}{\sqrt{2}}  & 0 \\
\tfrac{-1}{\sqrt{2}}  & \,\,\,\tfrac{1}{\sqrt{2}} & 0 \\
0 & 0 & -1
\end{bmatrix}, \
&\begin{bmatrix}[r]
\tfrac{1}{\sqrt{2}} & \tfrac{1}{\sqrt{2}}  & 0 \\
\tfrac{1}{\sqrt{2}}  & \,\,\,\tfrac{-1}{\sqrt{2}} & 0 \\
0 & 0 & -1
\end{bmatrix}.
\end{alignedat}
\end{equation*}

Up to isomorphism under $G'$, we obtain $30$ $c$-degree assignments. By the same argument, up to isomorphism they result in $30$ $c$-edge assignments.

The arguments for $a$-edge assignments and $b$-edge assignments are analogous.
\end{proof}

\section{The Graph Isomorphism Approach}
\label{Sec:Graph}

In this section, we present a different solution to the problem. For the edge assignments derived from both $B\mathcal{O}$ and $FB\mathcal{O}$, the underlying graph of the resulting assignment using the same labels in Figure \ref{Fig:edge-assign-BP8-FBP8} is represented by the adjacency matrix $A$ as follows
\begin{align}
\label{Eq-Adj-Matrix-A}
A=
\begin{bmatrix}
M_T & M_{TS} \\
M_{TS}^{\dagger} & M_S
\end{bmatrix},
\end{align}
where $M_{TS}^{\dagger}$ is the transpose of $M_{TS}$. For the edge assignments in $B\mathcal{O}$, the matrices $M_T$ (resp. $M_{TS}$ and $M_S$) are the adjacency matrix between $T_{*}, T_{*}$ (resp. $T_{*}, S_{*}$, and $S_{*}, S_{*}$) such that
\begin{align}
\label{Eq-BP8-MT}
M_T &= \begin{bmatrix}
0 & T_{12} & 0 & T_{14} & T_{15} & 0 & 0 & 0 \\
T_{12} & 0 & T_{23} & 0 & 0 & T_{26} & 0 & 0 \\
0 & T_{23} & 0 & T_{34} & 0 & 0 & T_{37} & 0 \\
T_{14} & 0 & T_{34} & 0 & 0 & 0 & 0 & T_{48} \\
T_{15} & 0 & 0 & 0 & 0 & T_{56} & 0 & T_{58} \\
0 & T_{26} & 0 & 0 & T_{56} & 0 & T_{67} & 0 \\
0 & 0 & T_{37} & 0 & 0 & T_{67} & 0 & T_{78} \\
0 & 0 & 0 & T_{48} & T_{58} & 0 & T_{78} & 0
\end{bmatrix}, \\[0.8ex]
\label{Eq-BP8-MTS}
M_{TS} &= \begin{bmatrix}
T_{11} & T_{12} & T_{13} & 0 & 0 & 0 \\
T_{21} & 0 & T_{22} & T_{23} & 0 & 0 \\
T_{31} & 0 & 0 & T_{32} & T_{33} & 0 \\
T_{41} & T_{42} & 0 & 0 & T_{43} & 0 \\
0 & T_{51} & T_{52} & 0 & 0 & T_{53} \\
0 & 0 & T_{61} & T_{62} & 0 & T_{63} \\
0 & 0 & 0 & T_{71} & T_{72} & T_{73} \\
0 & T_{81} & 0 & 0 & T_{82} & T_{83} 
\end{bmatrix},  \\[0.8ex]
\label{Eq-BP8-MS}
M_S &= \begin{bmatrix}
0 & S_{12} & S_{13} & S_{14} & S_{15} & 0 \\
S_{12} & 0 & S_{23} & 0 & S_{25} & S_{26} \\
S_{13} & S_{23} & 0 & S_{34} & 0 & S_{36} \\
S_{14} & 0 & S_{34} & 0 & S_{45} & S_{46} \\
S_{15} & S_{25} & 0 & S_{45} & 0 & S_{56} \\
0 & S_{26} & S_{36} & S_{46} & S_{56} & 0
\end{bmatrix}.
\end{align}


For the edge assigments in $FB\mathcal{O}$, the matrices $M_T$ (resp. $M_{TS}$ and $M_S$) is the adjacency matrix between $T_{*}, T_{*}'$ (resp. $T_{*}, T_{*}', S_{*}, S_{*}'$, and $S_{*}, S_{*}'$) such that
\begin{align}
\label{Eq-FBP8-MT}
M_T &= \begin{bmatrix}
0 & T_{12} & 0 & T_{14} & 0 & 0 & 0 & 0 \\
T_{12} & 0 & T_{23} & 0 & 0 & 0 & 0 & 0 \\
0 & T_{23} & 0 & T_{34} & 0 & 0 & 0 & 0 \\
T_{14} & 0 & T_{34} & 0 & 0 & 0 & 0 & 0 \\
0 & 0 & 0 & 0 & 0 & T_{56} & 0 & T_{58} \\
0 & 0 & 0 & 0 & T_{56} & 0 & T_{67} & 0 \\
0 & 0 & 0 & 0 & 0 & T_{67} & 0 & T_{78} \\
0 & 0 & 0 & 0 & T_{58} & 0 & T_{78} & 0
\end{bmatrix}, \\[0.8ex]
\label{Eq-FBP8-MTS}
M_{TS} &= \begin{bmatrix}
T_{11} & T_{12} & T_{13} & 0 & 0 & 1 & 0 & 0 & 0 & 0 \\
T_{21} & 0 & T_{22} & T_{23} & 0 & 0 & 1 & 0 & 0 & 0 \\
T_{31} & 0 & 0 & T_{32} & T_{33} & 0 & 0 & 1 & 0 & 0 \\
T_{41} & T_{42} & 0 & 0 & T_{43} & 0 & 0 & 0 & 1 & 0 \\
0 & 0 & 1 & 0 & 0 & T_{51} & T_{52} & 0 & 0 & T_{53} \\
0 & 0 & 0 & 1 & 0 & 0 & T_{61} & T_{62} & 0 & T_{63} \\
0 & 0 & 0 & 0 & 1 & 0 & 0 & T_{71} & T_{72} & T_{73} \\
0 & 1 & 0 & 0 & 0 & T_{81} & 0 & 0 & T_{82} & T_{83}
\end{bmatrix}, \\[0.8ex]
\label{Eq-FBP8-MS}
M_{S} &= \begin{bmatrix}
0 & S_{12} & S_{13} & S_{14} & S_{15} & 0 & 0 & 0 & 0 & 0 \\
S_{12} & 0 & 0 & 0 & 0 & S_{22} & 0 & 0 & S_{25} & 0 \\
S_{13} & 0 & 0 & 0 & 0 & S_{32} & S_{33} & 0 & 0 & 0 \\
S_{14} & 0 & 0 & 0 & 0 & 0 & S_{43} & S_{44} & 0 & 0 \\
S_{15} & 0 & 0 & 0 & 0 & 0 & 0 & S_{54} & S_{55} & 0 \\
0 & S_{22} & S_{32} & 0 & 0 & 0 & 0 & 0 & 0 & S_{26} \\
0 & 0 & S_{33} & S_{43} & 0 & 0 & 0 & 0 & 0 & S_{36} \\
0 & 0 & 0 & S_{44} & S_{54} & 0 & 0 & 0 & 0 & S_{46} \\
0 & S_{25} & 0 & 0 & S_{55} & 0 & 0 & 0 & 0 & S_{56} \\
0 & 0 & 0 & 0 & 0 & S_{26} & S_{36} & S_{46} & S_{56} & 0
\end{bmatrix},
\end{align}
where the $S_{\ast}, S_{\ast}'$ are in the order of $S_1, S_2', ..., S_5', S_2, ..., S_5, S_6$.

\begin{prop} \label{Prop-Graph-c} Up to graph isomorphism, there are $75$ graphs of $c$-edge assignments corresponding to dihedral f-tilings derived from the $B\mathcal{O}$ and $29$ derived from the $FB\mathcal{O}$.
\end{prop}

\begin{proof} For a desired tiling obtained by an $c$-edge assignment in $B\mathcal{O}$, its underlying graph $\mathcal{G}(A)$ is represented by the adjacency matrix $A$ \eqref{Eq-Adj-Matrix-A} such that the parameters in $M_T$ \eqref{Eq-BP8-MT} and $M_S$ \eqref{Eq-BP8-MS} are equal to $1$. That is, 
\begin{align*}
&M_T = \begin{bmatrix}
0 & 1 & 0 & 1 & 1 & 0 & 0 & 0 \\
1 & 0 & 1 & 0 & 0 & 1 & 0 & 0 \\
0 & 1 & 0 & 1 & 0 & 0 & 1 & 0 \\
1 & 0 & 1 & 0 & 0 & 0 & 0 & 1 \\
1 & 0 & 0 & 0 & 0 & 1 & 0 & 1 \\
0 & 1 & 0 & 0 & 1 & 0 & 1 & 0 \\
0 & 0 & 1 & 0 & 0 & 1 & 0 & 1 \\
0 & 0 & 0 & 1 & 1 & 0 & 1 & 0
\end{bmatrix}, \quad
M_{TS} = \begin{bmatrix}
T_{11} & T_{12} & T_{13} & 0 & 0 & 0 \\
T_{21} & 0 & T_{22} & T_{23} & 0 & 0 \\
T_{31} & 0 & 0 & T_{32} & T_{33} & 0 \\
T_{41} & T_{42} & 0 & 0 & T_{43} & 0 \\
0 & T_{51} & T_{52} & 0 & 0 & T_{53} \\
0 & 0 & T_{61} & T_{62} & 0 & T_{63} \\
0 & 0 & 0 & T_{71} & T_{72} & T_{73} \\
0 & T_{81} & 0 & 0 & T_{82} & T_{83} 
\end{bmatrix}, 
\end{align*}

\begin{align*}
&M_S = \begin{bmatrix}
0 & 1 & 1 & 1 & 1 & 0 \\
1 & 0 & 1 & 0 & 1 & 1 \\
1 & 1 & 0 & 1 & 0 & 1 \\
1 & 0 & 1 & 0 & 1 & 1 \\
1 & 1 & 0 & 1 & 0 & 1 \\
0 & 1 & 1 & 1 & 1 & 0
\end{bmatrix}.
\end{align*}
By Lemma \ref{Lem-even-deg-iff-alt-sum}, it suffices to enforce even degree at each vertex, which will guarantee the folding conditions. The even degree assumption implies that $M_{TS}$ has even sum in each column. Hence the following holds for the graph of the desired tilings
\begin{align}\label{Eq-P8-c-assign-cond}
\begin{cases}
&T_{11} + T_{21} + T_{31} + T_{41} \equiv 0 \mod 2, \\ 
&T_{12} + T_{42} + T_{51} + T_{81} \equiv 0 \mod 2, \\ 
&T_{13} + T_{22} + T_{52} + T_{61} \equiv 0 \mod 2, \\ 
&T_{23} + T_{32} + T_{62} + T_{71} \equiv 0 \mod 2, \\ 
&T_{33} + T_{43} + T_{72} + T_{82}  \equiv 0 \mod 2, \\ 
&T_{53} + T_{63} + T_{73} + T_{83}  \equiv 0 \mod 2. 
\end{cases}
\end{align}

To determine the tilings, we conduct the {\em enumeration process} of their adjacency matrices satisfying the following conditions
\begin{enumerate}
\item even sum in each column in $M_{TS}$, 
\item every pair of graphs $\mathcal{G}(A_1), \mathcal{G}(A_2)$ are not graph-isomorphic.
\end{enumerate}


Up to graph-isomorphism, the process yields $76$ graphs,  where one of them is the $B\mathcal{O}$.

Similarly, for a desired tiling derived from an $c$-edge assignment derived from $FB\mathcal{O}$, its underlying graph $\mathcal{G}(A)$ is represented by the adjacency matrix $A$ \eqref{Eq-Adj-Matrix-A}, where the parameters in $M_T$ \eqref{Eq-FBP8-MT} and $M_S$ \eqref{Eq-FBP8-MS} are equal to $1$. That is,
\begin{align*}
M_{T} &= \begin{bmatrix}
0 & 1 & 0 & 1 & 0 & 0 & 0 & 0 \\
1 & 0 & 1 & 0 & 0 & 0 & 0 & 0 \\
0 & 1 & 0 & 1 & 0 & 0 & 0 & 0 \\
1 & 0 & 1 & 0 & 0 & 0 & 0 & 0 \\
0 & 0 & 0 & 0 & 0 & 1 & 0 & 1 \\
0 & 0 & 0 & 0 & 1 & 0 & 1 & 0 \\
0 & 0 & 0 & 0 & 0 & 1 & 0 & 1 \\
0 & 0 & 0 & 0 & 1 & 0 & 1 & 0
\end{bmatrix}, \\[0.8ex]
M_{TS} &= \begin{bmatrix}
T_{11} & T_{12} & T_{13} & 0 & 0 & 1 & 0 & 0 & 0 & 0 \\
T_{21} & 0 & T_{22} & T_{23} & 0 & 0 & 1 & 0 & 0 & 0 \\
T_{31} & 0 & 0 & T_{32} & T_{33} & 0 & 0 & 1 & 0 & 0 \\
T_{41} & T_{42} & 0 & 0 & T_{43} & 0 & 0 & 0 & 1 & 0 \\
0 & 0 & 1 & 0 & 0 & T_{51} & T_{52} & 0 & 0 & T_{53} \\
0 & 0 & 0 & 1 & 0 & 0 & T_{61} & T_{62} & 0 & T_{63} \\
0 & 0 & 0 & 0 & 1 & 0 & 0 & T_{71} & T_{72} & T_{73} \\
0 & 1 & 0 & 0 & 0 & T_{81} & 0 & 0 & T_{82} & T_{83}
\end{bmatrix}, \\[0.8ex]
M_{S} &=
\begin{bmatrix}
0 & 1 & 1 & 1 & 1 & 0 & 0 & 0 & 0 & 0 \\
1 & 0 & 0 & 0 & 0 & 1 & 0 & 0 & 1 & 0 \\
1 & 0 & 0 & 0 & 0 & 1 & 1 & 0 & 0 & 0 \\
1 & 0 & 0 & 0 & 0 & 0 & 1 & 1 & 0 & 0 \\
1 & 0 & 0 & 0 & 0 & 0 & 0 & 1 & 1 & 0 \\
0 & 1 & 1 & 0 & 0 & 0 & 0 & 0 & 0 & 1 \\
0 & 0 & 1 & 1 & 0 & 0 & 0 & 0 & 0 & 1 \\
0 & 0 & 0 & 1 & 1 & 0 & 0 & 0 & 0 & 1 \\
0 & 1 & 0 & 0 & 1 & 0 & 0 & 0 & 0 & 1 \\
0 & 0 & 0 & 0 & 0 & 1 & 1 & 1 & 1 & 0
\end{bmatrix}.
\end{align*}
By Lemma \ref{Lem-even-deg-iff-alt-sum}, it suffices to enforce even degree at each vertex. The even degree assumption implies that $M_{TS}$ has even sums in the first and the last columns, and odd sums in the other columns. Up to graph-isomorphism, the enumeration process yields $30$ graphs, where one of them is the $FB\mathcal{O}$.
\end{proof}

\begin{prop} \label{Prop-Graph-b} Up to graph isomorphism, there are $12$ graphs of $b$-edge assignments corresponding to dihedral f-tilings derived from the $B\mathcal{O}$ and no such graphs derived from the $FB\mathcal{O}$.
\end{prop}

\begin{proof} Following the same argument in Proposition \ref{Prop-Graph-c}, the underlying graph $\mathcal{G}(A)$ of a $b$-edge assignment is represented by the adjacency matrix $A$ \eqref{Eq-Adj-Matrix-A} such that the parameters in $M_T$ \eqref{Eq-BP8-MT} and $M_{TS}$ \eqref{Eq-BP8-MTS} are all $1$'s. 

By Lemma \ref{Lem-even-deg-iff-alt-sum}, it suffices to enforce even degree at each vertex. The even degree assumption implies that $M_S$ \eqref{Eq-BP8-MS} has even sum in each row. Then the same enumeration process yields $14$ graphs up to graph isomorphism, one gives the triangular subdivision of the cube, one gives the $B\mathcal{O}$ and the remaining graphs give the $12$ dihedral f-tilings.

The argument for the $b$-edge assignments derived from the $FB\mathcal{O}$ is analogous by using $M_T$ \eqref{Eq-FBP8-MT} and $M_S$ \eqref{Eq-FBP8-MTS} with the parameters equal to $1$. 
\end{proof}

\begin{prop} \label{Prop-Graph-a} Up to graph isomorphism, there are $5$ graphs of $a$-edge assignments corresponding to dihedral f-tilings derived from the $B\mathcal{O}$ and 2 derived from the $FB\mathcal{O}$.
\end{prop}

\begin{proof} For $a$-edge assignments derived from the $B\mathcal{O}$, the argument is analogous by using $M_{TS}$ \eqref{Eq-BP8-MTS} and $M_{S}$ \eqref{Eq-BP8-MS} with parameters equal to constant $1$'s. The assumption of even degree implies that $M_{S}$ has odd sum in each row. Similar argument for $a$-edge assignments derived from the $FB\mathcal{O}$ applies to $M_{TS}$ \eqref{Eq-FBP8-MTS} and $M_S$ \eqref{Eq-FBP8-MS}.
\end{proof}

\section{Appendix} 
\label{Sec:Appendix}

In this section we present the plane representations of the dihedral f-tilings induced by the M\"obius triangle $(2,3,4)$ and the corresponding  geometric and combinatorial structure.
We use the notation $B_j \mathcal{O}_k$ and $FB_j \mathcal{O}_k$ for the dihedral f-tiling that corresponds to the $k$th element of the $j$-edge assignments ($j = a, b, c$) derived from the $B\mathcal{O}$ and $FB\mathcal{O}$ monohedral structures, respectively. 
Regarding the f-tilings with prototiles being the M\"obius triangle and the (\emph{i}) kite,  (\emph{ii}) isosceles triangle $\bar{a}c^2$,  (\emph{iii}) isosceles triangle $\bar{b}c^2$,  consider (\emph{i}) Figures \ref{Fig:BP8-c-I} and \ref{Fig:X-Y-octants-ne2e}, (\emph{ii}) Figure \ref{Fig:BP8-b}, (\emph{iii}) Figures \ref{Fig:BP8-a} and \ref{Fig:FBP8-a}, respectively. 
By pressing each image of these figures,  the corresponding 3D model can be viewed.  These 3D representations are also available at  {\small \url{https://www.geogebra.org/m/zfnap4pe}}.

\begin{figure}[h!] 
\centering
\begin{tikzpicture}

\tikzmath{
\s=2.4;
\x=0.25;
\l=3;
\th=360/4;
\rr=0.15*\x;
}

\begin{scope}[]

\begin{scope}[]

\node[align=center, text width=1.6cm,,,text opacity = 0] (c1) {\href{https://www.geogebra.org/m/bzdkx59v}{a b c d e f g h i j k l m n}};
\def\EvenXArray {{1,-1,-1,1,3,-3,-3,3}}
\def\OddXArray {{0,-2,0,2,0,-2,0,2}}

\def\EvenYArray {{1,1,-1,-1,3,3,-3,-3}}
\def\OddYArray {{2,0,-2,0,2,0,-2,0}}

\pgfmathtruncatemacro{\dimX}{dim{\EvenXArray}-1} \dimX

\foreach \a in {0,...,\dimX} {

\draw[red]
	(\EvenXArray[\a]*\x, \EvenYArray[\a]*\x) -- (\OddXArray[\a]*\x, \OddYArray[\a]*\x)
;

}
\node (s1c) {} node[below=9mm]{\tiny $B_c \mathcal{O}_1$};
\end{scope}

\begin{scope}[xshift=\s cm] 

\node[align=center, text width=1.6cm,,,text opacity = 0] (c1) {\href{https://www.geogebra.org/m/zewgdkat}{a b c d e f g h i j k l m n}};
\def\EvenXArray {{1,-1,-1,1,3,-3,-3,3}}
\def\OddXArray {{0,-2,0,2,2,0,-2,0}}

\def\EvenYArray {{1,1,-1,-1,3,3,-3,-3}}
\def\OddYArray {{2,0,-2,0,0,2,0,-2}}

\pgfmathtruncatemacro{\dimX}{dim{\EvenXArray}-1} \dimX

\foreach \a in {0,...,\dimX} {

\draw[red]
	(\EvenXArray[\a]*\x, \EvenYArray[\a]*\x) -- (\OddXArray[\a]*\x, \OddYArray[\a]*\x)
;

}
\node (s1c) {} node[below=9mm]{\tiny $B_c \mathcal{O}_2$};
\end{scope}

\begin{scope}[xshift=2*\s cm] 

\node[align=center, text width=1.6cm,,,text opacity = 0] (c1) {\href{https://www.geogebra.org/m/hdhdxtbv}{a b c d e f g h i j k l m n}};
\def\EvenXArray {{1,-1,-1,1,3,-3,-3,3}}
\def\OddXArray {{0,-2,0,0,4,0,-2,4}}

\def\EvenYArray {{1,1,-1,-1,3,3,-3,-3}}
\def\OddYArray {{2,0,-2,-2,4,2,0,-4}}

\pgfmathtruncatemacro{\dimX}{dim{\EvenXArray}-1} \dimX

\foreach \a in {0,...,\dimX} {

\draw[red]
	(\EvenXArray[\a]*\x, \EvenYArray[\a]*\x) -- (\OddXArray[\a]*\x, \OddYArray[\a]*\x)
;

}
\node (s1c) {} node[below=9mm]{\tiny $B_c \mathcal{O}_3$};
\end{scope}

\begin{scope}[xshift=3*\s cm] 

\node[align=center, text width=1.6cm,,,text opacity = 0] (c1) {\href{https://www.geogebra.org/m/yrskprv9}{a b c d e f g h i j k l m n}};
\def\EvenXArray {{1,-1,-1,1,3,-3,-3,3}}
\def\OddXArray {{0,-2,0,0,0,-4,-2,4}}

\def\EvenYArray {{1,1,-1,-1,3,3,-3,-3}}
\def\OddYArray {{2,0,-2,-2,2,4,0,-4}}

\pgfmathtruncatemacro{\dimX}{dim{\EvenXArray}-1} \dimX

\foreach \a in {0,...,\dimX} {

\draw[red]
	(\EvenXArray[\a]*\x, \EvenYArray[\a]*\x) -- (\OddXArray[\a]*\x, \OddYArray[\a]*\x)
;

}
\node (s1c) {} node[below=9mm]{\tiny $B_c \mathcal{O}_4$};
\end{scope}

\begin{scope}[xshift=4*\s cm] 

\node[align=center, text width=1.6cm,,,text opacity = 0] (c1) {\href{https://www.geogebra.org/m/jkvxd2je}{a b c d e f g h i j k l m n}};

\def\EvenXArray {{1,-1,-1,1,3,-3,-3,3}}
\def\OddXArray {{0,-2,0,0,0,-2,-4,4}}

\def\EvenYArray {{1,1,-1,-1,3,3,-3,-3}}
\def\OddYArray {{2,0,-2,-2,2,0,-4,-4}}

\pgfmathtruncatemacro{\dimX}{dim{\EvenXArray}-1} \dimX

\foreach \a in {0,...,\dimX} {

\draw[red]
	(\EvenXArray[\a]*\x, \EvenYArray[\a]*\x) -- (\OddXArray[\a]*\x, \OddYArray[\a]*\x)
;

}
\node (s1c) {} node[below=9mm]{\tiny $B_c \mathcal{O}_5$};
\end{scope}

\end{scope}

\begin{scope}[yshift=-1*\s cm] 

\begin{scope}[] 

\node[align=center, text width=1.6cm,,,text opacity = 0] (c1) {\href{https://www.geogebra.org/m/ukms5kbj}{a b c d e f g h i j k l m n}};

\def\EvenXArray {{1,-1,-1,1,3,-3,-3,3}}
\def\OddXArray {{0,-2,0,0,0,-2,0,0}}

\def\EvenYArray {{1,1,-1,-1,3,3,-3,-3}}
\def\OddYArray {{2,0,-2,-2,2,0,-2,-2}}

\pgfmathtruncatemacro{\dimX}{dim{\EvenXArray}-1} \dimX

\foreach \a in {0,...,\dimX} {

\draw[red]
	(\EvenXArray[\a]*\x, \EvenYArray[\a]*\x) -- (\OddXArray[\a]*\x, \OddYArray[\a]*\x)
;

}
\node (s1c) {} node[below=9mm]{\tiny $B_c \mathcal{O}_6$};
\end{scope}

\begin{scope}[xshift=\s cm] 

\node[align=center, text width=1.6cm,,,text opacity = 0] (c1) {\href{https://www.geogebra.org/m/pehvbyvx}{a b c d e f g h i j k l m n}};

\def\EvenXArray {{1,-1,-1,1,3,-3,-3,3}}
\def\OddXArray {{0,-2,0,0,2,0,-2,2}}

\def\EvenYArray {{1,1,-1,-1,3,3,-3,-3}}
\def\OddYArray {{2,0,-2,-2,0,2,0,0}}

\pgfmathtruncatemacro{\dimX}{dim{\EvenXArray}-1} \dimX

\foreach \a in {0,...,\dimX} {

\draw[red]
	(\EvenXArray[\a]*\x, \EvenYArray[\a]*\x) -- (\OddXArray[\a]*\x, \OddYArray[\a]*\x)
;

}
\node (s1c) {} node[below=9mm]{\tiny $B_c \mathcal{O}_7$};
\end{scope}

\begin{scope}[xshift=2*\s cm] 

\node[align=center, text width=1.6cm,,,text opacity = 0] (c1) {\href{https://www.geogebra.org/m/sqmc2u4e}{a b c d e f g h i j k l m n}};

\def\EvenXArray {{1,-1,-1,1,3,-3,-3,3}}
\def\OddXArray {{0,-2,-2,0,4,0,0,4}}

\def\EvenYArray {{1,1,-1,-1,3,3,-3,-3}}
\def\OddYArray {{2,0,0,-2,4,2,-2,-4}}

\pgfmathtruncatemacro{\dimX}{dim{\EvenXArray}-1} \dimX

\foreach \a in {0,...,\dimX} {

\draw[red]
	(\EvenXArray[\a]*\x, \EvenYArray[\a]*\x) -- (\OddXArray[\a]*\x, \OddYArray[\a]*\x)
;

}
\node (s1c) {} node[below=9mm]{\tiny $B_c \mathcal{O}_8$};
\end{scope}

\begin{scope}[xshift=3*\s cm] 

\node[align=center, text width=1.6cm,,,text opacity = 0] (c1) {\href{https://www.geogebra.org/m/j7nm9pnh}{a b c d e f g h i j k l m n}};

\def\EvenXArray {{1,-1,-1,1,3,-3,-3,3}}
\def\OddXArray {{0,-2,-2,0,0,-4,-4,0}}

\def\EvenYArray {{1,1,-1,-1,3,3,-3,-3}}
\def\OddYArray {{2,0,0,-2,2,4,-4,-2}}

\pgfmathtruncatemacro{\dimX}{dim{\EvenXArray}-1} \dimX

\foreach \a in {0,...,\dimX} {

\draw[red]
	(\EvenXArray[\a]*\x, \EvenYArray[\a]*\x) -- (\OddXArray[\a]*\x, \OddYArray[\a]*\x)
;

}
\node (s1c) {} node[below=9mm]{\tiny $B_c \mathcal{O}_9$};
\end{scope}

\begin{scope}[xshift=4*\s cm] 

\node[align=center, text width=1.6cm,,,text opacity = 0] (c1) {\href{https://www.geogebra.org/m/ffq3sbea}{a b c d e f g h i j k l m n}};

\def\EvenXArray {{1,-1,-1,1,3,-3,-3,3}}
\def\OddXArray {{0,-2,-2,0,0,-2,-2,0}}

\def\EvenYArray {{1,1,-1,-1,3,3,-3,-3}}
\def\OddYArray {{2,0,0,-2,2,0,0,-2}}

\pgfmathtruncatemacro{\dimX}{dim{\EvenXArray}-1} \dimX

\foreach \a in {0,...,\dimX} {

\draw[red]
	(\EvenXArray[\a]*\x, \EvenYArray[\a]*\x) -- (\OddXArray[\a]*\x, \OddYArray[\a]*\x)
;

}
\node (s1c) {} node[below=9mm]{\tiny $B_c \mathcal{O}_{10}$};
\end{scope}

\end{scope}

\begin{scope}[yshift=-2*\s cm]

\begin{scope}[] 

\node[align=center, text width=1.6cm,,,text opacity = 0] (c1) {\href{https://www.geogebra.org/m/g2nwfwse}{a b c d e f g h i j k l m n}};

\def\EvenXArray {{1,-1,-1,1,3,-3,-3,3}}
\def\OddXArray {{0,-2,-2,0,2,0,0,2}}

\def\EvenYArray {{1,1,-1,-1,3,3,-3,-3}}
\def\OddYArray {{2,0,0,-2,0,2,-2,0}}

\pgfmathtruncatemacro{\dimX}{dim{\EvenXArray}-1} \dimX

\foreach \a in {0,...,\dimX} {

\draw[red]
	(\EvenXArray[\a]*\x, \EvenYArray[\a]*\x) -- (\OddXArray[\a]*\x, \OddYArray[\a]*\x)
;

}
\node (s1c) {} node[below=9mm]{\tiny $B_c \mathcal{O}_{11}$};
\end{scope}

\begin{scope}[xshift=\s cm] 

\node[align=center, text width=1.6cm,,,text opacity = 0] (c1) {\href{https://www.geogebra.org/m/jw8mh79u}{a b c d e f g h i j k l m n}};

\def\EvenXArray {{1,-1,-1,1,3,-3,-3,3}}
\def\OddXArray {{0,0,0,0,0,0,0,0}}

\def\EvenYArray {{1,1,-1,-1,3,3,-3,-3}}
\def\OddYArray {{2,2,-2,-2,2,2,-2,-2}}

\pgfmathtruncatemacro{\dimX}{dim{\EvenXArray}-1} \dimX

\foreach \a in {0,...,\dimX} {

\draw[red]
	(\EvenXArray[\a]*\x, \EvenYArray[\a]*\x) -- (\OddXArray[\a]*\x, \OddYArray[\a]*\x)
;
}
\node (s1c) {} node[below=9mm]{\tiny $B_c \mathcal{O}_{12}$};
\end{scope}

\begin{scope}[xshift=2*\s cm] 

\node[align=center, text width=1.6cm,,,text opacity = 0] (c1) {\href{https://www.geogebra.org/m/bgmctshu}{a b c d e f g h i j k l m n}};

\def\EvenXArray {{1,-1,-1,1,3,-3,-3,3}}
\def\OddXArray {{0,0,0,0,2,-2,-2,2}}

\def\EvenYArray {{1,1,-1,-1,3,3,-3,-3}}
\def\OddYArray {{2,2,-2,-2,0,0,0,0}}

\pgfmathtruncatemacro{\dimX}{dim{\EvenXArray}-1} \dimX

\foreach \a in {0,...,\dimX} {

\draw[red]
	(\EvenXArray[\a]*\x, \EvenYArray[\a]*\x) -- (\OddXArray[\a]*\x, \OddYArray[\a]*\x)
;

}
\node (s1c) {} node[below=9mm]{\tiny $B_c \mathcal{O}_{13}$};
\end{scope}

\begin{scope}[xshift=3*\s cm] 

\node[align=center, text width=1.6cm,,,text opacity = 0] (c1) {\href{https://www.geogebra.org/m/seykee5y}{a b c d e f g h i j k l m n}};

\def\EvenXArray {{1,-1,-1,1,3,-3,-3,3,3,3}}
\def\OddXArray {{0,-2,0,2,4,0,-2,0,2,4}}

\def\EvenYArray {{1,1,-1,-1,3,3,-3,-3,-3,-3}}
\def\OddYArray {{2,0,-2,0,4,2,0,-2,0,-4}}

\pgfmathtruncatemacro{\dimX}{dim{\EvenXArray}-1} \dimX

\foreach \a in {0,...,\dimX} {

\draw[red]
	(\EvenXArray[\a]*\x, \EvenYArray[\a]*\x) -- (\OddXArray[\a]*\x, \OddYArray[\a]*\x)
;

}
\node (s1c) {} node[below=9mm]{\tiny $B_c \mathcal{O}_{14}$};
\end{scope}

\begin{scope}[xshift=4*\s cm] 

\node[align=center, text width=1.6cm,,,text opacity = 0] (c1) {\href{https://www.geogebra.org/m/twyckhte}{a b c d e f g h i j k l m n}};

\def\EvenXArray {{1,-1,-1,1,3,-3,-3,3,3,3}}
\def\OddXArray {{0,-2,0,2,0,-4,-2,0,2,4}}

\def\EvenYArray {{1,1,-1,-1,3,3,-3,-3,-3,-3}}
\def\OddYArray {{2,0,-2,0,2,4,0,-2,0,-4}}

\pgfmathtruncatemacro{\dimX}{dim{\EvenXArray}-1} \dimX

\foreach \a in {0,...,\dimX} {

\draw[red]
	(\EvenXArray[\a]*\x, \EvenYArray[\a]*\x) -- (\OddXArray[\a]*\x, \OddYArray[\a]*\x)
;

}
\node (s1c) {} node[below=9mm]{\tiny $B_c \mathcal{O}_{15}$};
\end{scope}

\end{scope}

\begin{scope}[yshift=-3*\s cm] 

\begin{scope}[] 

\node[align=center, text width=1.6cm,,,text opacity = 0] (c1) {\href{https://www.geogebra.org/m/mt8vdqk4}{a b c d e f g h i j k l m n}};

\def\EvenXArray {{1,-1,-1,1,3,-3,-3,3,3,3}}
\def\OddXArray {{0,-2,0,2,0,-2,-4,0,2,4}}

\def\EvenYArray {{1,1,-1,-1,3,3,-3,-3,-3,-3}}
\def\OddYArray {{2,0,-2,0,2,0,-4,-2,0,-4}}

\pgfmathtruncatemacro{\dimX}{dim{\EvenXArray}-1} \dimX

\foreach \a in {0,...,\dimX} {

\draw[red]
	(\EvenXArray[\a]*\x, \EvenYArray[\a]*\x) -- (\OddXArray[\a]*\x, \OddYArray[\a]*\x)
;

}
\node (s1c) {} node[below=9mm]{\tiny $B_c \mathcal{O}_{16}$};
\end{scope}

\begin{scope}[xshift=\s cm] 

\node[align=center, text width=1.6cm,,,text opacity = 0] (c1) {\href{https://www.geogebra.org/m/wcg7c7mh}{a b c d e f g h i j k l m n}};

\def\EvenXArray {{1,-1,-1,1,3,-3,-3,3,3,3}}
\def\OddXArray {{0,-2,-2,2,4,0,0,0,2,4}}

\def\EvenYArray {{1,1,-1,-1,3,3,-3,-3,-3,-3}}
\def\OddYArray {{2,0,0,0,4,2,-2,-2,0,-4}}

\pgfmathtruncatemacro{\dimX}{dim{\EvenXArray}-1} \dimX

\foreach \a in {0,...,\dimX} {

\draw[red]
	(\EvenXArray[\a]*\x, \EvenYArray[\a]*\x) -- (\OddXArray[\a]*\x, \OddYArray[\a]*\x)
;
}
\node (s1c) {} node[below=9mm]{\tiny $B_c \mathcal{O}_{17}$};
\end{scope}

\begin{scope}[xshift=2*\s cm] 

\node[align=center, text width=1.6cm,,,text opacity = 0] (c1) {\href{https://www.geogebra.org/m/qusdarpu}{a b c d e f g h i j k l m n}};

\def\EvenXArray {{1,-1,-1,1,3,-3,-3,3,3,3}}
\def\OddXArray {{0,-2,-2,2,0,-4,0,0,2,4}}

\def\EvenYArray {{1,1,-1,-1,3,3,-3,-3,-3,-3}}
\def\OddYArray {{2,0,0,0,2,4,-2,-2,0,-4}}

\pgfmathtruncatemacro{\dimX}{dim{\EvenXArray}-1} \dimX

\foreach \a in {0,...,\dimX} {

\draw[red]
	(\EvenXArray[\a]*\x, \EvenYArray[\a]*\x) -- (\OddXArray[\a]*\x, \OddYArray[\a]*\x)
;

}
\node (s1c) {} node[below=9mm]{\tiny $B_c \mathcal{O}_{18}$};
\end{scope}

\begin{scope}[xshift=3*\s cm] 

\node[align=center, text width=1.6cm,,,text opacity = 0] (c1) {\href{https://www.geogebra.org/m/vjymvp4f}{a b c d e f g h i j k l m n}};

\def\EvenXArray {{1,-1,-1,1,3,-3,-3,3,3,3}}
\def\OddXArray {{0,-2,-2,0,2,0,-4,0,2,4}}

\def\EvenYArray {{1,1,-1,-1,3,3,-3,-3,-3,-3}}
\def\OddYArray {{2,0,0,-2,0,2,-4,-2,0,-4}}

\pgfmathtruncatemacro{\dimX}{dim{\EvenXArray}-1} \dimX

\foreach \a in {0,...,\dimX} {

\draw[red]
	(\EvenXArray[\a]*\x, \EvenYArray[\a]*\x) -- (\OddXArray[\a]*\x, \OddYArray[\a]*\x)
;

}
\node (s1c) {} node[below=9mm]{\tiny $B_c \mathcal{O}_{19}$};
\end{scope}

\begin{scope}[xshift=4*\s cm] 

\node[align=center, text width=1.6cm,,,text opacity = 0] (c1) {\href{https://www.geogebra.org/m/j6smncne}{a b c d e f g h i j k l m n}};

\def\EvenXArray {{1,-1,-1,1,3,-3,-3,3,3,3}}
\def\OddXArray {{0,0,0,2,4,-4,-4,0,2,4}}

\def\EvenYArray {{1,1,-1,-1,3,3,-3,-3,-3,-3}}
\def\OddYArray {{2,2,-2,0,4,4,-4,-2,0,-4}}

\pgfmathtruncatemacro{\dimX}{dim{\EvenXArray}-1} \dimX

\foreach \a in {0,...,\dimX} {

\draw[red]
	(\EvenXArray[\a]*\x, \EvenYArray[\a]*\x) -- (\OddXArray[\a]*\x, \OddYArray[\a]*\x)
;

}
\node (s1c) {} node[below=9mm]{\tiny $B_c \mathcal{O}_{20}$};
\end{scope}

\end{scope}

\begin{scope}[yshift=-4*\s cm]

\begin{scope}[] 

\node[align=center, text width=1.6cm,,,text opacity = 0] (c1) {\href{https://www.geogebra.org/m/r2twvw7d}{a b c d e f g h i j k l m n}};

\def\EvenXArray {{1,-1,-1,1,3,-3,-3,3,3,3}}
\def\OddXArray {{0,0,0,2,4,-2,-2,0,2,4}}

\def\EvenYArray {{1,1,-1,-1,3,3,-3,-3,-3,-3}}
\def\OddYArray {{2,2,-2,0,4,0,0,-2,0,-4}}

\pgfmathtruncatemacro{\dimX}{dim{\EvenXArray}-1} \dimX

\foreach \a in {0,...,\dimX} {

\draw[red]
	(\EvenXArray[\a]*\x, \EvenYArray[\a]*\x) -- (\OddXArray[\a]*\x, \OddYArray[\a]*\x)
;

}
\node (s1c) {} node[below=9mm]{\tiny $B_c \mathcal{O}_{21}$};
\end{scope}

\begin{scope}[xshift=\s cm] 

\node[align=center, text width=1.6cm,,,text opacity = 0] (c1) {\href{https://www.geogebra.org/m/vh6mmby5}{a b c d e f g h i j k l m n}};

\def\EvenXArray {{1,-1,-1,1,3,-3,-3,3,3,3}}
\def\OddXArray {{0,0,0,2,0,0,-4,0,2,4}}

\def\EvenYArray {{1,1,-1,-1,3,3,-3,-3,-3,-3}}
\def\OddYArray {{2,2,-2,0,2,2,-4,-2,0,-4}}

\pgfmathtruncatemacro{\dimX}{dim{\EvenXArray}-1} \dimX

\foreach \a in {0,...,\dimX} {

\draw[red]
	(\EvenXArray[\a]*\x, \EvenYArray[\a]*\x) -- (\OddXArray[\a]*\x, \OddYArray[\a]*\x)
;

}
\node (s1c) {} node[below=9mm]{\tiny $B_c \mathcal{O}_{22}$};
\end{scope}

\begin{scope}[xshift=2*\s cm] 

\node[align=center, text width=1.6cm,,,text opacity = 0] (c1) {\href{https://www.geogebra.org/m/vc4fhfsk}{a b c d e f g h i j k l m n}};

\def\EvenXArray {{1,-1,-1,1,3,-3,-3,3,3,3}}
\def\OddXArray {{0,0,0,0,2,-4,0,0,2,4}}

\def\EvenYArray {{1,1,-1,-1,3,3,-3,-3,-3,-3}}
\def\OddYArray {{2,2,-2,-2,0,4,-2,-2,0,-4}}

\pgfmathtruncatemacro{\dimX}{dim{\EvenXArray}-1} \dimX

\foreach \a in {0,...,\dimX} {

\draw[red]
	(\EvenXArray[\a]*\x, \EvenYArray[\a]*\x) -- (\OddXArray[\a]*\x, \OddYArray[\a]*\x)
;

}
\node (s1c) {} node[below=9mm]{\tiny $B_c \mathcal{O}_{23}$};
\end{scope}

\begin{scope}[xshift=3*\s cm] 

\node[align=center, text width=1.6cm,,,text opacity = 0] (c1) {\href{https://www.geogebra.org/m/tm4tgdq2}{a b c d e f g h i j k l m n}};

\def\EvenXArray {{1,-1,-1,1,3,-3,-3,3,3,3}}
\def\OddXArray {{2,-2,0,2,2,-4,-2,0,2,4}}

\def\EvenYArray {{1,1,-1,-1,3,3,-3,-3,-3,-3}}
\def\OddYArray {{0,0,-2,0,0,4,0,-2,0,-4}}

\pgfmathtruncatemacro{\dimX}{dim{\EvenXArray}-1} \dimX

\foreach \a in {0,...,\dimX} {

\draw[red]
	(\EvenXArray[\a]*\x, \EvenYArray[\a]*\x) -- (\OddXArray[\a]*\x, \OddYArray[\a]*\x)
;

}
\node (s1c) {} node[below=9mm]{\tiny $B_c \mathcal{O}_{24}$};
\end{scope}

\begin{scope}[xshift=4*\s cm] 

\node[align=center, text width=1.6cm,,,text opacity = 0] (c1) {\href{https://www.geogebra.org/m/kud42wrm}{a b c d e f g h i j k l m n}};

\def\EvenXArray {{1,-1,-1,1,3,-3,-3,-3,-3,3,3,3}}
\def\OddXArray {{0,-2,0,0,2,0,-2,0,-4,0,2,4}}

\def\EvenYArray {{1,1,-1,-1,3,3,-3,-3,-3,-3,-3,-3}}
\def\OddYArray {{2,0,-2,-2,0,2,0,-2,-4,-2,0,-4}}

\pgfmathtruncatemacro{\dimX}{dim{\EvenXArray}-1} \dimX

\foreach \a in {0,...,\dimX} {

\draw[red]
	(\EvenXArray[\a]*\x, \EvenYArray[\a]*\x) -- (\OddXArray[\a]*\x, \OddYArray[\a]*\x)
;

}
\node (s1c) {} node[below=9mm]{\tiny $B_c \mathcal{O}_{25}$};
\end{scope}

\end{scope}

\begin{scope}[yshift=-5*\s cm] 

\begin{scope}[] 

\node[align=center, text width=1.6cm,,,text opacity = 0] (c1) {\href{https://www.geogebra.org/m/d7gxsn9t}{a b c d e f g h i j k l m n}};

\def\EvenXArray {{1,-1,-1,1,3,-3,-3,-3,-3,3,3,3}}
\def\OddXArray {{0,-2,-2,2,0,-2,-2,0,-4,0,2,4}}

\def\EvenYArray {{1,1,-1,-1,3,3,-3,-3,-3,-3,-3,-3}}
\def\OddYArray {{2,0,0,0,2,0,0,-2,-4,-2,0,-4}}

\pgfmathtruncatemacro{\dimX}{dim{\EvenXArray}-1} \dimX

\foreach \a in {0,...,\dimX} {

\draw[red]
	(\EvenXArray[\a]*\x, \EvenYArray[\a]*\x) -- (\OddXArray[\a]*\x, \OddYArray[\a]*\x)
;

}
\node (s1c) {} node[below=9mm]{\tiny $B_c \mathcal{O}_{26}$};
\end{scope}

\begin{scope}[xshift=\s cm] 

\node[align=center, text width=1.6cm,,,text opacity = 0] (c1) {\href{https://www.geogebra.org/m/hsbzca2d}{a b c d e f g h i j k l m n}};

\def\EvenXArray {{1,-1,-1,1,3,-3,-3,-3,-3,3,3,3}}
\def\OddXArray {{0,-2,0,0,2,0,-2,0,-4,0,2,4}}

\def\EvenYArray {{1,1,-1,-1,3,3,-3,-3,-3,-3,-3,-3}}
\def\OddYArray {{2,0,0,0,0,2,0,-2,-4,-2,0,-4}}

\pgfmathtruncatemacro{\dimX}{dim{\EvenXArray}-1} \dimX

\foreach \a in {0,...,\dimX} {

\draw[red]
	(\EvenXArray[\a]*\x, \EvenYArray[\a]*\x) -- (\OddXArray[\a]*\x, \OddYArray[\a]*\x)
;

}
\node (s1c) {} node[below=9mm]{\tiny $B_c \mathcal{O}_{27}$};
\end{scope}

\begin{scope}[xshift=2*\s cm] 

\node[align=center, text width=1.6cm,,,text opacity = 0] (c1) {\href{https://www.geogebra.org/m/cdnqh346}{a b c d e f g h i j k l m n}};

\def\EvenXArray {{1,-1,-1,1,3,-3,-3,-3,-3,3,3,3}}
\def\OddXArray {{0,0,0,0,2,-2,-2,0,-4,0,2,4}}

\def\EvenYArray {{1,1,-1,-1,3,3,-3,-3,-3,-3,-3,-3}}
\def\OddYArray {{2,2,-2,-2,0,0,0,-2,-4,-2,0,-4}}

\pgfmathtruncatemacro{\dimX}{dim{\EvenXArray}-1} \dimX

\foreach \a in {0,...,\dimX} {

\draw[red]
	(\EvenXArray[\a]*\x, \EvenYArray[\a]*\x) -- (\OddXArray[\a]*\x, \OddYArray[\a]*\x)
;

}
\node (s1c) {} node[below=9mm]{\tiny $B_c \mathcal{O}_{28}$};
\end{scope}

\begin{scope}[xshift=3*\s cm] 

\node[align=center, text width=1.6cm,,,text opacity = 0] (c1) {\href{https://www.geogebra.org/m/nuxksxrk}{a b c d e f g h i j k l m n}};

\def\EvenXArray {{1,-1,-1,1,3,-3,-3,-3,-3,3,3,3}}
\def\OddXArray {{0,0,-2,2,4,-4,-2,0,-4,0,2,4}}

\def\EvenYArray {{1,1,-1,-1,3,3,-3,-3,-3,-3,-3,-3}}
\def\OddYArray {{2,2,0,0,4,4,0,-2,-4,-2,0,-4}}

\pgfmathtruncatemacro{\dimX}{dim{\EvenXArray}-1} \dimX

\foreach \a in {0,...,\dimX} {

\draw[red]
	(\EvenXArray[\a]*\x, \EvenYArray[\a]*\x) -- (\OddXArray[\a]*\x, \OddYArray[\a]*\x)
;

}
\node (s1c) {} node[below=9mm]{\tiny $B_c \mathcal{O}_{29}$};
\end{scope}


\begin{scope}[xshift=4*\s cm] 

\node[align=center, text width=1.6cm,,,text opacity = 0] (c1) {\href{https://www.geogebra.org/m/j7patrv2}{a b c d e f g h i j k l m n}};

\def\EvenXArray {{1,-1,-1,1,3,-3,-3,-3,-3,3,3,3}}
\def\OddXArray {{0,0,-2,2,0,0,-2,0,-4,0,2,4}}

\def\EvenYArray {{1,1,-1,-1,3,3,-3,-3,-3,-3,-3,-3}}
\def\OddYArray {{2,2,0,0,2,2,0,-2,-4,-2,0,-4}}

\pgfmathtruncatemacro{\dimX}{dim{\EvenXArray}-1} \dimX

\foreach \a in {0,...,\dimX} {

\draw[red]
	(\EvenXArray[\a]*\x, \EvenYArray[\a]*\x) -- (\OddXArray[\a]*\x, \OddYArray[\a]*\x)
;

}
\node (s1c) {} node[below=9mm]{\tiny $B_c \mathcal{O}_{30}$};
\end{scope}

\end{scope}

\begin{scope}[yshift=-6*\s cm] 

\begin{scope}[] 
 
 \node[align=center, text width=1.6cm,,,text opacity = 0] (c1) {\href{https://www.geogebra.org/m/u3ga5ytk}{a b c d e f g h i j k l m n}};
 
\def\EvenXArray {{1,-1,-1,1,3,-3,-3,-3,-3,3,3,3}}
\def\OddXArray {{0,0,0,0,2,-2,-2,0,-4,0,2,4}}

\def\EvenYArray {{1,1,-1,-1,3,3,-3,-3,-3,-3,-3,-3}}
\def\OddYArray {{2,2,0,0,0,0,0,-2,-4,-2,0,-4}}

\pgfmathtruncatemacro{\dimX}{dim{\EvenXArray}-1} \dimX

\foreach \a in {0,...,\dimX} {

\draw[red]
	(\EvenXArray[\a]*\x, \EvenYArray[\a]*\x) -- (\OddXArray[\a]*\x, \OddYArray[\a]*\x)
;

}
\node (s1c) {} node[below=9mm]{\tiny $B_c \mathcal{O}_{31}$};
\end{scope}

\begin{scope}[xshift=\s cm] 

 \node[align=center, text width=1.6cm,,,text opacity = 0] (c1) {\href{https://www.geogebra.org/m/nrjrgjcz}{a b c d e f g h i j k l m n}};
 
\def\EvenXArray {{1,-1,-1,1,3,-3,-3,-3,-3,3,3,3}}
\def\OddXArray {{0,0,-2,0,2,0,-2,0,-4,0,2,4}}

\def\EvenYArray {{1,1,-1,-1,3,3,-3,-3,-3,-3,-3,-3}}
\def\OddYArray {{2,0,0,0,0,2,0,-2,-4,-2,0,-4}}

\pgfmathtruncatemacro{\dimX}{dim{\EvenXArray}-1} \dimX

\foreach \a in {0,...,\dimX} {

\draw[red]
	(\EvenXArray[\a]*\x, \EvenYArray[\a]*\x) -- (\OddXArray[\a]*\x, \OddYArray[\a]*\x)
;

}
\node (s1c) {} node[below=9mm]{\tiny $B_c \mathcal{O}_{32}$};
\end{scope}

\begin{scope}[xshift=2*\s cm] 

 \node[align=center, text width=1.6cm,,,text opacity = 0] (c1) {\href{https://www.geogebra.org/m/yk7mnnnu}{a b c d e f g h i j k l m n}};
 
\def\EvenXArray {{1,-1,-1,1,3,-3,-3,-3,-3,3,3,3}}
\def\OddXArray {{0,0,0,2,0,-2,-2,0,-4,0,2,4}}

\def\EvenYArray {{1,1,-1,-1,3,3,-3,-3,-3,-3,-3,-3}}
\def\OddYArray {{2,0,0,0,2,0,0,-2,-4,-2,0,-4}}

\pgfmathtruncatemacro{\dimX}{dim{\EvenXArray}-1} \dimX

\foreach \a in {0,...,\dimX} {

\draw[red]
	(\EvenXArray[\a]*\x, \EvenYArray[\a]*\x) -- (\OddXArray[\a]*\x, \OddYArray[\a]*\x)
;

}
\node (s1c) {} node[below=9mm]{\tiny $B_c \mathcal{O}_{33}$};
\end{scope}

\begin{scope}[xshift=3*\s cm] 

 \node[align=center, text width=1.6cm,,,text opacity = 0] (c1) {\href{https://www.geogebra.org/m/shyvcdd7}{a b c d e f g h i j k l m n}};
 
\def\EvenXArray {{1,-1,-1,1,3,-3,-3,-3,-3,3,3,3}}
\def\OddXArray {{2,-2,0,0,4,-4,-2,0,-4,0,2,4}}

\def\EvenYArray {{1,1,-1,-1,3,3,-3,-3,-3,-3,-3,-3}}
\def\OddYArray {{0,0,-2,-2,4,4,0,-2,-4,-2,0,-4}}

\pgfmathtruncatemacro{\dimX}{dim{\EvenXArray}-1} \dimX

\foreach \a in {0,...,\dimX} {

\draw[red]
	(\EvenXArray[\a]*\x, \EvenYArray[\a]*\x) -- (\OddXArray[\a]*\x, \OddYArray[\a]*\x)
;

}
\node (s1c) {} node[below=9mm]{\tiny $B_c \mathcal{O}_{34}$};
\end{scope}

\begin{scope}[xshift=4*\s cm] 

 \node[align=center, text width=1.6cm,,,text opacity = 0] (c1) {\href{https://www.geogebra.org/m/dpc4tjc2}{a b c d e f g h i j k l m n}};
 
\def\EvenXArray {{1,-1,-1,1,3,-3,-3,-3,-3,3,3,3}}
\def\OddXArray {{2,-2,0,0,0,0,-2,0,-4,0,2,4}}

\def\EvenYArray {{1,1,-1,-1,3,3,-3,-3,-3,-3,-3,-3}}
\def\OddYArray {{0,0,-2,-2,2,2,0,-2,-4,-2,0,-4}}

\pgfmathtruncatemacro{\dimX}{dim{\EvenXArray}-1} \dimX

\foreach \a in {0,...,\dimX} {

\draw[red]
	(\EvenXArray[\a]*\x, \EvenYArray[\a]*\x) -- (\OddXArray[\a]*\x, \OddYArray[\a]*\x)
;

}
\node (s1c) {} node[below=9mm]{\tiny $B_c \mathcal{O}_{35}$};
\end{scope}

\end{scope}

\begin{scope}[yshift=-7*\s cm] 

\begin{scope}[] 

 \node[align=center, text width=1.6cm,,,text opacity = 0] (c1) {\href{https://www.geogebra.org/m/bgfpa829}{a b c d e f g h i j k l m n}};
 
\def\EvenXArray {{1,-1,-1,1,3,-3,-3,-3,-3,3,3,3}}
\def\OddXArray {{2,-2,-2,2,2,-2,-2,0,-4,0,2,4}}

\def\EvenYArray {{1,1,-1,-1,3,3,-3,-3,-3,-3,-3,-3}}
\def\OddYArray {{0,0,0,0,0,0,0,-2,-4,-2,0,-4}}

\pgfmathtruncatemacro{\dimX}{dim{\EvenXArray}-1} \dimX

\foreach \a in {0,...,\dimX} {

\draw[red]
	(\EvenXArray[\a]*\x, \EvenYArray[\a]*\x) -- (\OddXArray[\a]*\x, \OddYArray[\a]*\x)
;

}
\node (s1c) {} node[below=9mm]{\tiny $B_c \mathcal{O}_{36}$};
\end{scope}

\begin{scope}[xshift=\s cm] 

 \node[align=center, text width=1.6cm,,,text opacity = 0] (c1) {\href{https://www.geogebra.org/m/ade4vwgz}{a b c d e f g h i j k l m n}};
 
\def\EvenXArray {{1,-1,-1,1,3,-3,-3,-3,-3,3,3,3}}
\def\OddXArray {{2,-2,0,0,0,0,-2,0,-4,0,2,4}}

\def\EvenYArray {{1,1,-1,-1,3,3,-3,-3,-3,-3,-3,-3}}
\def\OddYArray {{0,0,0,0,2,2,0,-2,-4,-2,0,-4}}

\pgfmathtruncatemacro{\dimX}{dim{\EvenXArray}-1} \dimX

\foreach \a in {0,...,\dimX} {

\draw[red]
	(\EvenXArray[\a]*\x, \EvenYArray[\a]*\x) -- (\OddXArray[\a]*\x, \OddYArray[\a]*\x)
;

}
\node (s1c) {} node[below=9mm]{\tiny $B_c \mathcal{O}_{37}$};
\end{scope}

\begin{scope}[xshift=2*\s cm] 

 \node[align=center, text width=1.6cm,,,text opacity = 0] (c1) {\href{https://www.geogebra.org/m/p7zezrxp}{a b c d e f g h i j k l m n}};
 
\def\EvenXArray {{1,-1,-1,1,3,-3,-3,-3,-3,3,3,3}}
\def\OddXArray {{0,-2,0,2,4,0,-2,-4,-4,0,2,4}}

\def\EvenYArray {{1,1,-1,-1,3,3,3,3,-3,-3,-3,-3}}
\def\OddYArray {{2,0,-2,0,4,2,0,4,-4,-2,0,-4}}

\pgfmathtruncatemacro{\dimX}{dim{\EvenXArray}-1} \dimX

\foreach \a in {0,...,\dimX} {

\draw[red]
	(\EvenXArray[\a]*\x, \EvenYArray[\a]*\x) -- (\OddXArray[\a]*\x, \OddYArray[\a]*\x)
;
}
\node (s1c) {} node[below=9mm]{\tiny $B_c \mathcal{O}_{38}$};
\end{scope}

\begin{scope}[xshift=3*\s cm] 

 \node[align=center, text width=1.6cm,,,text opacity = 0] (c1) {\href{https://www.geogebra.org/m/b5s29m8w}{a b c d e f g h i j k l m n}};

\def\EvenXArray {{1,-1,-1,1,3,-3,-3,-3,-3,3,3,3}}
\def\OddXArray {{0,-2,0,0,2,0,-2,-4,0,0,2,4}}

\def\EvenYArray {{1,1,-1,-1,3,3,3,3,-3,-3,-3,-3}}
\def\OddYArray {{2,0,-2,-2,0,2,0,4,-2,-2,0,-4}}

\pgfmathtruncatemacro{\dimX}{dim{\EvenXArray}-1} \dimX

\foreach \a in {0,...,\dimX} {

\draw[red]
	(\EvenXArray[\a]*\x, \EvenYArray[\a]*\x) -- (\OddXArray[\a]*\x, \OddYArray[\a]*\x)
;

}
\node (s1c) {} node[below=9mm]{\tiny $B_c \mathcal{O}_{39}$};
\end{scope}

\begin{scope}[xshift=4*\s cm] 

 \node[align=center, text width=1.6cm,,,text opacity = 0] (c1) {\href{https://www.geogebra.org/m/cefa3nmj}{a b c d e f g h i j k l m n}};

\def\EvenXArray {{1,-1,-1,1,3,-3,-3,-3,-3,3,3,3}}
\def\OddXArray {{0,-2,-2,0,2,0,-2,-4,-2,0,2,4}}

\def\EvenYArray {{1,1,-1,-1,3,3,3,3,-3,-3,-3,-3}}
\def\OddYArray {{2,0,0,-2,0,2,0,4,0,-2,0,-4}}

\pgfmathtruncatemacro{\dimX}{dim{\EvenXArray}-1} \dimX

\foreach \a in {0,...,\dimX} {

\draw[red]
	(\EvenXArray[\a]*\x, \EvenYArray[\a]*\x) -- (\OddXArray[\a]*\x, \OddYArray[\a]*\x)
;

}
\node (s1c) {} node[below=9mm]{\tiny $B_c \mathcal{O}_{40}$};
\end{scope}

\end{scope}


\foreach \xs in {0,...,4} {

\foreach \ys in {0,...,7} {

\tikzset{xshift=\xs*\s cm, yshift=-\ys*\s cm}

\foreach \a in {0,1,2,3} {

\tikzset{rotate=\a*\th}
	
\draw[gray!75, double, line width=0.5]
	(\x,\x) -- (-\x,\x)
	(\l*\x,\l*\x) -- (-\l*\x,\l*\x)
	(\x,\x) -- (\l*\x,\l*\x) 
;

\draw[blue!80!black]
	(0,0) -- (2*\x,0) -- (4*\x,0)
;

}

\draw[blue!80!black] (0,0) circle (2*\x);

\filldraw[blue!80!black] 
	(0,0) circle (\rr)
	(2*\x,0) circle (\rr)
	(0,2*\x) circle (\rr)
	(-2*\x,0) circle (\rr)
	(0,-2*\x) circle (\rr)
;

}
}

\end{tikzpicture}
\caption{The $c$-edge assignments in $B\mathcal{O}$ up to isomorphism}
\label{Fig:BP8-c-I}
\end{figure}


\begin{figure}[h!] 
\centering
\ContinuedFloat
\begin{tikzpicture}

\tikzmath{
\s=2.4;
\x=0.25;
\l=3;
\th=360/4;
\rr=0.15*\x;
}

\begin{scope}[]

\begin{scope}[]

 \node[align=center, text width=1.6cm,,,text opacity = 0] (c1) {\href{https://www.geogebra.org/m/wykrpnwk}{a b c d e f g h i j k l m n}};
 
\def\EvenXArray {{1,-1,-1,1,3,-3,-3,-3,-3,3,3,3}}
\def\OddXArray {{0,-2,0,0,2,0,-2,-4,0,0,2,4}}

\def\EvenYArray {{1,1,-1,-1,3,3,3,3,-3,-3,-3,-3}}
\def\OddYArray {{2,0,0,0,0,2,0,4,-2,-2,0,-4}}

\pgfmathtruncatemacro{\dimX}{dim{\EvenXArray}-1} \dimX

\foreach \a in {0,...,\dimX} {

\draw[red]
	(\EvenXArray[\a]*\x, \EvenYArray[\a]*\x) -- (\OddXArray[\a]*\x, \OddYArray[\a]*\x)
;

}
\node (s1c) {} node[below=9mm]{\tiny $B_c \mathcal{O}_{41}$};
\end{scope}

\begin{scope}[xshift=\s cm] 

 \node[align=center, text width=1.6cm,,,text opacity = 0] (c1) {\href{https://www.geogebra.org/m/wmvyuysy}{a b c d e f g h i j k l m n}};
 
\def\EvenXArray {{1,-1,-1,1,3,-3,-3,-3,-3,3,3,3}}
\def\OddXArray {{0,0,0,0,2,0,-2,-4,-2,0,2,4}}

\def\EvenYArray {{1,1,-1,-1,3,3,3,3,-3,-3,-3,-3}}
\def\OddYArray {{2,0,-2,0,0,2,0,4,0,-2,0,-4}}

\pgfmathtruncatemacro{\dimX}{dim{\EvenXArray}-1} \dimX

\foreach \a in {0,...,\dimX} {

\draw[red]
	(\EvenXArray[\a]*\x, \EvenYArray[\a]*\x) -- (\OddXArray[\a]*\x, \OddYArray[\a]*\x)
;

}
\node (s1c) {} node[below=9mm]{\tiny $B_c \mathcal{O}_{42}$};
\end{scope}

\begin{scope}[xshift=2*\s cm] 

 \node[align=center, text width=1.6cm,,,text opacity = 0] (c1) {\href{https://www.geogebra.org/m/evasryab}{a b c d e f g h i j k l m n}};

\def\EvenXArray {{1,-1,-1,1,3,-3,-3,-3,-3,3,3,3}}
\def\OddXArray {{0,0,-2,0,2,0,-2,-4,0,0,2,4}}

\def\EvenYArray {{1,1,-1,-1,3,3,3,3,-3,-3,-3,-3}}
\def\OddYArray {{2,0,0,0,0,2,0,4,-2,-2,0,-4}}

\pgfmathtruncatemacro{\dimX}{dim{\EvenXArray}-1} \dimX

\foreach \a in {0,...,\dimX} {

\draw[red]
	(\EvenXArray[\a]*\x, \EvenYArray[\a]*\x) -- (\OddXArray[\a]*\x, \OddYArray[\a]*\x)
;
}
\node (s1c) {} node[below=9mm]{\tiny $B_c \mathcal{O}_{43}$};
\end{scope}

\begin{scope}[xshift=3*\s cm] 

 \node[align=center, text width=1.6cm,,,text opacity = 0] (c1) {\href{https://www.geogebra.org/m/ntnrvayk}{a b c d e f g h i j k l m n}};

\def\EvenXArray {{1,-1,-1,1,3,-3,-3,-3,-3,3,3,3}}
\def\OddXArray {{0,0,0,0,2,0,-2,-4,-2,0,2,4}}

\def\EvenYArray {{1,1,-1,-1,3,3,3,3,-3,-3,-3,-3}}
\def\OddYArray {{2,0,0,-2,0,2,0,4,0,-2,0,-4}}

\pgfmathtruncatemacro{\dimX}{dim{\EvenXArray}-1} \dimX

\foreach \a in {0,...,\dimX} {

\draw[red]
	(\EvenXArray[\a]*\x, \EvenYArray[\a]*\x) -- (\OddXArray[\a]*\x, \OddYArray[\a]*\x)
;

}
\node (s1c) {} node[below=9mm]{\tiny $B_c \mathcal{O}_{44}$};
\end{scope}

\begin{scope}[xshift=4*\s cm] 

 \node[align=center, text width=1.6cm,,,text opacity = 0] (c1) {\href{https://www.geogebra.org/m/mca9uqse}{a b c d e f g h i j k l m n}};

\def\EvenXArray {{1,-1,-1,1,3,-3,-3,-3,-3,3,3,3}}
\def\OddXArray {{0,-2,0,2,0,0,-2,-4,0,0,2,4}}

\def\EvenYArray {{1,1,-1,-1,3,3,3,3,-3,-3,-3,-3}}
\def\OddYArray {{0,0,0,0,2,2,0,4,-2,-2,0,-4}}

\pgfmathtruncatemacro{\dimX}{dim{\EvenXArray}-1} \dimX

\foreach \a in {0,...,\dimX} {

\draw[red]
	(\EvenXArray[\a]*\x, \EvenYArray[\a]*\x) -- (\OddXArray[\a]*\x, \OddYArray[\a]*\x)
;

}
\node (s1c) {} node[below=9mm]{\tiny $B_c \mathcal{O}_{45}$};
\end{scope}

\end{scope}

\begin{scope}[yshift=-1*\s cm] 

\begin{scope}[] 

 \node[align=center, text width=1.6cm,,,text opacity = 0] (c1) {\href{https://www.geogebra.org/m/cjn9gkv6}{a b c d e f g h i j k l m n}};

\def\EvenXArray {{1,-1,-1,1,3,-3,-3,-3,-3,-3,-3,3,3,3}}
\def\OddXArray {{0,-2,-2,2,4,0,-2,-4,-2,0,-4,0,2,4}}

\def\EvenYArray {{1,1,-1,-1,3,3,3,3,-3,-3,-3,-3,-3,-3}}
\def\OddYArray {{2,0,0,0,4,2,0,4,0,-2,-4,-2,0,-4}}

\pgfmathtruncatemacro{\dimX}{dim{\EvenXArray}-1} \dimX

\foreach \a in {0,...,\dimX} {

\draw[red]
	(\EvenXArray[\a]*\x, \EvenYArray[\a]*\x) -- (\OddXArray[\a]*\x, \OddYArray[\a]*\x)
;

}
\node (s1c) {} node[below=9mm]{\tiny $B_c \mathcal{O}_{46}$};
\end{scope}

\begin{scope}[xshift=\s cm] 

 \node[align=center, text width=1.6cm,,,text opacity = 0] (c1) {\href{https://www.geogebra.org/m/cub7enmw}{a b c d e f g h i j k l m n}};

\def\EvenXArray {{1,-1,-1,1,3,-3,-3,-3,-3,-3,-3,3,3,3}}
\def\OddXArray {{0,0,0,2,4,0,-2,-4,-2,0,-4,0,2,4}}

\def\EvenYArray {{1,1,-1,-1,3,3,3,3,-3,-3,-3,-3,-3,-3}}
\def\OddYArray {{2,0,0,0,4,2,0,4,0,-2,-4,-2,0,-4}}

\pgfmathtruncatemacro{\dimX}{dim{\EvenXArray}-1} \dimX

\foreach \a in {0,...,\dimX} {

\draw[red]
	(\EvenXArray[\a]*\x, \EvenYArray[\a]*\x) -- (\OddXArray[\a]*\x, \OddYArray[\a]*\x)
;

}
\node (s1c) {} node[below=9mm]{\tiny $B_c \mathcal{O}_{47}$};
\end{scope}

\begin{scope}[xshift=2*\s cm] 

 \node[align=center, text width=1.6cm,,,text opacity = 0] (c1) {\href{https://www.geogebra.org/m/nk7pdueq}{a b c d e f g h i j k l m n}};

\def\EvenXArray {{1,-1,-1,1,3,-3,-3,-3,-3,-3,-3,3,3,3}}
\def\OddXArray {{0,0,0,2,4,0,-2,-4,-2,0,-4,0,2,4}}

\def\EvenYArray {{1,1,-1,-1,3,3,3,3,-3,-3,-3,-3,-3,-3}}
\def\OddYArray {{0,2,0,0,4,2,0,4,0,-2,-4,-2,0,-4}}

\pgfmathtruncatemacro{\dimX}{dim{\EvenXArray}-1} \dimX

\foreach \a in {0,...,\dimX} {

\draw[red]
	(\EvenXArray[\a]*\x, \EvenYArray[\a]*\x) -- (\OddXArray[\a]*\x, \OddYArray[\a]*\x)
;

}
\node (s1c) {} node[below=9mm]{\tiny $B_c \mathcal{O}_{48}$};
\end{scope}

\begin{scope}[xshift=3*\s cm] 

 \node[align=center, text width=1.6cm,,,text opacity = 0] (c1) {\href{https://www.geogebra.org/m/hjw4jff5}{a b c d e f g h i j k l m n}};

\def\EvenXArray {{1,-1,-1,1,3,3,3,-3,-3,-3,-3,-3,-3,3,3,3}}
\def\OddXArray {{0,0,0,0,2,0,4,0,-2,-4,-2,0,-4,0,2,4}}

\def\EvenYArray {{1,1,-1,-1,3,3,3,3,3,3,-3,-3,-3,-3,-3,-3}}
\def\OddYArray {{2,2,-2,-2,0,2,4,2,0,4,0,-2,-4,-2,0,-4}}

\pgfmathtruncatemacro{\dimX}{dim{\EvenXArray}-1} \dimX

\foreach \a in {0,...,\dimX} {

\draw[red]
	(\EvenXArray[\a]*\x, \EvenYArray[\a]*\x) -- (\OddXArray[\a]*\x, \OddYArray[\a]*\x)
;

}
\node (s1c) {} node[below=9mm]{\tiny $B_c \mathcal{O}_{49}$};
\end{scope}

\begin{scope}[xshift=4*\s cm] 

 \node[align=center, text width=1.6cm,,,text opacity = 0] (c1) {\href{https://www.geogebra.org/m/nypuxme5}{a b c d e f g h i j k l m n}};

\def\EvenXArray {{1,-1,-1,1,3,3,3,-3,-3,-3,-3,-3,-3,3,3,3}}
\def\OddXArray {{0,0,0,0,2,0,4,0,-2,-4,-2,0,-4,0,2,4}}

\def\EvenYArray {{1,1,-1,-1,3,3,3,3,3,3,-3,-3,-3,-3,-3,-3}}
\def\OddYArray {{2,2,0,0,0,2,4,2,0,4,0,-2,-4,-2,0,-4}}

\pgfmathtruncatemacro{\dimX}{dim{\EvenXArray}-1} \dimX

\foreach \a in {0,...,\dimX} {

\draw[red]
	(\EvenXArray[\a]*\x, \EvenYArray[\a]*\x) -- (\OddXArray[\a]*\x, \OddYArray[\a]*\x)
;

}
\node (s1c) {} node[below=9mm]{\tiny $B_c \mathcal{O}_{50}$};
\end{scope}

\end{scope}

\begin{scope}[yshift=-2*\s cm]

\begin{scope}[] 

 \node[align=center, text width=1.6cm,,,text opacity = 0] (c1) {\href{https://www.geogebra.org/m/c6xt2nmw}{a b c d e f g h i j k l m n}};

\def\EvenXArray {{1,-1,-1,1,3,3,3,-3,-3,-3,-3,-3,-3,3,3,3}}
\def\OddXArray {{0,0,0,0,2,0,4,0,-2,-4,-2,0,-4,0,2,4}}

\def\EvenYArray {{1,1,-1,-1,3,3,3,3,3,3,-3,-3,-3,-3,-3,-3}}
\def\OddYArray {{0,0,0,0,0,2,4,2,0,4,0,-2,-4,-2,0,-4}}

\pgfmathtruncatemacro{\dimX}{dim{\EvenXArray}-1} \dimX

\foreach \a in {0,...,\dimX} {

\draw[red]
	(\EvenXArray[\a]*\x, \EvenYArray[\a]*\x) -- (\OddXArray[\a]*\x, \OddYArray[\a]*\x)
;

}
\node (s1c) {} node[below=9mm]{\tiny $B_c \mathcal{O}_{51}$};
\end{scope}

\begin{scope}[xshift=\s cm] 

 \node[align=center, text width=1.6cm,,,text opacity = 0] (c1) {\href{https://www.geogebra.org/m/byurzrbv}{a b c d e f g h i j k l m n}};

\def\EvenXArray {{1,-1,-1,1,1,1,3,-3,-3,-3,-3,3}}
\def\OddXArray {{0,-2,0,0,0,2,4,0,-2,-4,0,2}}

\def\EvenYArray {{1,1,-1,-1,-1,-1,3,3,3,3,-3,-3}}
\def\OddYArray {{2,0,0,0,-2,0,4,2,0,4,-2,0}}

\pgfmathtruncatemacro{\dimX}{dim{\EvenXArray}-1} \dimX

\foreach \a in {0,...,\dimX} {

\draw[red]
	(\EvenXArray[\a]*\x, \EvenYArray[\a]*\x) -- (\OddXArray[\a]*\x, \OddYArray[\a]*\x)
;

}
\node (s1c) {} node[below=9mm]{\tiny $B_c \mathcal{O}_{52}$};
\end{scope}

\begin{scope}[xshift=2*\s cm] 

 \node[align=center, text width=1.6cm,,,text opacity = 0] (c1) {\href{https://www.geogebra.org/m/na2dppzb}{a b c d e f g h i j k l m n}};

\def\EvenXArray {{1,-1,-1,1,1,1,3,-3,-3,-3,-3,3}}
\def\OddXArray {{0,-2,0,0,0,2,2,0,-2,-4,-4,0}}

\def\EvenYArray {{1,1,-1,-1,-1,-1,3,3,3,3,-3,-3}}
\def\OddYArray {{2,0,0,0,-2,0,0,2,0,4,-4,-2}}

\pgfmathtruncatemacro{\dimX}{dim{\EvenXArray}-1} \dimX

\foreach \a in {0,...,\dimX} {

\draw[red]
	(\EvenXArray[\a]*\x, \EvenYArray[\a]*\x) -- (\OddXArray[\a]*\x, \OddYArray[\a]*\x)
;

}
\node (s1c) {} node[below=9mm]{\tiny $B_c \mathcal{O}_{53}$};
\end{scope}

\begin{scope}[xshift=3*\s cm] 

 \node[align=center, text width=1.6cm,,,text opacity = 0] (c1) {\href{https://www.geogebra.org/m/x5jfkzyu}{a b c d e f g h i j k l m n}};

\def\EvenXArray {{1,-1,-1,1,1,1,3,-3,-3,-3,-3,3}}
\def\OddXArray {{0,-2,0,0,0,2,2,0,-2,-4,0,4}}

\def\EvenYArray {{1,1,-1,-1,-1,-1,3,3,3,3,-3,-3}}
\def\OddYArray {{2,0,0,0,-2,0,0,2,0,4,-2,-4}}

\pgfmathtruncatemacro{\dimX}{dim{\EvenXArray}-1} \dimX

\foreach \a in {0,...,\dimX} {

\draw[red]
	(\EvenXArray[\a]*\x, \EvenYArray[\a]*\x) -- (\OddXArray[\a]*\x, \OddYArray[\a]*\x)
;

}
\node (s1c) {} node[below=9mm]{\tiny $B_c \mathcal{O}_{54}$};
\end{scope}

\begin{scope}[xshift=4*\s cm] 

 \node[align=center, text width=1.6cm,,,text opacity = 0] (c1) {\href{https://www.geogebra.org/m/mbpsnbwn}{a b c d e f g h i j k l m n}};
 
\def\EvenXArray {{1,-1,-1,1,1,1,3,-3,-3,-3,-3,3}}
\def\OddXArray {{0,0,-2,0,0,2,2,0,-2,-4,0,4}}

\def\EvenYArray {{1,1,-1,-1,-1,-1,3,3,3,3,-3,-3}}
\def\OddYArray {{2,0,0,0,-2,0,0,2,0,4,-2,-4}}

\pgfmathtruncatemacro{\dimX}{dim{\EvenXArray}-1} \dimX

\foreach \a in {0,...,\dimX} {

\draw[red]
	(\EvenXArray[\a]*\x, \EvenYArray[\a]*\x) -- (\OddXArray[\a]*\x, \OddYArray[\a]*\x)
;

}
\node (s1c) {} node[below=9mm]{\tiny $B_c \mathcal{O}_{55}$};
\end{scope}

\end{scope}

\begin{scope}[yshift=-3*\s cm] 

\begin{scope}[] 

 \node[align=center, text width=1.6cm,,,text opacity = 0] (c1) {\href{https://www.geogebra.org/m/pcrydjtu}{a b c d e f g h i j k l m n}};

\def\EvenXArray {{1,-1,-1,1,1,1,3,-3,-3,-3,-3,3}}
\def\OddXArray {{2,-2,0,0,0,2,0,0,-2,-4,-4,0}}

\def\EvenYArray {{1,1,-1,-1,-1,-1,3,3,3,3,-3,-3}}
\def\OddYArray {{0,0,0,0,-2,0,2,2,0,4,-4,-2}}

\pgfmathtruncatemacro{\dimX}{dim{\EvenXArray}-1} \dimX

\foreach \a in {0,...,\dimX} {

\draw[red]
	(\EvenXArray[\a]*\x, \EvenYArray[\a]*\x) -- (\OddXArray[\a]*\x, \OddYArray[\a]*\x)
;

}
\node (s1c) {} node[below=9mm]{\tiny $B_c \mathcal{O}_{56}$};
\end{scope}

\begin{scope}[xshift=\s cm] 

 \node[align=center, text width=1.6cm,,,text opacity = 0] (c1) {\href{https://www.geogebra.org/m/wupzm3ng}{a b c d e f g h i j k l m n}};

\def\EvenXArray {{1,-1,-1,1,1,1,3,-3,-3,-3,-3,3}}
\def\OddXArray {{2,0,0,0,0,2,4,0,-2,-4,-2,0}}

\def\EvenYArray {{1,1,-1,-1,-1,-1,3,3,3,3,-3,-3}}
\def\OddYArray {{0,2,0,0,-2,0,4,2,0,4,0,-2}}

\pgfmathtruncatemacro{\dimX}{dim{\EvenXArray}-1} \dimX

\foreach \a in {0,...,\dimX} {

\draw[red]
	(\EvenXArray[\a]*\x, \EvenYArray[\a]*\x) -- (\OddXArray[\a]*\x, \OddYArray[\a]*\x)
;

}
\node (s1c) {} node[below=9mm]{\tiny $B_c \mathcal{O}_{57}$};
\end{scope}

\begin{scope}[xshift=2*\s cm] 

 \node[align=center, text width=1.6cm,,,text opacity = 0] (c1) {\href{https://www.geogebra.org/m/ukmuz89e}{a b c d e f g h i j k l m n}};

\def\EvenXArray {{1,-1,-1,1,1,1,3,-3,-3,-3,-3,3,3,3}}
\def\OddXArray {{0,-2,0,0,0,2,4,0,-2,-4,-4,0,2,4}}

\def\EvenYArray {{1,1,-1,-1,-1,-1,3,3,3,3,-3,-3,-3,-3}}
\def\OddYArray {{2,0,0,0,-2,0,4,2,0,4,-4,-2,0,-4}}

\pgfmathtruncatemacro{\dimX}{dim{\EvenXArray}-1} \dimX

\foreach \a in {0,...,\dimX} {

\draw[red]
	(\EvenXArray[\a]*\x, \EvenYArray[\a]*\x) -- (\OddXArray[\a]*\x, \OddYArray[\a]*\x)
;

}
\node (s1c) {} node[below=9mm]{\tiny $B_c \mathcal{O}_{58}$};
\end{scope}

\begin{scope}[xshift=3*\s cm] 

 \node[align=center, text width=1.6cm,,,text opacity = 0] (c1) {\href{https://www.geogebra.org/m/dryqmqk6}{a b c d e f g h i j k l m n}};
 
\def\EvenXArray {{1,-1,-1,1,1,1,3,-3,-3,-3,-3,3,3,3}}
\def\OddXArray {{0,0,0,0,0,2,0,0,-2,-4,-2,0,2,4}}

\def\EvenYArray {{1,1,-1,-1,-1,-1,3,3,3,3,-3,-3,-3,-3}}
\def\OddYArray {{2,2,0,0,-2,0,2,2,0,4,0,-2,0,-4}}

\pgfmathtruncatemacro{\dimX}{dim{\EvenXArray}-1} \dimX

\foreach \a in {0,...,\dimX} {

\draw[red]
	(\EvenXArray[\a]*\x, \EvenYArray[\a]*\x) -- (\OddXArray[\a]*\x, \OddYArray[\a]*\x)
;

}
\node (s1c) {} node[below=9mm]{\tiny $B_c \mathcal{O}_{59}$};
\end{scope}

\begin{scope}[xshift=4*\s cm] 

 \node[align=center, text width=1.6cm,,,text opacity = 0] (c1) {\href{https://www.geogebra.org/m/dzfdycpj}{a b c d e f g h i j k l m n}};

\def\EvenXArray {{1,-1,-1,1,1,1,3,-3,-3,-3,-3,3,3,3}}
\def\OddXArray {{0,0,-2,0,0,2,4,0,-2,-4,-4,0,2,4}}

\def\EvenYArray {{1,1,-1,-1,-1,-1,3,3,3,3,-3,-3,-3,-3}}
\def\OddYArray {{2,0,0,0,-2,0,4,2,0,4,-4,-2,0,-4}}

\pgfmathtruncatemacro{\dimX}{dim{\EvenXArray}-1} \dimX

\foreach \a in {0,...,\dimX} {

\draw[red]
	(\EvenXArray[\a]*\x, \EvenYArray[\a]*\x) -- (\OddXArray[\a]*\x, \OddYArray[\a]*\x)
;

}
\node (s1c) {} node[below=9mm]{\tiny $B_c \mathcal{O}_{60}$};
\end{scope}

\end{scope}

\begin{scope}[yshift=-4*\s cm]

\begin{scope}[] 

 \node[align=center, text width=1.6cm,,,text opacity = 0] (c1) {\href{https://www.geogebra.org/m/frw8hgj2}{a b c d e f g h i j k l m n}};
 
\def\EvenXArray {{1,-1,-1,1,1,1,3,-3,-3,-3,-3,3,3,3}}
\def\OddXArray {{2,0,0,0,0,2,2,0,-2,-4,-2,0,2,4}}

\def\EvenYArray {{1,1,-1,-1,-1,-1,3,3,3,3,-3,-3,-3,-3}}
\def\OddYArray {{0,2,0,0,-2,0,0,2,0,4,0,-2,0,-4}}

\pgfmathtruncatemacro{\dimX}{dim{\EvenXArray}-1} \dimX

\foreach \a in {0,...,\dimX} {

\draw[red]
	(\EvenXArray[\a]*\x, \EvenYArray[\a]*\x) -- (\OddXArray[\a]*\x, \OddYArray[\a]*\x)
;

}
\node (s1c) {} node[below=9mm]{\tiny $B_c \mathcal{O}_{61}$};
\end{scope}

\begin{scope}[xshift=\s cm] 

 \node[align=center, text width=1.6cm,,,text opacity = 0] (c1) {\href{https://www.geogebra.org/m/pqunyan8}{a b c d e f g h i j k l m n}};
 
\def\EvenXArray {{1,-1,-1,1,1,1,3,-3,-3,-3,-3,3,3,3}}
\def\OddXArray {{0,0,0,0,0,2,0,0,-2,-4,-2,0,2,4}}

\def\EvenYArray {{1,1,-1,-1,-1,-1,3,3,3,3,-3,-3,-3,-3}}
\def\OddYArray {{0,0,0,0,-2,0,2,2,0,4,0,-2,0,-4}}

\pgfmathtruncatemacro{\dimX}{dim{\EvenXArray}-1} \dimX

\foreach \a in {0,...,\dimX} {

\draw[red]
	(\EvenXArray[\a]*\x, \EvenYArray[\a]*\x) -- (\OddXArray[\a]*\x, \OddYArray[\a]*\x)
;

}
\node (s1c) {} node[below=9mm]{\tiny $B_c \mathcal{O}_{62}$};
\end{scope}

\begin{scope}[xshift=2*\s cm] 

 \node[align=center, text width=1.6cm,,,text opacity = 0] (c1) {\href{https://www.geogebra.org/m/nbbsrc7a}{a b c d e f g h i j k l m n}};

\def\EvenXArray {{1,-1,-1,1,1,1,3,-3,-3,-3,-3,-3,-3,3,3,3}}
\def\OddXArray {{0,0,0,0,0,2,4,0,-2,-4,-2,0,-4,0,2,4}}

\def\EvenYArray {{1,1,-1,-1,-1,-1,3,3,3,3,-3,-3,-3,-3,-3,-3}}
\def\OddYArray {{2,0,-2,0,-2,0,4,2,0,4,0,-2,-4,-2,0,-4}}

\pgfmathtruncatemacro{\dimX}{dim{\EvenXArray}-1} \dimX

\foreach \a in {0,...,\dimX} {

\draw[red]
	(\EvenXArray[\a]*\x, \EvenYArray[\a]*\x) -- (\OddXArray[\a]*\x, \OddYArray[\a]*\x)
;

}
\node (s1c) {} node[below=9mm]{\tiny $B_c \mathcal{O}_{63}$};
\end{scope}

\begin{scope}[xshift=3*\s cm] 

 \node[align=center, text width=1.6cm,,,text opacity = 0] (c1) {\href{https://www.geogebra.org/m/jtzrvd4j}{a b c d e f g h i j k l m n}};

\def\EvenXArray {{1,-1,-1,1,1,1,3,-3,-3,-3,-3,-3,-3,3,3,3}}
\def\OddXArray {{0,0,0,0,0,2,4,0,-2,-4,-2,0,-4,0,2,4}}

\def\EvenYArray {{1,1,-1,-1,-1,-1,3,3,3,3,-3,-3,-3,-3,-3,-3}}
\def\OddYArray {{0,2,-2,0,-2,0,4,2,0,4,0,-2,-4,-2,0,-4}}

\pgfmathtruncatemacro{\dimX}{dim{\EvenXArray}-1} \dimX

\foreach \a in {0,...,\dimX} {

\draw[red]
	(\EvenXArray[\a]*\x, \EvenYArray[\a]*\x) -- (\OddXArray[\a]*\x, \OddYArray[\a]*\x)
;

}
\node (s1c) {} node[below=9mm]{\tiny $B_c \mathcal{O}_{64}$};
\end{scope}

\begin{scope}[xshift=4*\s cm] 

 \node[align=center, text width=1.6cm,,,text opacity = 0] (c1) {\href{https://www.geogebra.org/m/xkqcx2zu}{a b c d e f g h i j k l m n}};

\def\EvenXArray {{1,-1,-1,1,1,1,3,3,3,-3,-3,-3,-3,3}}
\def\OddXArray {{0,-2,0,0,0,2,2,0,4,-4,-2,0,-4,4}}

\def\EvenYArray {{1,1,-1,-1,-1,-1,3,3,3,3,-3,-3,-3,-3}}
\def\OddYArray {{2,0,0,0,-2,0,0,2,4,4,0,-2,-4,-4}}

\pgfmathtruncatemacro{\dimX}{dim{\EvenXArray}-1} \dimX

\foreach \a in {0,...,\dimX} {

\draw[red]
	(\EvenXArray[\a]*\x, \EvenYArray[\a]*\x) -- (\OddXArray[\a]*\x, \OddYArray[\a]*\x)
;

}
\node (s1c) {} node[below=9mm]{\tiny $B_c \mathcal{O}_{65}$};
\end{scope}

\end{scope}

\begin{scope}[yshift=-5*\s cm] 

\begin{scope}[] 

 \node[align=center, text width=1.6cm,,,text opacity = 0] (c1) {\href{https://www.geogebra.org/m/qcesvu2n}{a b c d e f g h i j k l m n}};

\def\EvenXArray {{1,-1,-1,1,1,1,3,3,3,-3,-3,-3,-3,3}}
\def\OddXArray {{0,0,-2,0,0,2,2,0,4,-4,-2,0,-4,4}}

\def\EvenYArray {{1,1,-1,-1,-1,-1,3,3,3,3,-3,-3,-3,-3}}
\def\OddYArray {{2,0,0,0,-2,0,0,2,4,4,0,-2,-4,-4}}

\pgfmathtruncatemacro{\dimX}{dim{\EvenXArray}-1} \dimX

\foreach \a in {0,...,\dimX} {

\draw[red]
	(\EvenXArray[\a]*\x, \EvenYArray[\a]*\x) -- (\OddXArray[\a]*\x, \OddYArray[\a]*\x)
;

}
\node (s1c) {} node[below=9mm]{\tiny $B_c \mathcal{O}_{66}$};
\end{scope}

\begin{scope}[xshift=\s cm] 

 \node[align=center, text width=1.6cm,,,text opacity = 0] (c1) {\href{https://www.geogebra.org/m/f4fpcakb}{a b c d e f g h i j k l m n}};

\def\EvenXArray {{1,-1,-1,1,1,1,3,3,3,-3,-3,-3,-3,-3,-3,3}}
\def\OddXArray {{0,0,0,0,0,2,2,0,4,0,-2,-4,-2,0,-4,4}}

\def\EvenYArray {{1,1,-1,-1,-1,-1,3,3,3,3,3,3,-3,-3,-3,-3}}
\def\OddYArray {{2,2,0,0,-2,0,0,2,4,2,0,4,0,-2,-4,-4}}

\pgfmathtruncatemacro{\dimX}{dim{\EvenXArray}-1} \dimX

\foreach \a in {0,...,\dimX} {

\draw[red]
	(\EvenXArray[\a]*\x, \EvenYArray[\a]*\x) -- (\OddXArray[\a]*\x, \OddYArray[\a]*\x)
;

}
\node (s1c) {} node[below=9mm]{\tiny $B_c \mathcal{O}_{67}$};
\end{scope}

\begin{scope}[xshift=2*\s cm] 

 \node[align=center, text width=1.6cm,,,text opacity = 0] (c1) {\href{https://www.geogebra.org/m/yfdtk8sv}{a b c d e f g h i j k l m n}};
 
\def\EvenXArray {{1,-1,-1,1,1,1,3,3,3,-3,-3,-3,-3,-3,-3,3}}
\def\OddXArray {{0,0,0,0,0,2,2,0,4,0,-2,-4,-2,0,-4,4}}

\def\EvenYArray {{1,1,-1,-1,-1,-1,3,3,3,3,3,3,-3,-3,-3,-3}}
\def\OddYArray {{0,0,0,0,-2,0,0,2,4,2,0,4,0,-2,-4,-4}}

\pgfmathtruncatemacro{\dimX}{dim{\EvenXArray}-1} \dimX

\foreach \a in {0,...,\dimX} {

\draw[red]
	(\EvenXArray[\a]*\x, \EvenYArray[\a]*\x) -- (\OddXArray[\a]*\x, \OddYArray[\a]*\x)
;

}
\node (s1c) {} node[below=9mm]{\tiny $B_c \mathcal{O}_{68}$};
\end{scope}

\begin{scope}[xshift=3*\s cm] 

 \node[align=center, text width=1.6cm,,,text opacity = 0] (c1) {\href{https://www.geogebra.org/m/caq6rauf}{a b c d e f g h i j k l m n}};

\def\EvenXArray {{1,-1,-1,1,1,1,3,3,3,-3,-3,-3,-3,-3,-3,3,3,3}}
\def\OddXArray {{2,0,0,0,0,2,2,0,4,0,-2,-4,-2,0,-4,0,2,4}}

\def\EvenYArray {{1,1,-1,-1,-1,-1,3,3,3,3,3,3,-3,-3,-3,-3,-3,-3}}
\def\OddYArray {{0,0,-2,0,-2,0,0,2,4,2,0,4,0,-2,-4,-2,0,-4}}

\pgfmathtruncatemacro{\dimX}{dim{\EvenXArray}-1} \dimX

\foreach \a in {0,...,\dimX} {

\draw[red]
	(\EvenXArray[\a]*\x, \EvenYArray[\a]*\x) -- (\OddXArray[\a]*\x, \OddYArray[\a]*\x)
;

}
\node (s1c) {} node[below=9mm]{\tiny $B_c \mathcal{O}_{69}$};
\end{scope}

\begin{scope}[xshift=4*\s cm] 

 \node[align=center, text width=1.6cm,,,text opacity = 0] (c1) {\href{https://www.geogebra.org/m/h25sqmqr}{a b c d e f g h i j k l m n}};
 
\def\EvenXArray {{1,-1,-1,-1,-1,1,1,1,3,3,3,-3,-3,-3,-3,3}}
\def\OddXArray {{0,0,0,-2,0,0,0,2,2,0,4,0,-2,-4,-4,4}}

\def\EvenYArray {{1,1,-1,-1,-1,-1,-1,-1,3,3,3,3,3,3,-3,-3}}
\def\OddYArray {{2,2,0,0,-2,0,-2,0,0,2,4,2,0,4,-4,-4}}

\pgfmathtruncatemacro{\dimX}{dim{\EvenXArray}-1} \dimX

\foreach \a in {0,...,\dimX} {

\draw[red]
	(\EvenXArray[\a]*\x, \EvenYArray[\a]*\x) -- (\OddXArray[\a]*\x, \OddYArray[\a]*\x)
;

}
\node (s1c) {} node[below=9mm]{\tiny $B_c \mathcal{O}_{70}$};
\end{scope}

\end{scope}

\begin{scope}[yshift=-6*\s cm] 

\begin{scope}[] 
 
 \node[align=center, text width=1.6cm,,,text opacity = 0] (c1) {\href{https://www.geogebra.org/m/jr8zyueu}{a b c d e f g h i j k l m n}}; 
 
\def\EvenXArray {{1,-1,-1,-1,-1,1,1,1,3,3,3,-3,-3,-3,-3,3}}
\def\OddXArray {{0,0,0,-2,0,0,0,2,2,0,4,0,-2,-4,0,0}}

\def\EvenYArray {{1,1,-1,-1,-1,-1,-1,-1,3,3,3,3,3,3,-3,-3}}
\def\OddYArray {{2,2,0,0,-2,0,-2,0,0,2,4,2,0,4,-2,-2}}

\pgfmathtruncatemacro{\dimX}{dim{\EvenXArray}-1} \dimX

\foreach \a in {0,...,\dimX} {

\draw[red]
	(\EvenXArray[\a]*\x, \EvenYArray[\a]*\x) -- (\OddXArray[\a]*\x, \OddYArray[\a]*\x)
;

}
\node (s1c) {} node[below=9mm]{\tiny $B_c \mathcal{O}_{71}$};
\end{scope}

\begin{scope}[xshift=\s cm] 

 \node[align=center, text width=1.6cm,,,text opacity = 0] (c1) {\href{https://www.geogebra.org/m/uxdzkhuv}{a b c d e f g h i j k l m n}};
 
\def\EvenXArray {{1,-1,-1,-1,-1,1,1,1,3,3,3,-3,-3,-3,-3,3}}
\def\OddXArray {{2,-2,0,-2,0,0,0,2,2,0,4,0,-2,-4,-2,2}}

\def\EvenYArray {{1,1,-1,-1,-1,-1,-1,-1,3,3,3,3,3,3,-3,-3}}
\def\OddYArray {{0,0,0,0,-2,0,-2,0,0,2,4,2,0,4,0,0}}

\pgfmathtruncatemacro{\dimX}{dim{\EvenXArray}-1} \dimX

\foreach \a in {0,...,\dimX} {

\draw[red]
	(\EvenXArray[\a]*\x, \EvenYArray[\a]*\x) -- (\OddXArray[\a]*\x, \OddYArray[\a]*\x)
;

}
\node (s1c) {} node[below=9mm]{\tiny $B_c \mathcal{O}_{72}$};
\end{scope}

\begin{scope}[xshift=2*\s cm] 

 \node[align=center, text width=1.6cm,,,text opacity = 0] (c1) {\href{https://www.geogebra.org/m/gqyp7vgh}{a b c d e f g h i j k l m n}};

\def\EvenXArray {{1,-1,-1,-1,-1,1,1,1,3,3,3,-3,-3,-3,-3,-3,-3,3,3,3}}
\def\OddXArray {{2,-2,0,-2,0,0,0,2,2,0,4,0,-2,-4,-2,0,-4,0,2,4}}

\def\EvenYArray {{1,1,-1,-1,-1,-1,-1,-1,3,3,3,3,3,3,-3,-3,-3,-3,-3,-3}}
\def\OddYArray {{0,0,0,0,-2,0,-2,0,0,2,4,2,0,4,0,-2,-4,-2,0,-4}}

\pgfmathtruncatemacro{\dimX}{dim{\EvenXArray}-1} \dimX

\foreach \a in {0,...,\dimX} {

\draw[red]
	(\EvenXArray[\a]*\x, \EvenYArray[\a]*\x) -- (\OddXArray[\a]*\x, \OddYArray[\a]*\x)
;

}
\node (s1c) {} node[below=9mm]{\tiny $B_c \mathcal{O}_{73}$};
\end{scope}

\begin{scope}[xshift=3*\s cm] 

 \node[align=center, text width=1.6cm,,,text opacity = 0] (c1) {\href{https://www.geogebra.org/m/ecquksbm}{a b c d e f g h i j k l m n}};
 
\def\EvenXArray {{1,-1,-1,-1,-1,1,1,1,3,3,3,-3,-3,-3,-3,3}}
\def\OddXArray {{0,0,0,-2,0,0,0,2,2,0,4,0,-2,0,-4,0}}

\def\EvenYArray {{1,1,1,1,-1,-1,-1,-1,3,3,3,3,-3,-3,-3,-3}}
\def\OddYArray {{2,0,2,0,-2,0,-2,0,0,2,4,2,0,-2,-4,-2}}

\pgfmathtruncatemacro{\dimX}{dim{\EvenXArray}-1} \dimX

\foreach \a in {0,...,\dimX} {

\draw[red]
	(\EvenXArray[\a]*\x, \EvenYArray[\a]*\x) -- (\OddXArray[\a]*\x, \OddYArray[\a]*\x)
;

}
\node (s1c) {} node[below=9mm]{\tiny $B_c \mathcal{O}_{74}$};
\end{scope}

\begin{scope}[xshift=4*\s cm] 

 \node[align=center, text width=1.6cm,,,text opacity = 0] (c1) {\href{https://www.geogebra.org/m/dfpnkpxz}{a b c d e f g h i j k l m n}};

\def\EvenXArray {{1,-1,-1,-1,-1,1,1,1,3,3,3,-3,-3,-3,-3,3,3,3}}
\def\OddXArray {{2,0,0,-2,0,0,0,2,2,0,4,-4,-2,0,-4,0,2,4}}

\def\EvenYArray {{1,1,1,1,-1,-1,-1,-1,3,3,3,3,-3,-3,-3,-3,-3,-3}}
\def\OddYArray {{0,0,2,0,-2,0,-2,0,0,2,4,4,0,-2,-4,-2,0,-4}}

\pgfmathtruncatemacro{\dimX}{dim{\EvenXArray}-1} \dimX

\foreach \a in {0,...,\dimX} {

\draw[red]
	(\EvenXArray[\a]*\x, \EvenYArray[\a]*\x) -- (\OddXArray[\a]*\x, \OddYArray[\a]*\x)
;

}
\node (s1c) {} node[below=9mm]{\tiny $B_c \mathcal{O}_{75}$};
\end{scope}

\end{scope}


\foreach \xs in {0,...,4} {

\foreach \ys in {0,...,6} {

\tikzset{xshift=\xs*\s cm, yshift=-\ys*\s cm}

\foreach \a in {0,1,2,3} {

\tikzset{rotate=\a*\th}
	
\draw[gray!75, double, line width=0.5]
	(\x,\x) -- (-\x,\x)
	(\l*\x,\l*\x) -- (-\l*\x,\l*\x)
	(\x,\x) -- (\l*\x,\l*\x) 
;

\draw[blue!80!black]
	(0,0) -- (2*\x,0) -- (4*\x,0)
;

}

\draw[blue!80!black] (0,0) circle (2*\x);

\filldraw[blue!80!black] 
	(0,0) circle (\rr)
	(2*\x,0) circle (\rr)
	(0,2*\x) circle (\rr)
	(-2*\x,0) circle (\rr)
	(0,-2*\x) circle (\rr)
;

}
}

\end{tikzpicture}
\caption{The $c$-edge assignments in $B\mathcal{O}$ up to isomorphism (cont.)}
\end{figure}


\begin{figure}[h!] 
\centering
\begin{tikzpicture}

\tikzmath{
\s=2.4;
\x=0.25;
\l=3;
\th=360/4;
\rr=0.15*\x;
}

\begin{scope}[]

\begin{scope}[]

 \node[align=center, text width=1.6cm,,,text opacity = 0] (c1) {\href{https://www.geogebra.org/m/bvadyjsn}{a b c d e f g h i j k l m n}};

\def\EvenXArray {{1,0,-1,0,3,-3,-3,3}}
\def\OddXArray {{sqrt(2),sqrt(2),-sqrt(2),-sqrt(2),4,-4,-4,4}}

\def\EvenYArray {{0,1,0,-1,3,3,-3,-3}}
\def\OddYArray {{sqrt(2),sqrt(2),-sqrt(2),-sqrt(2),4,4,-4,-4}}

\pgfmathtruncatemacro{\dimX}{dim{\EvenXArray}-1} \dimX

\foreach \a in {0,...,\dimX} {

\draw[red]
	(\EvenXArray[\a]*\x, \EvenYArray[\a]*\x) -- (\OddXArray[\a]*\x, \OddYArray[\a]*\x)
;
}
\node (s1c) {} node[below=9mm]{\tiny $FB_c \mathcal{O}_{1}$};
\end{scope}

\begin{scope}[xshift=\s cm] 

 \node[align=center, text width=1.6cm,,,text opacity = 0] (c1) {\href{https://www.geogebra.org/m/bn4q2mms}{a b c d e f g h i j k l m n}};

\def\EvenXArray {{1,0,-1,0,3,-3,-3,3}}
\def\OddXArray {{sqrt(2),sqrt(2),-sqrt(2),-sqrt(2),4,-4,0,0}}

\def\EvenYArray {{0,1,0,-1,3,3,-3,-3}}
\def\OddYArray {{sqrt(2),sqrt(2),-sqrt(2),-sqrt(2),4,4,-2,-2}}

\pgfmathtruncatemacro{\dimX}{dim{\EvenXArray}-1} \dimX

\foreach \a in {0,...,\dimX} {

\draw[red]
	(\EvenXArray[\a]*\x, \EvenYArray[\a]*\x) -- (\OddXArray[\a]*\x, \OddYArray[\a]*\x)
;
}
\node (s1c) {} node[below=9mm]{\tiny $FB_c \mathcal{O}_{2}$};
\end{scope}

\begin{scope}[xshift=2*\s cm] 

 \node[align=center, text width=1.6cm,,,text opacity = 0] (c1) {\href{https://www.geogebra.org/m/ubvsmd6j}{a b c d e f g h i j k l m n}};

\def\EvenXArray {{1,0,-1,0,3,-3,-3,3}}
\def\OddXArray {{sqrt(2),sqrt(2),-sqrt(2),-sqrt(2),0,0,0,0}}

\def\EvenYArray {{0,1,0,-1,3,3,-3,-3}}
\def\OddYArray {{sqrt(2),sqrt(2),-sqrt(2),-sqrt(2),2,2,-2,-2}}

\pgfmathtruncatemacro{\dimX}{dim{\EvenXArray}-1} \dimX

\foreach \a in {0,...,\dimX} {

\draw[red]
	(\EvenXArray[\a]*\x, \EvenYArray[\a]*\x) -- (\OddXArray[\a]*\x, \OddYArray[\a]*\x)
;
}
\node (s1c) {} node[below=9mm]{\tiny $FB_c \mathcal{O}_{3}$};
\end{scope}

\begin{scope}[xshift=3*\s cm] 

 \node[align=center, text width=1.6cm,,,text opacity = 0] (c1) {\href{https://www.geogebra.org/m/rutfsyzb}{a b c d e f g h i j k l m n}};

\def\EvenXArray {{1,0,-1,0,3,-3,-3,3}}
\def\OddXArray {{sqrt(2),sqrt(2),0,0,4,-4,-4,4}}

\def\EvenYArray {{0,1,0,-1,3,3,-3,-3}}
\def\OddYArray {{sqrt(2),sqrt(2),0,0,4,4,-4,-4}}

\pgfmathtruncatemacro{\dimX}{dim{\EvenXArray}-1} \dimX

\foreach \a in {0,...,\dimX} {

\draw[red]
	(\EvenXArray[\a]*\x, \EvenYArray[\a]*\x) -- (\OddXArray[\a]*\x, \OddYArray[\a]*\x)
;
}
\node (s1c) {} node[below=9mm]{\tiny $FB_c \mathcal{O}_{4}$};
\end{scope}

\begin{scope}[xshift=4*\s cm] 

 \node[align=center, text width=1.6cm,,,text opacity = 0] (c1) {\href{https://www.geogebra.org/m/mpsfvbsh}{a b c d e f g h i j k l m n}};

\def\EvenXArray {{1,0,-1,0,3,-3,-3,3}}
\def\OddXArray {{sqrt(2),sqrt(2),0,0,4,-4,0,0}}

\def\EvenYArray {{0,1,0,-1,3,3,-3,-3}}
\def\OddYArray {{sqrt(2),sqrt(2),0,0,4,4,-2,-2}}

\pgfmathtruncatemacro{\dimX}{dim{\EvenXArray}-1} \dimX

\foreach \a in {0,...,\dimX} {

\draw[red]
	(\EvenXArray[\a]*\x, \EvenYArray[\a]*\x) -- (\OddXArray[\a]*\x, \OddYArray[\a]*\x)
;
}
\node (s1c) {} node[below=9mm]{\tiny $FB_c \mathcal{O}_{5}$};
\end{scope}

\end{scope}

\begin{scope}[yshift=-1*\s cm] 

\begin{scope}[] 

 \node[align=center, text width=1.6cm,,,text opacity = 0] (c1) {\href{https://www.geogebra.org/m/utzx5ymg}{a b c d e f g h i j k l m n}};

\def\EvenXArray {{1,0,-1,0,3,-3,-3,3}}
\def\OddXArray {{sqrt(2),sqrt(2),0,0,0,0,-4,4}}

\def\EvenYArray {{0,1,0,-1,3,3,-3,-3}}
\def\OddYArray {{sqrt(2),sqrt(2),0,0,2,2,-4,-4}}

\pgfmathtruncatemacro{\dimX}{dim{\EvenXArray}-1} \dimX

\foreach \a in {0,...,\dimX} {

\draw[red]
	(\EvenXArray[\a]*\x, \EvenYArray[\a]*\x) -- (\OddXArray[\a]*\x, \OddYArray[\a]*\x)
;
}
\node (s1c) {} node[below=9mm]{\tiny $FB_c \mathcal{O}_{6}$};
\end{scope}

\begin{scope}[xshift=\s cm] 

 \node[align=center, text width=1.6cm,,,text opacity = 0] (c1) {\href{https://www.geogebra.org/m/m4uhkqnx}{a b c d e f g h i j k l m n}};

\def\EvenXArray {{1,0,-1,0,3,-3,-3,3}}
\def\OddXArray {{0,0,0,0,4,-4,-4,4}}

\def\EvenYArray {{0,1,0,-1,3,3,-3,-3}}
\def\OddYArray {{0,0,0,0,4,4,-4,-4}}

\pgfmathtruncatemacro{\dimX}{dim{\EvenXArray}-1} \dimX

\foreach \a in {0,...,\dimX} {

\draw[red]
	(\EvenXArray[\a]*\x, \EvenYArray[\a]*\x) -- (\OddXArray[\a]*\x, \OddYArray[\a]*\x)
;
}
\node (s1c) {} node[below=9mm]{\tiny $FB_c \mathcal{O}_{7}$};
\end{scope}

\begin{scope}[xshift=2*\s cm] 

 \node[align=center, text width=1.6cm,,,text opacity = 0] (c1) {\href{https://www.geogebra.org/m/a6zrmtmk}{a b c d e f g h i j k l m n}};

\def\EvenXArray {{1,0,-1,0,3,-3,-3,3,3,3}}
\def\OddXArray {{sqrt(2),sqrt(2),-sqrt(2),-sqrt(2),2,-4,0,0,2,4}}

\def\EvenYArray {{0,1,0,-1,3,3,-3,-3,-3,-3}}
\def\OddYArray {{sqrt(2),sqrt(2),-sqrt(2),-sqrt(2),0,4,-2,-2,0,-4}}

\pgfmathtruncatemacro{\dimX}{dim{\EvenXArray}-1} \dimX

\foreach \a in {0,...,\dimX} {

\draw[red]
	(\EvenXArray[\a]*\x, \EvenYArray[\a]*\x) -- (\OddXArray[\a]*\x, \OddYArray[\a]*\x)
;
}
\node (s1c) {} node[below=9mm]{\tiny $FB_c \mathcal{O}_{8}$};
\end{scope}

\begin{scope}[xshift=3*\s cm] 

 \node[align=center, text width=1.6cm,,,text opacity = 0] (c1) {\href{https://www.geogebra.org/m/bhvcvvuj}{a b c d e f g h i j k l m n}};

\def\EvenXArray {{1,0,-1,0,3,-3,-3,3,3,3}}
\def\OddXArray {{sqrt(2),sqrt(2),0,0,2,-4,0,0,2,4}}

\def\EvenYArray {{0,1,0,-1,3,3,-3,-3,-3,-3}}
\def\OddYArray {{sqrt(2),sqrt(2),0,0,0,4,-2,-2,0,-4}}

\pgfmathtruncatemacro{\dimX}{dim{\EvenXArray}-1} \dimX

\foreach \a in {0,...,\dimX} {

\draw[red]
	(\EvenXArray[\a]*\x, \EvenYArray[\a]*\x) -- (\OddXArray[\a]*\x, \OddYArray[\a]*\x)
;
}
\node (s1c) {} node[below=9mm]{\tiny $FB_c \mathcal{O}_{9}$};
\end{scope}

\begin{scope}[xshift=4*\s cm] 

 \node[align=center, text width=1.6cm,,,text opacity = 0] (c1) {\href{https://www.geogebra.org/m/ztgwep4e}{a b c d e f g h i j k l m n}};

\def\EvenXArray {{1,0,-1,0,3,-3,-3,3,3,3}}
\def\OddXArray {{sqrt(2),-sqrt(2),-sqrt(2),sqrt(2),2,-4,0,0,2,4}}

\def\EvenYArray {{0,1,0,-1,3,3,-3,-3,-3,-3}}
\def\OddYArray {{-sqrt(2),sqrt(2),sqrt(2),-sqrt(2),0,4,-2,-2,0,-4}}

\pgfmathtruncatemacro{\dimX}{dim{\EvenXArray}-1} \dimX

\foreach \a in {0,...,\dimX} {

\draw[red]
	(\EvenXArray[\a]*\x, \EvenYArray[\a]*\x) -- (\OddXArray[\a]*\x, \OddYArray[\a]*\x)
;
}
\node (s1c) {} node[below=9mm]{\tiny $FB_c \mathcal{O}_{10}$};
\end{scope}

\end{scope}

\begin{scope}[yshift=-2*\s cm]

\begin{scope}[] 

 \node[align=center, text width=1.6cm,,,text opacity = 0] (c1) {\href{https://www.geogebra.org/m/em6qtbek}{a b c d e f g h i j k l m n}};

\def\EvenXArray {{1,0,-1,0,3,-3,-3,3,3,3}}
\def\OddXArray {{sqrt(2),0,0,sqrt(2),2,-4,0,0,2,4}}

\def\EvenYArray {{0,1,0,-1,3,3,-3,-3,-3,-3}}
\def\OddYArray {{-sqrt(2),0,0,-sqrt(2),0,4,-2,-2,0,-4}}

\pgfmathtruncatemacro{\dimX}{dim{\EvenXArray}-1} \dimX

\foreach \a in {0,...,\dimX} {

\draw[red]
	(\EvenXArray[\a]*\x, \EvenYArray[\a]*\x) -- (\OddXArray[\a]*\x, \OddYArray[\a]*\x)
;
}
\node (s1c) {} node[below=9mm]{\tiny $FB_c \mathcal{O}_{11}$};
\end{scope}

\begin{scope}[xshift=\s cm] 

 \node[align=center, text width=1.6cm,,,text opacity = 0] (c1) {\href{https://www.geogebra.org/m/qh6acgmp}{a b c d e f g h i j k l m n}};

\def\EvenXArray {{1,0,-1,0,3,-3,-3,3,3,3}}
\def\OddXArray {{0,-sqrt(2),-sqrt(2),0,2,-4,0,0,2,4}}

\def\EvenYArray {{0,1,0,-1,3,3,-3,-3,-3,-3}}
\def\OddYArray {{0,sqrt(2),sqrt(2),0,0,4,-2,-2,0,-4}}

\pgfmathtruncatemacro{\dimX}{dim{\EvenXArray}-1} \dimX

\foreach \a in {0,...,\dimX} {

\draw[red]
	(\EvenXArray[\a]*\x, \EvenYArray[\a]*\x) -- (\OddXArray[\a]*\x, \OddYArray[\a]*\x)
;
}
\node (s1c) {} node[below=9mm]{\tiny $FB_c \mathcal{O}_{12}$};
\end{scope}

\begin{scope}[xshift=2*\s cm] 

 \node[align=center, text width=1.6cm,,,text opacity = 0] (c1) {\href{https://www.geogebra.org/m/bsdbssfw}{a b c d e f g h i j k l m n}};

\def\EvenXArray {{1,0,-1,0,3,-3,-3,3,3,3}}
\def\OddXArray {{0,0,0,0,2,-4,0,0,2,4}}

\def\EvenYArray {{0,1,0,-1,3,3,-3,-3,-3,-3}}
\def\OddYArray {{0,0,0,0,0,4,-2,-2,0,-4}}

\pgfmathtruncatemacro{\dimX}{dim{\EvenXArray}-1} \dimX

\foreach \a in {0,...,\dimX} {

\draw[red]
	(\EvenXArray[\a]*\x, \EvenYArray[\a]*\x) -- (\OddXArray[\a]*\x, \OddYArray[\a]*\x)
;
}
\node (s1c) {} node[below=9mm]{\tiny $FB_c \mathcal{O}_{13}$};
\end{scope}

\begin{scope}[xshift=3*\s cm] 

 \node[align=center, text width=1.6cm,,,text opacity = 0] (c1) {\href{https://www.geogebra.org/m/znfggsg3}{a b c d e f g h i j k l m n}};

\def\EvenXArray {{1,0,-1,0,3,-3,-3,-3,-3,3,3,3}}
\def\OddXArray {{sqrt(2),sqrt(2),-sqrt(2),-sqrt(2),2,-2,-2,0,-4,0,2,4}}

\def\EvenYArray {{0,1,0,-1,3,3,-3,-3,-3,-3,-3,-3}}
\def\OddYArray {{sqrt(2),sqrt(2),-sqrt(2),-sqrt(2),0,0,0,-2,-4,-2,0,-4}}

\pgfmathtruncatemacro{\dimX}{dim{\EvenXArray}-1} \dimX

\foreach \a in {0,...,\dimX} {

\draw[red]
	(\EvenXArray[\a]*\x, \EvenYArray[\a]*\x) -- (\OddXArray[\a]*\x, \OddYArray[\a]*\x)
;
}
\node (s1c) {} node[below=9mm]{\tiny $FB_c \mathcal{O}_{14}$};
\end{scope}

\begin{scope}[xshift=4*\s cm] 

 \node[align=center, text width=1.6cm,,,text opacity = 0] (c1) {\href{https://www.geogebra.org/m/dgeyzc4e}{a b c d e f g h i j k l m n}};

\def\EvenXArray {{1,0,-1,0,3,-3,-3,-3,-3,3,3,3}}
\def\OddXArray {{sqrt(2),sqrt(2),0,0,2,-2,-2,0,-4,0,2,4}}

\def\EvenYArray {{0,1,0,-1,3,3,-3,-3,-3,-3,-3,-3}}
\def\OddYArray {{sqrt(2),sqrt(2),0,0,0,0,0,-2,-4,-2,0,-4}}

\pgfmathtruncatemacro{\dimX}{dim{\EvenXArray}-1} \dimX

\foreach \a in {0,...,\dimX} {

\draw[red]
	(\EvenXArray[\a]*\x, \EvenYArray[\a]*\x) -- (\OddXArray[\a]*\x, \OddYArray[\a]*\x)
;
}
\node (s1c) {} node[below=9mm]{\tiny $FB_c \mathcal{O}_{15}$};
\end{scope}

\end{scope}

\begin{scope}[yshift=-3*\s cm] 

\begin{scope}[] 

 \node[align=center, text width=1.6cm,,,text opacity = 0] (c1) {\href{https://www.geogebra.org/m/jmpbuppx}{a b c d e f g h i j k l m n}};

\def\EvenXArray {{1,0,-1,0,3,-3,-3,-3,-3,3,3,3}}
\def\OddXArray {{sqrt(2),0,0,sqrt(2),2,-2,-2,0,-4,0,2,4}}

\def\EvenYArray {{0,1,0,-1,3,3,-3,-3,-3,-3,-3,-3}}
\def\OddYArray {{-sqrt(2),0,0,-sqrt(2),0,0,0,-2,-4,-2,0,-4}}

\pgfmathtruncatemacro{\dimX}{dim{\EvenXArray}-1} \dimX

\foreach \a in {0,...,\dimX} {

\draw[red]
	(\EvenXArray[\a]*\x, \EvenYArray[\a]*\x) -- (\OddXArray[\a]*\x, \OddYArray[\a]*\x)
;
}
\node (s1c) {} node[below=9mm]{\tiny $FB_c \mathcal{O}_{16}$};
\end{scope}

\begin{scope}[xshift=\s cm] 

 \node[align=center, text width=1.6cm,,,text opacity = 0] (c1) {\href{https://www.geogebra.org/m/mrzfdcbt}{a b c d e f g h i j k l m n}};

\def\EvenXArray {{1,0,-1,0,3,-3,-3,-3,-3,3,3,3}}
\def\OddXArray {{0,0,0,0,2,-2,-2,0,-4,0,2,4}}

\def\EvenYArray {{0,1,0,-1,3,3,-3,-3,-3,-3,-3,-3}}
\def\OddYArray {{0,0,0,0,0,0,0,-2,-4,-2,0,-4}}

\pgfmathtruncatemacro{\dimX}{dim{\EvenXArray}-1} \dimX

\foreach \a in {0,...,\dimX} {

\draw[red]
	(\EvenXArray[\a]*\x, \EvenYArray[\a]*\x) -- (\OddXArray[\a]*\x, \OddYArray[\a]*\x)
;
}
\node (s1c) {} node[below=9mm]{\tiny $FB_c \mathcal{O}_{17}$};
\end{scope}

\begin{scope}[xshift=2*\s cm] 

 \node[align=center, text width=1.6cm,,,text opacity = 0] (c1) {\href{https://www.geogebra.org/m/ngqfxjvb}{a b c d e f g h i j k l m n}};

\def\EvenXArray {{1,0,-1,0,3,3,3,-3,-3,-3,-3,-3,-3,3,3,3}}
\def\OddXArray {{sqrt(2),sqrt(2),-sqrt(2),-sqrt(2),2,0,4,0,-2,-4,-2,0,-4,0,2,4}}

\def\EvenYArray {{0,1,0,-1,3,3,3,3,3,3,-3,-3,-3,-3,-3,-3}}
\def\OddYArray {{sqrt(2),sqrt(2),-sqrt(2),-sqrt(2),0,2,4,2,0,4,0,-2,-4,-2,0,-4}}

\pgfmathtruncatemacro{\dimX}{dim{\EvenXArray}-1} \dimX

\foreach \a in {0,...,\dimX} {

\draw[red]
	(\EvenXArray[\a]*\x, \EvenYArray[\a]*\x) -- (\OddXArray[\a]*\x, \OddYArray[\a]*\x)
;
}
\node (s1c) {} node[below=9mm]{\tiny $FB_c \mathcal{O}_{18}$};
\end{scope}

\begin{scope}[xshift=3*\s cm] 

 \node[align=center, text width=1.6cm,,,text opacity = 0] (c1) {\href{https://www.geogebra.org/m/zwwydqfc}{a b c d e f g h i j k l m n}};

\def\EvenXArray {{1,0,-1,0,3,3,3,-3,-3,-3,-3,-3,-3,3,3,3}}
\def\OddXArray {{sqrt(2),sqrt(2),0,0,2,0,4,0,-2,-4,-2,0,-4,0,2,4}}

\def\EvenYArray {{0,1,0,-1,3,3,3,3,3,3,-3,-3,-3,-3,-3,-3}}
\def\OddYArray {{sqrt(2),sqrt(2),0,0,0,2,4,2,0,4,0,-2,-4,-2,0,-4}}

\pgfmathtruncatemacro{\dimX}{dim{\EvenXArray}-1} \dimX

\foreach \a in {0,...,\dimX} {

\draw[red]
	(\EvenXArray[\a]*\x, \EvenYArray[\a]*\x) -- (\OddXArray[\a]*\x, \OddYArray[\a]*\x)
;
}
\node (s1c) {} node[below=9mm]{\tiny $FB_c \mathcal{O}_{19}$};
\end{scope}

\begin{scope}[xshift=4*\s cm] 

 \node[align=center, text width=1.6cm,,,text opacity = 0] (c1) {\href{https://www.geogebra.org/m/ranjp4nu}{a b c d e f g h i j k l m n}};

\def\EvenXArray {{1,0,-1,0,3,3,3,-3,-3,-3,-3,-3,-3,3,3,3}}
\def\OddXArray {{0,0,0,0,2,0,4,0,-2,-4,-2,0,-4,0,2,4}}

\def\EvenYArray {{0,1,0,-1,3,3,3,3,3,3,-3,-3,-3,-3,-3,-3}}
\def\OddYArray {{0,0,0,0,0,2,4,2,0,4,0,-2,-4,-2,0,-4}}

\pgfmathtruncatemacro{\dimX}{dim{\EvenXArray}-1} \dimX

\foreach \a in {0,...,\dimX} {

\draw[red]
	(\EvenXArray[\a]*\x, \EvenYArray[\a]*\x) -- (\OddXArray[\a]*\x, \OddYArray[\a]*\x)
;
}
\node (s1c) {} node[below=9mm]{\tiny $FB_c \mathcal{O}_{20}$};
\end{scope}

\end{scope}

\begin{scope}[yshift=-4*\s cm]

\begin{scope}[] 

 \node[align=center, text width=1.6cm,,,text opacity = 0] (c1) {\href{https://www.geogebra.org/m/kkeuz8jp}{a b c d e f g h i j k l m n}};

\def\EvenXArray {{1,0,-1,0,0,0,3,-3,-3,3,3,3}}
\def\OddXArray {{sqrt(2),0,-sqrt(2),0,-sqrt(2),sqrt(2),2,-4,0,0,2,4}}

\def\EvenYArray {{0,1,0,-1,-1,-1,3,3,-3,-3,-3,-3}}
\def\OddYArray {{-sqrt(2),0,-sqrt(2),0,-sqrt(2),-sqrt(2),0,4,-2,-2,0,-4}}

\pgfmathtruncatemacro{\dimX}{dim{\EvenXArray}-1} \dimX

\foreach \a in {0,...,\dimX} {

\draw[red]
	(\EvenXArray[\a]*\x, \EvenYArray[\a]*\x) -- (\OddXArray[\a]*\x, \OddYArray[\a]*\x)
;
}
\node (s1c) {} node[below=9mm]{\tiny $FB_c \mathcal{O}_{21}$};
\end{scope}

\begin{scope}[xshift=\s cm] 

 \node[align=center, text width=1.6cm,,,text opacity = 0] (c1) {\href{https://www.geogebra.org/m/s4prb6ev}{a b c d e f g h i j k l m n}};

\def\EvenXArray {{1,0,-1,0,0,0,3,-3,-3,-3,-3,3,3,3}}
\def\OddXArray {{sqrt(2),0,-sqrt(2),0,-sqrt(2),sqrt(2),2,-2,-2,0,-4,0,2,4}}

\def\EvenYArray {{0,1,0,-1,-1,-1,3,3,-3,-3,-3,-3,-3,-3}}
\def\OddYArray {{-sqrt(2),0,-sqrt(2),0,-sqrt(2),-sqrt(2),0,0,0,-2,-4,-2,0,-4}}

\pgfmathtruncatemacro{\dimX}{dim{\EvenXArray}-1} \dimX

\foreach \a in {0,...,\dimX} {

\draw[red]
	(\EvenXArray[\a]*\x, \EvenYArray[\a]*\x) -- (\OddXArray[\a]*\x, \OddYArray[\a]*\x)
;
}
\node (s1c) {} node[below=9mm]{\tiny $FB_c \mathcal{O}_{22}$};
\end{scope}

\begin{scope}[xshift=2*\s cm] 

 \node[align=center, text width=1.6cm,,,text opacity = 0] (c1) {\href{https://www.geogebra.org/m/ftvwa5ex}{a b c d e f g h i j k l m n}};

\def\EvenXArray {{1,0,-1,0,0,0,3,-3,-3,-3,-3,3}}
\def\OddXArray {{sqrt(2),0,-sqrt(2),0,-sqrt(2),sqrt(2),0,0,-2,-4,-2,4}}

\def\EvenYArray {{0,1,0,-1,-1,-1,3,3,3,3,-3,-3}}
\def\OddYArray {{-sqrt(2),0,-sqrt(2),0,-sqrt(2),-sqrt(2),2,2,0,4,0,-4}}

\pgfmathtruncatemacro{\dimX}{dim{\EvenXArray}-1} \dimX

\foreach \a in {0,...,\dimX} {

\draw[red]
	(\EvenXArray[\a]*\x, \EvenYArray[\a]*\x) -- (\OddXArray[\a]*\x, \OddYArray[\a]*\x)
;
}
\node (s1c) {} node[below=9mm]{\tiny $FB_c \mathcal{O}_{23}$};
\end{scope}

\begin{scope}[xshift=3*\s cm] 

 \node[align=center, text width=1.6cm,,,text opacity = 0] (c1) {\href{https://www.geogebra.org/m/ddmedemu}{a b c d e f g h i j k l m n}};

\def\EvenXArray {{1,0,-1,0,0,0,3,-3,-3,-3,-3,-3,-3,3}}
\def\OddXArray {{sqrt(2),0,-sqrt(2),0,-sqrt(2),sqrt(2),0,0,-2,-4,-2,0,-4,0}}

\def\EvenYArray {{0,1,0,-1,-1,-1,3,3,3,3,-3,-3,-3,-3}}
\def\OddYArray {{-sqrt(2),0,-sqrt(2),0,-sqrt(2),-sqrt(2),2,2,0,4,0,-2,-4,-2}}

\pgfmathtruncatemacro{\dimX}{dim{\EvenXArray}-1} \dimX

\foreach \a in {0,...,\dimX} {

\draw[red]
	(\EvenXArray[\a]*\x, \EvenYArray[\a]*\x) -- (\OddXArray[\a]*\x, \OddYArray[\a]*\x)
;
}
\node (s1c) {} node[below=9mm]{\tiny $FB_c \mathcal{O}_{24}$};
\end{scope}

\begin{scope}[xshift=4*\s cm] 

 \node[align=center, text width=1.6cm,,,text opacity = 0] (c1) {\href{https://www.geogebra.org/m/sqbpfdzu}{a b c d e f g h i j k l m n}};

\def\EvenXArray {{1,0,-1,0,0,0,3,3,3,-3,-3,-3,-3,3}}
\def\OddXArray {{sqrt(2),0,-sqrt(2),0,-sqrt(2),sqrt(2),2,0,4,0,-2,-4,-2,2}}

\def\EvenYArray {{0,1,0,-1,-1,-1,3,3,3,3,3,3,-3,-3}}
\def\OddYArray {{-sqrt(2),0,-sqrt(2),0,-sqrt(2),-sqrt(2),0,2,4,2,0,4,0,0}}

\pgfmathtruncatemacro{\dimX}{dim{\EvenXArray}-1} \dimX

\foreach \a in {0,...,\dimX} {

\draw[red]
	(\EvenXArray[\a]*\x, \EvenYArray[\a]*\x) -- (\OddXArray[\a]*\x, \OddYArray[\a]*\x)
;
}
\node (s1c) {} node[below=9mm]{\tiny $FB_c \mathcal{O}_{25}$};
\end{scope}

\end{scope}

\begin{scope}[yshift=-5*\s cm] 

\begin{scope}[] 

 \node[align=center, text width=1.6cm,,,text opacity = 0] (c1) {\href{https://www.geogebra.org/m/wqdtkkkf}{a b c d e f g h i j k l m n}};

\def\EvenXArray {{1,0,-1,0,0,0,3,3,3,-3,-3,-3,-3,-3,-3,3,3,3}}
\def\OddXArray {{sqrt(2),0,-sqrt(2),0,-sqrt(2),sqrt(2),2,0,4,0,-2,-4,-2,0,-4,0,2,4}}

\def\EvenYArray {{0,1,0,-1,-1,-1,3,3,3,3,3,3,-3,-3,-3,-3,-3,-3}}
\def\OddYArray {{-sqrt(2),0,-sqrt(2),0,-sqrt(2),-sqrt(2),0,2,4,2,0,4,0,-2,-4,-2,0,-4}}

\pgfmathtruncatemacro{\dimX}{dim{\EvenXArray}-1} \dimX

\foreach \a in {0,...,\dimX} {

\draw[red]
	(\EvenXArray[\a]*\x, \EvenYArray[\a]*\x) -- (\OddXArray[\a]*\x, \OddYArray[\a]*\x)
;
}
\node (s1c) {} node[below=9mm]{\tiny $FB_c \mathcal{O}_{26}$};
\end{scope}

\begin{scope}[xshift=\s cm] 

 \node[align=center, text width=1.6cm,,,text opacity = 0] (c1) {\href{https://www.geogebra.org/m/srvq8kxc}{a b c d e f g h i j k l m n}};

\def\EvenXArray {{1,0,-1,-1,-1,0,0,0,3,-3,-3,-3,-3,3,3,3}}
\def\OddXArray {{sqrt(2),-sqrt(2),0,-sqrt(2),-sqrt(2),0,-sqrt(2),sqrt(2),2,-2,-2,0,-4,0,2,4}}

\def\EvenYArray {{0,1,0,0,0,-1,-1,-1,3,3,-3,-3,-3,-3,-3,-3}}
\def\OddYArray {{-sqrt(2),sqrt(2),0,sqrt(2),-sqrt(2),0,-sqrt(2),-sqrt(2),0,0,0,-2,-4,-2,0,-4}}

\pgfmathtruncatemacro{\dimX}{dim{\EvenXArray}-1} \dimX

\foreach \a in {0,...,\dimX} {

\draw[red]
	(\EvenXArray[\a]*\x, \EvenYArray[\a]*\x) -- (\OddXArray[\a]*\x, \OddYArray[\a]*\x)
;
}
\node (s1c) {} node[below=9mm]{\tiny $FB_c \mathcal{O}_{27}$};
\end{scope}

\begin{scope}[xshift=2*\s cm] 

 \node[align=center, text width=1.6cm,,,text opacity = 0] (c1) {\href{https://www.geogebra.org/m/tw9nv38t}{a b c d e f g h i j k l m n}};

\def\EvenXArray {{1,0,-1,-1,-1,0,0,0,3,3,3,-3,-3,3,3,3}}
\def\OddXArray {{sqrt(2),-sqrt(2),0,-sqrt(2),-sqrt(2),0,-sqrt(2),sqrt(2),2,0,4,0,0,0,2,4}}

\def\EvenYArray {{0,1,0,0,0,-1,-1,-1,3,3,3,3,-3,-3,-3,-3}}
\def\OddYArray {{-sqrt(2),sqrt(2),0,sqrt(2),-sqrt(2),0,-sqrt(2),-sqrt(2),0,2,4,2,-2,-2,0,-4}}

\pgfmathtruncatemacro{\dimX}{dim{\EvenXArray}-1} \dimX

\foreach \a in {0,...,\dimX} {

\draw[red]
	(\EvenXArray[\a]*\x, \EvenYArray[\a]*\x) -- (\OddXArray[\a]*\x, \OddYArray[\a]*\x)
;
}
\node (s1c) {} node[below=9mm]{\tiny $FB_c \mathcal{O}_{28}$};
\end{scope}

\begin{scope}[xshift=3*\s cm] 

 \node[align=center, text width=1.6cm,,,text opacity = 0] (c1) {\href{https://www.geogebra.org/m/rywjw9tp}{a b c d e f g h i j k l m n}};

\def\EvenXArray {{1,0,-1,-1,-1,0,0,0,3,3,3,-3,-3,-3,-3,-3,-3,3,3,3}}
\def\OddXArray {{sqrt(2),-sqrt(2),0,-sqrt(2),-sqrt(2),0,-sqrt(2),sqrt(2),2,0,4,0,-2,-4,-2,0,-4,0,2,4}}

\def\EvenYArray {{0,1,0,0,0,-1,-1,-1,3,3,3,3,3,3,-3,-3,-3,-3,-3,-3}}
\def\OddYArray {{-sqrt(2),sqrt(2),0,sqrt(2),-sqrt(2),0,-sqrt(2),-sqrt(2),0,2,4,2,0,4,0,-2,-4,-2,0,-4}}

\pgfmathtruncatemacro{\dimX}{dim{\EvenXArray}-1} \dimX

\foreach \a in {0,...,\dimX} {

\draw[red]
	(\EvenXArray[\a]*\x, \EvenYArray[\a]*\x) -- (\OddXArray[\a]*\x, \OddYArray[\a]*\x)
;
}
\node (s1c) {} node[below=9mm]{\tiny $FB_c \mathcal{O}_{29}$};
\end{scope}

\begin{scope}[xshift=4*\s cm] 






\end{scope}

\end{scope}


\foreach \xs in {0,...,4} {

\foreach \ys in {0,...,4} {

\tikzset{xshift=\xs*\s cm, yshift=-\ys*\s cm}

\foreach \a in {0,1,2,3} {

\tikzset{rotate=\a*\th}
	
\draw[gray!75, double, line width=0.5]
	(90:\x) -- (90+\th:\x)
	(\x,0) -- (2*\x,0)
	(\l*\x,\l*\x) -- (-\l*\x,\l*\x)
	(45:2*\x) -- (\l*\x,\l*\x) 
;

}

\foreach \a in {0,1,2,3} {

\tikzset{rotate=\a*\th}

\draw[blue!80!black!50]
	(0,0) -- (45:2*\x) 
	(2*\x,0) -- (4*\x,0)
;

\filldraw[blue!80!black] 
	(2*\x,0) circle (\rr)
	(45:2*\x) circle (\rr)
;
}

\filldraw[blue!80!black] 
	(0,0) circle (\rr)
;

\draw[blue!80!black] (0,0) circle (2*\x);

}
}

\foreach \xs in {0,...,3} {

\foreach \ys in {5} {

\tikzset{xshift=\xs*\s cm, yshift=-\ys*\s cm}

\foreach \a in {0,1,2,3} {

\tikzset{rotate=\a*\th}
	
\draw[gray!75, double, line width=0.5]
	(90:\x) -- (90+\th:\x)
	(\x,0) -- (2*\x,0)
	(\l*\x,\l*\x) -- (-\l*\x,\l*\x)
	(45:2*\x) -- (\l*\x,\l*\x) 
;

}

\foreach \a in {0,1,2,3} {

\tikzset{rotate=\a*\th}

\draw[blue!80!black!50]
	(0,0) -- (45:2*\x) 
	(2*\x,0) -- (4*\x,0)
;

\filldraw[blue!80!black] 
	(2*\x,0) circle (\rr)
	(45:2*\x) circle (\rr)
;
}

\filldraw[blue!80!black] 
	(0,0) circle (\rr)
;

\draw[blue!80!black] (0,0) circle (2*\x);

}
}

\end{tikzpicture}
\caption{The $c$-edge assignments in $FB\mathcal{O}$ up to isomorphism}
\label{Fig:X-Y-octants-ne2e}
\end{figure}


\begin{figure}[h!] 
\centering
\begin{tikzpicture}

\tikzmath{
\s=2.4;
\x=0.25;
\l=3;
\th=360/4;
\rr=0.15*\x;
}

\begin{scope}[]

\begin{scope}[]

 \node[align=center, text width=1.6cm,,,text opacity = 0] (c1) {\href{https://www.geogebra.org/m/gcyvwc7f}{a b c d e f g h i j k l m n}};

\def\IndexArrayII {{4}}

\def\IndexArrayIII {{1,4}}

\pgfmathtruncatemacro{\dimL}{dim{\IndexArrayII}-1} \dimL

\foreach \a in {0,...,\dimL} {

\pgfmathsetmacro{\A}{\IndexArrayII[\a]-1} 

\tikzset{rotate=\A*\th}

\draw[blue!80!black]
	(2*\x,0) arc (0:90:2*\x)
;

}

\pgfmathtruncatemacro{\dimL}{dim{\IndexArrayIII}-1} \dimL

\foreach \a in {0,...,\dimL} {

\pgfmathsetmacro{\A}{\IndexArrayIII[\a]-1} 

\tikzset{rotate=\A*\th}

\draw[blue!80!black]
	(2*\x,0) -- (4*\x,0)
;

}
\node (s1c) {} node[below=9mm]{\tiny $B_b \mathcal{O}_{1}$};
\end{scope}

\begin{scope}[xshift=\s cm] 

 \node[align=center, text width=1.6cm,,,text opacity = 0] (c1) {\href{https://www.geogebra.org/m/jyhyh2yk}{a b c d e f g h i j k l m n}};

\def\IndexArrayII {{3,4}}

\def\IndexArrayIII {{1,3}}

\pgfmathtruncatemacro{\dimL}{dim{\IndexArrayII}-1} \dimL

\foreach \a in {0,...,\dimL} {

\pgfmathsetmacro{\A}{\IndexArrayII[\a]-1} 

\tikzset{rotate=\A*\th}

\draw[blue!80!black]
	(2*\x,0) arc (0:90:2*\x)
;

}

\pgfmathtruncatemacro{\dimL}{dim{\IndexArrayIII}-1} \dimL

\foreach \a in {0,...,\dimL} {

\pgfmathsetmacro{\A}{\IndexArrayIII[\a]-1} 

\tikzset{rotate=\A*\th}

\draw[blue!80!black]
	(2*\x,0) -- (4*\x,0)
;
}
\node (s1c) {} node[below=9mm]{\tiny $B_b \mathcal{O}_{2}$};
\end{scope}

\begin{scope}[xshift=2*\s cm] 

 \node[align=center, text width=1.6cm,,,text opacity = 0] (c1) {\href{https://www.geogebra.org/m/jtju5meq}{a b c d e f g h i j k l m n}};

\def\IndexArrayII {{2,4}}

\def\IndexArrayIII {{1,2,3,4}}

\pgfmathtruncatemacro{\dimL}{dim{\IndexArrayII}-1} \dimL

\foreach \a in {0,...,\dimL} {

\pgfmathsetmacro{\A}{\IndexArrayII[\a]-1} 

\tikzset{rotate=\A*\th}

\draw[blue!80!black]
	(2*\x,0) arc (0:90:2*\x)
;

}

\pgfmathtruncatemacro{\dimL}{dim{\IndexArrayIII}-1} \dimL

\foreach \a in {0,...,\dimL} {

\pgfmathsetmacro{\A}{\IndexArrayIII[\a]-1} 

\tikzset{rotate=\A*\th}

\draw[blue!80!black]
	(2*\x,0) -- (4*\x,0)
;
}
\node (s1c) {} node[below=9mm]{\tiny $B_b \mathcal{O}_{3}$};
\end{scope}

\begin{scope}[xshift=3*\s cm] 

 \node[align=center, text width=1.6cm,,,text opacity = 0] (c1) {\href{https://www.geogebra.org/m/d2bbpg8e}{a b c d e f g h i j k l m n}};

\def\IndexArrayII {{2,3,4}}

\def\IndexArrayIII {{1,2}}

\pgfmathtruncatemacro{\dimL}{dim{\IndexArrayII}-1} \dimL

\foreach \a in {0,...,\dimL} {

\pgfmathsetmacro{\A}{\IndexArrayII[\a]-1} 

\tikzset{rotate=\A*\th}

\draw[blue!80!black]
	(2*\x,0) arc (0:90:2*\x)
;
}

\pgfmathtruncatemacro{\dimL}{dim{\IndexArrayIII}-1} \dimL

\foreach \a in {0,...,\dimL} {

\pgfmathsetmacro{\A}{\IndexArrayIII[\a]-1} 

\tikzset{rotate=\A*\th}

\draw[blue!80!black]
	(2*\x,0) -- (4*\x,0)
;
}
\node (s1c) {} node[below=9mm]{\tiny $B_b \mathcal{O}_{4}$};
\end{scope}

\begin{scope}[xshift=4*\s cm] 

 \node[align=center, text width=1.6cm,,,text opacity = 0] (c1) {\href{https://www.geogebra.org/m/g8zqnnur}{a b c d e f g h i j k l m n}};

\def\IndexArrayII {{1,2,3,4}}

\pgfmathtruncatemacro{\dimL}{dim{\IndexArrayII}-1} \dimL

\foreach \a in {0,...,\dimL} {

\pgfmathsetmacro{\A}{\IndexArrayII[\a]-1} 

\tikzset{rotate=\A*\th}

\draw[blue!80!black]
	(2*\x,0) arc (0:90:2*\x)
;
}
\node (s1c) {} node[below=9mm]{\tiny $B_b \mathcal{O}_{5}$};
\end{scope}

\end{scope}

\begin{scope}[yshift=-1*\s cm] 

\begin{scope}[] 

 \node[align=center, text width=1.6cm,,,text opacity = 0] (c1) {\href{https://www.geogebra.org/m/uvqm8ub9}{a b c d e f g h i j k l m n}};

\def\IndexArrayI {{3,4}}

\def\IndexArrayII {{2,4}}

\def\IndexArrayIII {{1,2}}

\pgfmathtruncatemacro{\dimL}{dim{\IndexArrayI}-1} \dimL

\foreach \a in {0,...,\dimL} {

\pgfmathsetmacro{\A}{\IndexArrayI[\a]-1} 

\tikzset{rotate=\A*\th}

\draw[blue!80!black]
	(0,0) -- (2*\x,0) 
;
}
\pgfmathtruncatemacro{\dimL}{dim{\IndexArrayII}-1} \dimL

\foreach \a in {0,...,\dimL} {

\pgfmathsetmacro{\A}{\IndexArrayII[\a]-1} 

\tikzset{rotate=\A*\th}

\draw[blue!80!black]
	(2*\x,0) arc (0:90:2*\x)
;
}

\pgfmathtruncatemacro{\dimL}{dim{\IndexArrayIII}-1} \dimL

\foreach \a in {0,...,\dimL} {

\pgfmathsetmacro{\A}{\IndexArrayIII[\a]-1} 

\tikzset{rotate=\A*\th}

\draw[blue!80!black]
	(2*\x,0) -- (4*\x,0)
;
}
\node (s1c) {} node[below=9mm]{\tiny $B_b \mathcal{O}_{6}$};
\end{scope}

\begin{scope}[xshift=\s cm] 

 \node[align=center, text width=1.6cm,,,text opacity = 0] (c1) {\href{https://www.geogebra.org/m/cmgrzysw}{a b c d e f g h i j k l m n}};

\def\IndexArrayI {{3,4}}

\def\IndexArrayII {{2,3,4}}

\def\IndexArrayIII {{1,2,3,4}}

\pgfmathtruncatemacro{\dimL}{dim{\IndexArrayI}-1} \dimL

\foreach \a in {0,...,\dimL} {

\pgfmathsetmacro{\A}{\IndexArrayI[\a]-1} 

\tikzset{rotate=\A*\th}

\draw[blue!80!black]
	(0,0) -- (2*\x,0) 
;
}

\pgfmathtruncatemacro{\dimL}{dim{\IndexArrayII}-1} \dimL

\foreach \a in {0,...,\dimL} {

\pgfmathsetmacro{\A}{\IndexArrayII[\a]-1} 

\tikzset{rotate=\A*\th}

\draw[blue!80!black]
	(2*\x,0) arc (0:90:2*\x)
;
}

\pgfmathtruncatemacro{\dimL}{dim{\IndexArrayIII}-1} \dimL

\foreach \a in {0,...,\dimL} {

\pgfmathsetmacro{\A}{\IndexArrayIII[\a]-1} 

\tikzset{rotate=\A*\th}

\draw[blue!80!black]
	(2*\x,0) -- (4*\x,0)
;
}
\node (s1c) {} node[below=9mm]{\tiny $B_b \mathcal{O}_{7}$};
\end{scope}

\begin{scope}[xshift=2*\s cm] 

 \node[align=center, text width=1.6cm,,,text opacity = 0] (c1) {\href{https://www.geogebra.org/m/yuzh3spe}{a b c d e f g h i j k l m n}};

\def\IndexArrayI {{3,4}}

\def\IndexArrayII {{1}}

\def\IndexArrayIII {{1,2,3,4}}

\pgfmathtruncatemacro{\dimL}{dim{\IndexArrayI}-1} \dimL

\foreach \a in {0,...,\dimL} {

\pgfmathsetmacro{\A}{\IndexArrayI[\a]-1} 

\tikzset{rotate=\A*\th}

\draw[blue!80!black]
	(0,0) -- (2*\x,0) 
;
}

\pgfmathtruncatemacro{\dimL}{dim{\IndexArrayII}-1} \dimL

\foreach \a in {0,...,\dimL} {

\pgfmathsetmacro{\A}{\IndexArrayII[\a]-1} 

\tikzset{rotate=\A*\th}

\draw[blue!80!black]
	(2*\x,0) arc (0:90:2*\x)
;
}

\pgfmathtruncatemacro{\dimL}{dim{\IndexArrayIII}-1} \dimL

\foreach \a in {0,...,\dimL} {

\pgfmathsetmacro{\A}{\IndexArrayIII[\a]-1} 

\tikzset{rotate=\A*\th}

\draw[blue!80!black]
	(2*\x,0) -- (4*\x,0)
;
}
\node (s1c) {} node[below=9mm]{\tiny $B_b \mathcal{O}_{8}$};
\end{scope}

\begin{scope}[xshift=3*\s cm] 

 \node[align=center, text width=1.6cm,,,text opacity = 0] (c1) {\href{https://www.geogebra.org/m/tf2rnnct}{a b c d e f g h i j k l m n}};

\def\IndexArrayI {{3,4}}

\def\IndexArrayII {{1,4}}

\def\IndexArrayIII {{2,3}}

\pgfmathtruncatemacro{\dimL}{dim{\IndexArrayI}-1} \dimL

\foreach \a in {0,...,\dimL} {

\pgfmathsetmacro{\A}{\IndexArrayI[\a]-1} 

\tikzset{rotate=\A*\th}

\draw[blue!80!black]
	(0,0) -- (2*\x,0) 
;
}

\pgfmathtruncatemacro{\dimL}{dim{\IndexArrayII}-1} \dimL

\foreach \a in {0,...,\dimL} {

\pgfmathsetmacro{\A}{\IndexArrayII[\a]-1} 

\tikzset{rotate=\A*\th}

\draw[blue!80!black]
	(2*\x,0) arc (0:90:2*\x)
;
}

\pgfmathtruncatemacro{\dimL}{dim{\IndexArrayIII}-1} \dimL

\foreach \a in {0,...,\dimL} {

\pgfmathsetmacro{\A}{\IndexArrayIII[\a]-1} 

\tikzset{rotate=\A*\th}

\draw[blue!80!black]
	(2*\x,0) -- (4*\x,0)
;
}
\node (s1c) {} node[below=9mm]{\tiny $B_b \mathcal{O}_{9}$};
\end{scope}

\begin{scope}[xshift=4*\s cm] 

 \node[align=center, text width=1.6cm,,,text opacity = 0] (c1) {\href{https://www.geogebra.org/m/x4uzmjfz}{a b c d e f g h i j k l m n}};

\def\IndexArrayI {{3,4}}

\def\IndexArrayII {{1,3}}

\def\IndexArrayIII {{1,2}}

\pgfmathtruncatemacro{\dimL}{dim{\IndexArrayI}-1} \dimL

\foreach \a in {0,...,\dimL} {

\pgfmathsetmacro{\A}{\IndexArrayI[\a]-1} 

\tikzset{rotate=\A*\th}

\draw[blue!80!black]
	(0,0) -- (2*\x,0) 
;
}

\pgfmathtruncatemacro{\dimL}{dim{\IndexArrayII}-1} \dimL

\foreach \a in {0,...,\dimL} {

\pgfmathsetmacro{\A}{\IndexArrayII[\a]-1} 

\tikzset{rotate=\A*\th}

\draw[blue!80!black]
	(2*\x,0) arc (0:90:2*\x)
;
}

\pgfmathtruncatemacro{\dimL}{dim{\IndexArrayIII}-1} \dimL

\foreach \a in {0,...,\dimL} {

\pgfmathsetmacro{\A}{\IndexArrayIII[\a]-1} 

\tikzset{rotate=\A*\th}

\draw[blue!80!black]
	(2*\x,0) -- (4*\x,0)
;
}
\node (s1c) {} node[below=9mm]{\tiny $B_b \mathcal{O}_{10}$};
\end{scope}

\end{scope}

\begin{scope}[yshift=-2*\s cm] 

\begin{scope}[] 

 \node[align=center, text width=1.6cm,,,text opacity = 0] (c1) {\href{https://www.geogebra.org/m/jgcd6gzk}{a b c d e f g h i j k l m n}};

\def\IndexArrayI {{3,4}}

\def\IndexArrayII {{1,2,3,4}}

\def\IndexArrayIII {{3,4}}

\pgfmathtruncatemacro{\dimL}{dim{\IndexArrayI}-1} \dimL

\foreach \a in {0,...,\dimL} {

\pgfmathsetmacro{\A}{\IndexArrayI[\a]-1} 

\tikzset{rotate=\A*\th}

\draw[blue!80!black]
	(0,0) -- (2*\x,0) 
;
}

\pgfmathtruncatemacro{\dimL}{dim{\IndexArrayII}-1} \dimL

\foreach \a in {0,...,\dimL} {

\pgfmathsetmacro{\A}{\IndexArrayII[\a]-1} 

\tikzset{rotate=\A*\th}

\draw[blue!80!black]
	(2*\x,0) arc (0:90:2*\x)
;
}

\pgfmathtruncatemacro{\dimL}{dim{\IndexArrayIII}-1} \dimL

\foreach \a in {0,...,\dimL} {

\pgfmathsetmacro{\A}{\IndexArrayIII[\a]-1} 

\tikzset{rotate=\A*\th}

\draw[blue!80!black]
	(2*\x,0) -- (4*\x,0)
;
}
\node (s1c) {} node[below=9mm]{\tiny $B_b \mathcal{O}_{11}$};
\end{scope}

\begin{scope}[xshift=\s cm] 

 \node[align=center, text width=1.6cm,,,text opacity = 0] (c1) {\href{https://www.geogebra.org/m/auasachy}{a b c d e f g h i j k l m n}};

\def\IndexArrayI {{2,4}}

\def\IndexArrayII {{1,2,3,4}}

\def\IndexArrayIII {{2,4}}

\pgfmathtruncatemacro{\dimL}{dim{\IndexArrayI}-1} \dimL

\foreach \a in {0,...,\dimL} {

\pgfmathsetmacro{\A}{\IndexArrayI[\a]-1} 

\tikzset{rotate=\A*\th}

\draw[blue!80!black]
	(0,0) -- (2*\x,0) 
;
}

\pgfmathtruncatemacro{\dimL}{dim{\IndexArrayII}-1} \dimL

\foreach \a in {0,...,\dimL} {

\pgfmathsetmacro{\A}{\IndexArrayII[\a]-1} 

\tikzset{rotate=\A*\th}

\draw[blue!80!black]
	(2*\x,0) arc (0:90:2*\x)
;
}

\pgfmathtruncatemacro{\dimL}{dim{\IndexArrayIII}-1} \dimL

\foreach \a in {0,...,\dimL} {

\pgfmathsetmacro{\A}{\IndexArrayIII[\a]-1} 

\tikzset{rotate=\A*\th}

\draw[blue!80!black]
	(2*\x,0) -- (4*\x,0)
;
}
\node (s1c) {} node[below=9mm]{\tiny $B_b \mathcal{O}_{12}$};
\end{scope}

\begin{scope}[xshift=2*\s cm] 

 \node[align=center, text width=1.6cm,,,text opacity = 0] (c1) {\href{}{a b c d e f g h i j k l m n}};

\end{scope}

\begin{scope}[xshift=3*\s cm] 






















\end{scope}

\end{scope}


\foreach \xs in {0,...,4} {

\foreach \ys in {0,...,1} {

\tikzset{xshift=\xs*\s cm, yshift=-\ys*\s cm}

\foreach \a in {0,...,3} {

\tikzset{rotate=\a*\th}

\draw[red]
	(0,0) -- (\x,\x)
	(0,2*\x) -- (\x,\x)
	(0,2*\x) -- (-\x,\x)
	(0,2*\x) -- (3*\x,3*\x)
	(0,2*\x) -- (-3*\x,3*\x)
	(3*\x,3*\x) -- (4*\x,4*\x)
;

}

\foreach \a in {0,...,3} {

\tikzset{rotate=\a*\th}
	
\draw[gray!75, double, line width=0.5]
	(\x,\x) -- (-\x,\x)
	(\l*\x,\l*\x) -- (-\l*\x,\l*\x)
	(\x,\x) -- (\l*\x,\l*\x) 
;

}

\foreach \a in {0,...,3} {

\tikzset{rotate=\a*\th}

	(0,0) -- (2*\x,0) 
	(2*\x,0) -- (4*\x,0)
	(2*\x,0) arc (0:90:2*\x)
;

}


\filldraw[blue!80!black] 
	(0,0) circle (\rr)
	(2*\x,0) circle (\rr)
	(0,2*\x) circle (\rr)
	(-2*\x,0) circle (\rr)
	(0,-2*\x) circle (\rr)
;

}
}


\foreach \xs in {0,...,1} {

\foreach \ys in {2} {

\tikzset{xshift=\xs*\s cm, yshift=-\ys*\s cm}

\foreach \a in {0,...,3} {

\tikzset{rotate=\a*\th}

\draw[red]
	(0,0) -- (\x,\x)
	(0,2*\x) -- (\x,\x)
	(0,2*\x) -- (-\x,\x)
	(0,2*\x) -- (3*\x,3*\x)
	(0,2*\x) -- (-3*\x,3*\x)
	(3*\x,3*\x) -- (4*\x,4*\x)
;

}

\foreach \a in {0,...,3} {

\tikzset{rotate=\a*\th}
	
\draw[gray!75, double, line width=0.5]
	(\x,\x) -- (-\x,\x)
	(\l*\x,\l*\x) -- (-\l*\x,\l*\x)
	(\x,\x) -- (\l*\x,\l*\x) 
;

}

\foreach \a in {0,...,3} {

\tikzset{rotate=\a*\th}

	(0,0) -- (2*\x,0) 
	(2*\x,0) -- (4*\x,0)
	(2*\x,0) arc (0:90:2*\x)
;

}


\filldraw[blue!80!black] 
	(0,0) circle (\rr)
	(2*\x,0) circle (\rr)
	(0,2*\x) circle (\rr)
	(-2*\x,0) circle (\rr)
	(0,-2*\x) circle (\rr)
;

}
}

\end{tikzpicture}
\caption{The $b$-edge assignments in $B\mathcal{O}$ up to isomorphism}
\label{Fig:BP8-b}
\end{figure}


\begin{figure}[h!] 
\centering
\begin{tikzpicture}

\tikzmath{
\s=2.4;
\x=0.25;
\l=3;
\th=360/4;
\rr=0.15*\x;
}

\begin{scope}[]

\begin{scope}[]

 \node[align=center, text width=1.6cm,,,text opacity = 0] (c1) {\href{https://www.geogebra.org/m/wbgykqcw}{a b c d e f g h i j k l m n}};

\def\IndexArrayII {{1,2,3,4}}

\pgfmathtruncatemacro{\dimL}{dim{\IndexArrayII}-1} \dimL

\foreach \a in {0,...,\dimL} {

\pgfmathsetmacro{\A}{\IndexArrayII[\a]-1} 

\tikzset{rotate=\A*\th}

\draw[gray!75, double, line width=0.5]
	(\x,\x) -- (\l*\x,\l*\x) 
;
}
\node (s1c) {} node[below=9mm]{\tiny $B_a \mathcal{O}_{1}$};
\end{scope}

\begin{scope}[xshift=\s cm] 

 \node[align=center, text width=1.6cm,,,text opacity = 0] (c1) {\href{https://www.geogebra.org/m/tawcwe7u}{a b c d e f g h i j k l m n}};

\def\IndexArrayII {{1,2,3,4}}

\def\IndexArrayIII {{1,2,3,4}}

\pgfmathtruncatemacro{\dimL}{dim{\IndexArrayII}-1} \dimL

\foreach \a in {0,...,\dimL} {

\pgfmathsetmacro{\A}{\IndexArrayII[\a]-1} 

\tikzset{rotate=\A*\th}

\draw[gray!75, double, line width=0.5]
	(\x,\x) -- (\l*\x,\l*\x) 
;
}

\pgfmathtruncatemacro{\dimL}{dim{\IndexArrayIII}-1} \dimL

\foreach \a in {0,...,\dimL} {

\pgfmathsetmacro{\A}{\IndexArrayIII[\a]-1} 

\tikzset{rotate=\A*\th}

\draw[gray!75, double, line width=0.5]
	(\l*\x,\l*\x) -- (\l*\x,-\l*\x)
;
}
\node (s1c) {} node[below=9mm]{\tiny $B_a \mathcal{O}_{2}$};
\end{scope}

\begin{scope}[xshift=2*\s cm] 

 \node[align=center, text width=1.6cm,,,text opacity = 0] (c1) {\href{https://www.geogebra.org/m/yjfspx7y}{a b c d e f g h i j k l m n}};

\def\IndexArrayI {{4}}

\def\IndexArrayII {{1,2}}

\def\IndexArrayIII {{4}}

\pgfmathtruncatemacro{\dimL}{dim{\IndexArrayI}-1} \dimL

\foreach \a in {0,...,\dimL} {

\pgfmathsetmacro{\A}{\IndexArrayI[\a]-1} 

\tikzset{rotate=\A*\th}

\draw[gray!75, double, line width=0.5]
	(\x,\x) -- (\x,-\x)
;
}

\pgfmathtruncatemacro{\dimL}{dim{\IndexArrayII}-1} \dimL

\foreach \a in {0,...,\dimL} {

\pgfmathsetmacro{\A}{\IndexArrayII[\a]-1} 

\tikzset{rotate=\A*\th}

\draw[gray!75, double, line width=0.5]
	(\x,\x) -- (\l*\x,\l*\x) 
;
}

\pgfmathtruncatemacro{\dimL}{dim{\IndexArrayIII}-1} \dimL

\foreach \a in {0,...,\dimL} {

\pgfmathsetmacro{\A}{\IndexArrayIII[\a]-1} 

\tikzset{rotate=\A*\th}

\draw[gray!75, double, line width=0.5]
	(\l*\x,\l*\x) -- (\l*\x,-\l*\x)
;
}
\node (s1c) {} node[below=9mm]{\tiny $B_a \mathcal{O}_{3}$};
\end{scope}

\begin{scope}[xshift=3*\s cm] 

 \node[align=center, text width=1.6cm,,,text opacity = 0] (c1) {\href{https://www.geogebra.org/m/emudmmxf}{a b c d e f g h i j k l m n}};

\def\IndexArrayI {{4}}

\def\IndexArrayII {{1,2}}

\def\IndexArrayIII {{1,2,3}}

\pgfmathtruncatemacro{\dimL}{dim{\IndexArrayI}-1} \dimL

\foreach \a in {0,...,\dimL} {

\pgfmathsetmacro{\A}{\IndexArrayI[\a]-1} 

\tikzset{rotate=\A*\th}

\draw[gray!75, double, line width=0.5]
	(\x,\x) -- (\x,-\x)
;
}

\pgfmathtruncatemacro{\dimL}{dim{\IndexArrayII}-1} \dimL

\foreach \a in {0,...,\dimL} {

\pgfmathsetmacro{\A}{\IndexArrayII[\a]-1} 

\tikzset{rotate=\A*\th}

\draw[gray!75, double, line width=0.5]
	(\x,\x) -- (\l*\x,\l*\x) 
;
}

\pgfmathtruncatemacro{\dimL}{dim{\IndexArrayIII}-1} \dimL

\foreach \a in {0,...,\dimL} {

\pgfmathsetmacro{\A}{\IndexArrayIII[\a]-1} 

\tikzset{rotate=\A*\th}

\draw[gray!75, double, line width=0.5]
	(\l*\x,\l*\x) -- (\l*\x,-\l*\x)
;
}
\node (s1c) {} node[below=9mm]{\tiny $B_a \mathcal{O}_{4}$};
\end{scope}

\begin{scope}[xshift=4*\s cm] 

 \node[align=center, text width=1.6cm,,,text opacity = 0] (c1) {\href{https://www.geogebra.org/m/hahctbck}{a b c d e f g h i j k l m n}};

\def\IndexArrayI {{3,4}}

\def\IndexArrayII {{1,3}}

\def\IndexArrayIII {{1,2}}

\pgfmathtruncatemacro{\dimL}{dim{\IndexArrayI}-1} \dimL

\foreach \a in {0,...,\dimL} {

\pgfmathsetmacro{\A}{\IndexArrayI[\a]-1} 

\tikzset{rotate=\A*\th}

\draw[gray!75, double, line width=0.5]
	(\x,\x) -- (\x,-\x)
;
}

\pgfmathtruncatemacro{\dimL}{dim{\IndexArrayII}-1} \dimL

\foreach \a in {0,...,\dimL} {

\pgfmathsetmacro{\A}{\IndexArrayII[\a]-1} 

\tikzset{rotate=\A*\th}

\draw[gray!75, double, line width=0.5]
	(\x,\x) -- (\l*\x,\l*\x)
;
}

\pgfmathtruncatemacro{\dimL}{dim{\IndexArrayIII}-1} \dimL

\foreach \a in {0,...,\dimL} {

\pgfmathsetmacro{\A}{\IndexArrayIII[\a]-1} 

\tikzset{rotate=\A*\th}

\draw[gray!75, double, line width=0.5]
	(\l*\x,\l*\x) -- (\l*\x,-\l*\x)
;
}
\node (s1c) {} node[below=9mm]{\tiny $B_a \mathcal{O}_{5}$};
\end{scope}

\begin{scope}[xshift=5*\s cm]  






















\end{scope}

\end{scope}


\foreach \xs in {0,...,4} {

\foreach \ys in {0} {

\tikzset{xshift=\xs*\s cm, yshift=-\ys*\s cm}

\foreach \a in {0,...,3} {

\tikzset{rotate=\a*\th}

\draw[red]
	(0,0) -- (\x,\x)
	(0,2*\x) -- (\x,\x)
	(0,2*\x) -- (-\x,\x)
	(0,2*\x) -- (3*\x,3*\x)
	(0,2*\x) -- (-3*\x,3*\x)
	(3*\x,3*\x) -- (4*\x,4*\x)
;

}

\foreach \a in {0,...,3} {

\tikzset{rotate=\a*\th}
	
	(\x,\x) -- (-\x,\x)
	(\l*\x,\l*\x) -- (-\l*\x,\l*\x)
	(\x,\x) -- (\l*\x,\l*\x) 
;

}

\foreach \a in {0,...,3} {

\tikzset{rotate=\a*\th}

\draw[blue!80!black]
	(0,0) -- (2*\x,0) 
	(2*\x,0) -- (4*\x,0)
	(2*\x,0) arc (0:90:2*\x)
;

}


\filldraw[blue!80!black] 
	(0,0) circle (\rr)
	(2*\x,0) circle (\rr)
	(0,2*\x) circle (\rr)
	(-2*\x,0) circle (\rr)
	(0,-2*\x) circle (\rr)
;

}
}

\end{tikzpicture}
\caption{The $a$-edge assignments in $B\mathcal{O}$ up to isomorphism}
\label{Fig:BP8-a}
\end{figure}


\begin{figure}[h!] 
\centering
\begin{tikzpicture}

\tikzmath{
\s=2.4;
\x=0.25;
\l=3;
\th=360/4;
\rr=0.15*\x;
}

\begin{scope}[]

\begin{scope}[]

 \node[align=center, text width=1.6cm,,,text opacity = 0] (c1) {\href{https://www.geogebra.org/m/gdn54tej}{a b c d e f g h i j k l m n}};
 
\node (s1c) {} node[below=9mm]{\tiny $FB_a \mathcal{O}_{1}$};
\end{scope}

\begin{scope}[xshift=\s cm] 

 \node[align=center, text width=1.6cm,,,text opacity = 0] (c1) {\href{https://www.geogebra.org/m/csdfxy4n}{a b c d e f g h i j k l m n}};

\def\IndexArrayIV {{1,2,3,4}}

\pgfmathtruncatemacro{\dimL}{dim{\IndexArrayIV}-1} \dimL

\foreach \a in {0,...,\dimL} {

\pgfmathsetmacro{\A}{\IndexArrayIV[\a]-1} 

\tikzset{rotate=\A*\th}

\draw[gray!75, double, line width=0.5]
	(\l*\x,\l*\x) -- (\l*\x,-\l*\x)
;
}
\node (s1c) {} node[below=9mm]{\tiny $FB_a \mathcal{O}_{2}$};
\end{scope}

\begin{scope}[xshift=2*\s cm] 















\end{scope}

\end{scope}


\foreach \xs in {0,...,1} {

\foreach \ys in {0} {

\tikzset{xshift=\xs*\s cm, yshift=-\ys*\s cm}

\foreach \a in {0,...,3} {

\tikzset{rotate=\a*\th}

\draw[red]
	(0,0) -- (\x,0)
;

\draw[red]
	(\x,0) -- (45:2*\x)
	(0,\x) -- (45:2*\x)
	(0,2*\x) -- (\l*\x,\l*\x)
	(2*\x,0) -- (\l*\x,\l*\x) 
;

\draw[red]
	(0,0) -- (\x,0)
	%
	%
	(0,2*\x) -- (3*\x,3*\x)
	(0,2*\x) -- (-3*\x,3*\x)
	(3*\x,3*\x) -- (4*\x,4*\x)
;

}

\foreach \a in {0,1,2,3} {

\tikzset{rotate=\a*\th}
	
	(\x,0) -- (0,-\x)
	%
	%
	%
	(\l*\x,\l*\x) -- (\l*\x,-\l*\x)
;

\draw[gray!75, double, line width=0.5]
	(\x,0) -- (2*\x,0)
	(45:2*\x) -- (\l*\x,\l*\x) 	
;

}

\foreach \a in {0,1,2,3} {

\tikzset{rotate=\a*\th}

\draw[blue!80!black]
	(0,0) -- (45:2*\x) 
	(2*\x,0) -- (4*\x,0)
;

\filldraw[blue!80!black] 
	(2*\x,0) circle (\rr)
	(45:2*\x) circle (\rr)
;
}

\filldraw[blue!80!black] 
	(0,0) circle (\rr)
;

\draw[blue!80!black] (0,0) circle (2*\x);

}
}

\end{tikzpicture}
\caption{The $a$-edge assignments in $FB\mathcal{O}$ up to isomorphism}
\label{Fig:FBP8-a}
\end{figure}


\clearpage 

{
Using $\sigma_1=I$ and the following generators of the octahedral group $O_h$,
\begin{equation*}
\begin{alignedat}{3} 
&\sigma_2 = \begin{bmatrix}[r]
\,\,\,1 & 0 & 0 \\
0 & \,\,\,\,\,1 & 0 \\
0 & 0 & \,\,-1
\end{bmatrix},  ~~
&\sigma_3 = \begin{bmatrix}[r]
1 & \,\,\,\,\,0 & 0 \\
\,\,\,0 & 0 & 1 \\
0 & 1 & \,\,\,\,0
\end{bmatrix}, ~~
&\sigma_4 = \begin{bmatrix}[r]
\,\,\,0 & \,\,\, 1 & 0 \\
1 & 0 & 0 \\
0 & 0 & \,\,\,1
\end{bmatrix},
\end{alignedat}
\end{equation*}
we represent below each automorphism $\pm \sigma_k \in O_h$ for $k = 5,  \ldots, 24$},
{
\begin{align} \label{Eq-Oh}
&\sigma_5 = \sigma_2 \sigma_3,&
&\sigma_6 = \sigma_2 \sigma_4,&
&\sigma_7 = \sigma_3 \sigma_2,&
&\sigma_8 = \sigma_3 \sigma_4,& \\ \notag
&\sigma_9 = \sigma_4 \sigma_3,& 
&\sigma_{10} = \sigma_2 \sigma_3 \sigma_2,&
&\sigma_{11} = \sigma_3 \sigma_2 \sigma_3,&
&\sigma_{12} = \sigma_3 \sigma_4 \sigma_3,& \\ \notag
&\sigma_{13} = \sigma_2 \sigma_3 \sigma_4,&
&\sigma_{14} = \sigma_2 \sigma_4 \sigma_3,&
&\sigma_{15} = \sigma_3 \sigma_2 \sigma_4,& 
&\sigma_{16} = \sigma_4 \sigma_3 \sigma_2,& \\ \notag
&\sigma_{17} = \sigma_2 \sigma_3 \sigma_2 \sigma_3,&
&\sigma_{18} = \sigma_3 \sigma_4 \sigma_3 \sigma_2,&
&\sigma_{19} = \sigma_4 \sigma_3 \sigma_2 \sigma_3,&
&\sigma_{20} = \sigma_3 \sigma_2 \sigma_4 \sigma_3,& \\ \notag
&\sigma_{21} = \sigma_3 \sigma_2 \sigma_3 \sigma_4,&
&\sigma_{22} = \sigma_2 \sigma_3 \sigma_2 \sigma_4,&
&\sigma_{23} = \sigma_2 \sigma_4 \sigma_3 \sigma_2,&
&\sigma_{24} = \sigma_4 \sigma_2 \sigma_3 \sigma_4.&
\end{align}
}
{Similarly for $G' = D_8$, using $\sigma_1'=I$ and
\begin{equation*}
\begin{alignedat}{3} 
&\sigma_2' = \begin{bmatrix}[r]
\tfrac{1}{\sqrt{2}} & \tfrac{1}{\sqrt{2}}  & 0 \\
\tfrac{-1}{\sqrt{2}}  &\,\,\, \tfrac{1}{\sqrt{2}} & 0 \\
0 & 0 & -1
\end{bmatrix}, ~~
&\sigma' = \begin{bmatrix}[r]
1 & \,\,\,\,\,0 & 0 \\
\,\,\,0 & -1 & 0 \\
0 & 0 & \,\,\,\,1
\end{bmatrix},
\end{alignedat}
\end{equation*}
we represent below each $\sigma'_k \in G'$ for $k = 1,  \ldots,16$,
\begin{align*}
\sigma'_k  &= (\sigma_2')^{k-1},  ~~~k = 1,  \ldots, 8, \\ 
\sigma'_k  &= \sigma' (\sigma_2')^{k-9},  ~k = 9,  \ldots, 16.
\end{align*}
}
{The data of the f-tilings are provided in Tables \ref{combstruct1}-\ref{combstruct5}.}

{\begin{table}[h!]
\begin{center}
\bgroup \scriptsize 
\def\arraystretch{1.5}
    \begin{tabular}[]{ | c | c | c | c | c |} 
		\hline
		\multirow{2}{*}{f-tilings} & \multirow{2}{*}{\parbox[c]{2cm}{\centering Symmetry group}} & \multicolumn{2}{c |}{Prototiles} & \multirow{2}{*}{\parbox[c]{2cm}{\centering Generators (Determinants)}} \\
		\cline{3-4}
		 &  & $\# \triangle abc$ & $\# \square a^2b^2$ &  \\
		 \hhline{|=====|} 
$B_c \mathcal{O}_3$
 & \multirow{5}{*}{$C_1$} & 16 & 16 & \multirow{5}{*}{$\sigma_1$ $(+)$}  \\  \cline{1-1} \cline{3-4}
$B_c \mathcal{O}_{k}$, $k=14,...,21,23$ &  & 20 & 14 &  \\ \cline{1-1} \cline{3-4}
$B_c \mathcal{O}_{k}$, $k=25,...,27,39,...,41,44$ & & 24 & 12 &  \\ \cline{1-1} \cline{3-4}
$B_c \mathcal{O}_{k}$, $k=46,47,58,59,61,65$ & & 28 & 10 &  \\ \cline{1-1} \cline{3-4}
$B_c \mathcal{O}_{67}$  & & {32} & {8} & \\ \hline
	\end{tabular}
\egroup
\end{center}
\vspace*{-0.6cm} 
\caption{{Data of $B_c \mathcal{O}_k$,  $k = 1,  \ldots, 75$}}
\label{combstruct1}
\end{table}}

\addtocounter{table}{-1}

{\begin{table}[h!]
\begin{center}
\bgroup \scriptsize 
\def\arraystretch{1.5}
    \begin{tabular}[]{ | c | c | c | c | c |} 
		\hline
		\multirow{2}{*}{f-tilings} & \multirow{2}{*}{\parbox[c]{2cm}{\centering Symmetry group}} & \multicolumn{2}{c |}{Prototiles} & \multirow{2}{*}{\parbox[c]{2cm}{\centering Generators (Determinants)}} \\
		\cline{3-4}
		 &  & $\# \triangle abc$ & $\# \square a^2b^2$ &  \\
		 \hhline{|=====|} 
$B_c \mathcal{O}_4, B_c \mathcal{O}_5$ & \multirow{18}{*}{$C_2$} & \multirow{4}{*}{16} & \multirow{4}{*}{16} & {$\sigma_{20}$ $(+)$} \\ \cline{1-1} \cline{5-5}
{$B_c \mathcal{O}_6$}  & &  & &  {{$\sigma_{2}$ $(-)$}} \\
\cline{1-1} \cline{5-5}
$B_c \mathcal{O}_7$ & &  & &  {{$-\sigma_{4}$ $(+)$}} \\
\cline{1-1} \cline{5-5}
$B_c \mathcal{O}_9$  & &  &  & {{$\sigma_{11}$ $(-)$}} \\ \cline{1-1} \cline{3-5}
$B_c \mathcal{O}_{22}$  & &  \multirow{2}{*}{20} & \multirow{2}{*}{14} &  {{$-\sigma_{20}$ $(-)$}} \\ \cline{1-1} \cline{5-5}
$B_c \mathcal{O}_{24}$  & & & & {$\sigma_{3}$ $(-)$} \\ \cline{1-1} \cline{3-5}
$B_c \mathcal{O}_{k}$, $k=28,...,31,35$ & & \multirow{6}{*}{24} & \multirow{6}{*}{12} &  {{$-\sigma_{17}$ $(-)$}} \\ \cline{1-1} \cline{5-5}
$B_c \mathcal{O}_{32}$, $B_c \mathcal{O}_{33}$  & &  &  & {{$-\sigma_{10}$ $(+)$}}  \\ \cline{1-1} \cline{5-5}
$B_c \mathcal{O}_{38}$, $B_c \mathcal{O}_{42}$,  $B_c \mathcal{O}_{45}$  & & &  & {{$-\sigma_{2}$ $(+)$}} \\ \cline{1-1} \cline{5-5}
$B_c \mathcal{O}_{43}$ & & &  & {{$-\sigma_{6}$ $(-)$}} \\ \cline{1-1} \cline{5-5}
$B_c \mathcal{O}_{53}$  & &  &  &  {{$\sigma_{6}$ $(+)$}} \\ \cline{1-1} \cline{5-5}
$B_c \mathcal{O}_{54}$  & &  &  & {{$-\sigma_{12}$ $(+)$}} \\ \cline{1-1} \cline{3-5} 
$B_c \mathcal{O}_{48}$  & & \multirow{2}{*}{28} & \multirow{2}{*}{10} &   {{$\sigma_{4}$ $(-)$}} \\ \cline{1-1} \cline{5-5}
$B_c \mathcal{O}_{60}$, $B_c \mathcal{O}_{62}$, $B_c \mathcal{O}_{66}$ & & & & {{$-\sigma_{6}$ $(-)$}} \\ \cline{1-1} \cline{3-5} 
$B_c \mathcal{O}_{50}$  & & \multirow{3}{*}{32} & \multirow{3}{*}{8} & {{$-\sigma_{17}$ $(-)$}} \\ \cline{1-1} \cline{5-5}
$B_c \mathcal{O}_{63}$, $B_c \mathcal{O}_{64}$  & &  & & {{$-\sigma_{10}$ $(+)$}} \\ \cline{1-1} \cline{5-5}
$B_c \mathcal{O}_{68}$  & &   &  &  {{$-\sigma_{6}$ $(-)$}} \\ \cline{1-1} \cline{3-5} 
$B_c \mathcal{O}_{69}$ & & {36} & {6}  & {{$-\sigma_{6}$ $(-)$}} \\ \hline
$B_c \mathcal{O}_{57}$  & $C_6$ &  24 & 12 & $\sigma_{16}$ $(-)$ \\ \hline
{$B_c \mathcal{O}_8$}  &  \multirow{9}{*}{$C_2 \times C_2$} & \multirow{3}{*}{16} & \multirow{3}{*}{16} & 
{$\sigma_{11}$},  {$\sigma_{12}$} $(-,-)$ \\ \cline{1-1} \cline{5-5}
{$B_c \mathcal{O}_{10}$}  &  &  &  & 
{$\sigma_{2}$},  {$\sigma_{11}$} $(-,-)$ \\ \cline{1-1} \cline{5-5}
{$B_c \mathcal{O}_{11}$}  &  &  &  & 
{$\pm \sigma_{11}$} $(-,+)$ \\ \cline{1-1} \cline{3-5} 
{$B_c \mathcal{O}_{34}$,  $B_c \mathcal{O}_{36}$, $B_c \mathcal{O}_{37}$}  & &  \multirow{2}{*}{24} & \multirow{2}{*}{12} & {$\sigma_{3}$}, {$-\sigma_{17}$} $(-,-)$ \\ \cline{1-1} \cline{5-5}
{$B_c \mathcal{O}_{52}$}  &  &   &  & 
{$-\sigma_{1}$}, {$\sigma_{12}$} $(-,-)$ \\
\cline{1-1} \cline{3-5}
{$B_c \mathcal{O}_{49}$}  &  & \multirow{3}{*}{32} &  \multirow{3}{*}{8}& {$-\sigma_{2}$}, {$-\sigma_{17}$} $(+,-)$ \\ \cline{1-1} \cline{5-5}
{$B_c \mathcal{O}_{70}$}  &  &  &  & 
{$-\sigma_{3}$},  {$-\sigma_{17}$} $(+,-)$ \\ \cline{1-1} \cline{5-5}
{$B_c \mathcal{O}_{71}$}  &  &  &  & 
{$-\sigma_{1}$},   {$-\sigma_{17}$}, $(-,-)$ \\  \cline{1-1} \cline{3-5} 
{$B_c \mathcal{O}_{73}$}  & & 40 & 4 & 
{$\sigma_{3}$},  {$-\sigma_{17}$} $(-,-)$ \\ \hline
{$B_c \mathcal{O}_1$}  & {$C_2 \times C_4$} & 16 & 16 &
{$\sigma_{2}$},  {$\sigma_{19}$} $(-,+)$ \\ \hline
{$B_c \mathcal{O}_{72}$} & {$C_2 \times C_2 \times C_2$} & 32 & 8 & 
{$-\sigma_1$},  {$\sigma_3$},  {$-\sigma_{17}$} $(-, +, -)$\\ \hline
{$B_c \mathcal{O}_{56}$}  & \multirow{2}{*}{$D_3$} &  {24} & {12} & 
{$\sigma_{6}$},  {$-\sigma_{16}$} $(+,+)$ \\ 
\cline{1-1} \cline{3-5}
{$B_c \mathcal{O}_{75}$}  & & 36 & 6 & 
{$-\sigma_{6}$},  {$-\sigma_{14}$} $(-,+)$ \\
\hline
{$B_c \mathcal{O}_2$}  &  \multirow{4}{*}{$D_4$} &  \multirow{2}{*}{16} &  \multirow{2}{*}{16} & {$\sigma_{6}$},  {$\sigma_{19}$} $(+,+)$ \\
\cline{1-1} \cline{5-5}
{$B_c \mathcal{O}_{13}$}  &  &&  &  {$\sigma_{6}$},  {$-\sigma_{21}$} $(+,-)$ \\
\cline{1-1} \cline{3-5}
{$B_c \mathcal{O}_{51}$}  &  & \multirow{2}{*}{32} &  \multirow{2}{*}{8}  & {$\sigma_4$},  {$\sigma_{19}$} $(-,+)$ \\ 
\cline{1-1} \cline{5-5}
{$B_c \mathcal{O}_{74}$}  & && & {$-\sigma_{2}$},  {$-\sigma_{24}$} $(-,+)$ \\ \hline
	\end{tabular}
\egroup
\end{center}
\vspace*{-0.6cm} 
\caption{{Data of $B_c \mathcal{O}_k$,  $k = 1,  \ldots, 75$ (cont.)}}
\end{table}}

\addtocounter{table}{-1}

{\begin{table}[h!]
\begin{center}
\bgroup \scriptsize 
\def\arraystretch{1.5}
    \begin{tabular}[]{ | c | c | c | c | c |} 
		\hline
		\multirow{2}{*}{f-tilings} & \multirow{2}{*}{\parbox[c]{2cm}{\centering Symmetry group}} & \multicolumn{2}{c |}{Prototiles} & \multirow{2}{*}{\parbox[c]{2cm}{\centering Generators (Determinants)}} \\
		\cline{3-4}
		 &  & $\# \triangle abc$ & $\# \square a^2b^2$ &  \\
		 \hhline{|=====|} 
$B_c \mathcal{O}_{55}$  & $D_6$ & 24 & 12 & 
{$-\sigma_{6}$},  {$\sigma_{16}$} $(-,-)$ \\ \hline
$B_c \mathcal{O}_{12}$  & $C_2 \times D_4$ &  16 & 16 &  
{$\sigma_{2}$},  {$\sigma_{11}$},   {$\sigma_{18}$} $(-,-,-)$ \\ 
		\hline
	\end{tabular}
\egroup
\end{center}
\vspace*{-0.6cm} 
\caption{{Data of $B_c \mathcal{O}_k$,  $k = 1,  \ldots, 75$} (cont.)}
\end{table}}

{\begin{table}[h!]
\begin{center}
\bgroup \scriptsize 
\def\arraystretch{1.5}
    \begin{tabular}[]{ | c | c | c | c | c |} 
		\hline
		\multirow{2}{*}{f-tilings} & \multirow{2}{*}{\parbox[c]{2cm}{\centering Symmetry group}} & \multicolumn{2}{c |}{Prototiles} & \multirow{2}{*}{\parbox[c]{2cm}{\centering Generators (Determinants)}} \\
		\cline{3-4}
		 &  & $\# \triangle abc$ & $\# \triangle \bar{b}c^2$ &  \\
		 \hhline{|=====|} 
$B_b \mathcal{O}_4$ & \multirow{2}{*}{$C_2$} &  20 & 14 &
\multirow{2}{*}{$\sigma_{4}$ $(-)$} \\ 
\cline{1-1} \cline{3-4}
$B_b \mathcal{O}_8$  & & 28 & 10 & \\ 
\hline
$B_b \mathcal{O}_{9}$  & $C_4$ & 24 & 12 & 
{$-\sigma_{5}$} $(-)$ \\ \hline
$B_b \mathcal{O}_2$  &   \multirow{3}{*}{$C_2 \times C_2$} & 16  &  16 & 
{$\sigma_{3}$},  {$-\sigma_{17}$} $(+,-)$ \\ \cline{1-1} \cline{3-5}
{$B_b \mathcal{O}_{3}$}  &  & 24  &  12 &
{$\sigma_{4}$},  {$-\sigma_{6}$} $(-,-)$ \\ 
\cline{1-1} \cline{3-5}
{$B_b \mathcal{O}_{11}$}  &   & 32  &  8 & 
{$\sigma_{2}$}, {$\sigma_{4}$} $(-,-)$ \\ 
\hline
{$B_b \mathcal{O}_{1}$}  & \multirow{2}{*}{$D_3$} &  12 & 8 &
{$-\sigma_{6}$},  {$-\sigma_{15}$} $(-,+)$ \\
\cline{1-1} \cline{3-5}
{$B_b \mathcal{O}_{7}$}  & & 36 & 6 &  
{$\sigma_{9}$},  {$\sigma_{12}$} $(+,-)$ \\  \hline
{$B_b \mathcal{O}_{6}$}  & \multirow{2}{*}{$D_6$} & \multirow{2}{*}{24} & \multirow{2}{*}{12} & 
{$\sigma_{4}$},  {$-\sigma_{8}$} $(-,+)$ \\
\cline{1-1} \cline{5-5}
{$B_b \mathcal{O}_{10}$}  & & & & 
{$\sigma_{4}$},  {$-\sigma_{22}$} $(-,-)$ \\
\hline
$B_b \mathcal{O}_{5}$ & \multirow{2}{*}{$C_2 \times D_4$}  & 16 & 16 &  
{$\sigma_{2}$},   {$-\sigma_{17}$},  {$\sigma_{19}$} $(-,-,+)$ \\
\cline{1-1} \cline{3-5}
{$B_b \mathcal{O}_{12}$}  & &  32 & 8 &
{$\sigma_{2}$},  {$\sigma_{11}$},  {$\sigma_{24}$} $(-,-,+)$ \\
\hline
	\end{tabular}
\egroup
\end{center}
\vspace*{-0.6cm} 
\caption{{Data of $B_b \mathcal{O}_k$, $k = 1,  \ldots, 12$}}
\label{combstruct2}
\end{table}}

{\begin{table}[h!]
\begin{center}
\bgroup \scriptsize 
\def\arraystretch{1.5}
    \begin{tabular}[]{ | c | c | c | c | c |} 
		\hline
		\multirow{2}{*}{f-tilings} & \multirow{2}{*}{\parbox[c]{2cm}{\centering Symmetry group}} & \multicolumn{2}{c |}{Prototiles} & \multirow{2}{*}{\parbox[c]{2cm}{\centering Generators (Determinants)}} \\
		\cline{3-4}
		 &  & $\# \triangle abc$ & $\# \triangle \bar{b}c^2$ &  \\
		 \hhline{|=====|} 
{$B_a \mathcal{O}_4$} & $C_2 \times C_2$  & 24 & 12 &
{$\sigma_{10}$},  {$-\sigma_{17}$} $(-,-)$ \\ 
\hline
$B_a \mathcal{O}_{3}$  & \multirow{2}{*}{$D_4$} & 16 & 16 &
{$\sigma_{2}$},  {$-\sigma_{18}$} $(-,+)$\\ 
\cline{1-1} \cline{3-5}
$B_a \mathcal{O}_{2}$  &  & 32 &  8 & 
{$\sigma_{11}$},  {$\sigma_{19}$} $(-,+)$ \\ 
\hline
{$B_a \mathcal{O}_{5}$}  & $D_6$ & 24 & 12 & {$\sigma_{4}$},  {$-\sigma_{23}$} $(-,-)$ \\ 
\hline
$B_a \mathcal{O}_{1}$  & $C_2 \times D_4$  & 16 & 16 &   
{$\sigma_{2}$},  {$-\sigma_{17}$},  {$\sigma_{19}$} $(-,-,+)$ \\
\hline
	\end{tabular}
\egroup
\end{center}
\vspace*{-0.6cm} 
\caption{{Data of $B_a \mathcal{O}_k$, $k = 1,  \ldots, 5$}}
\label{combstruct3}
\end{table}}

{\begin{table}[h!]
\begin{center}
\bgroup \scriptsize 
\def\arraystretch{1.5}
    \begin{tabular}[]{ | c | c | c | c | c |} 
		\hline
		\multirow{2}{*}{f-tilings} & \multirow{2}{*}{\parbox[c]{2cm}{\centering Symmetry group}} & \multicolumn{2}{c |}{Prototiles} & \multirow{2}{*}{\parbox[c]{2cm}{\centering Generators (Determinants)}} \\
		\cline{3-4}
		 &  & $\# \triangle abc$ & $\# \square a^2b^2$ &  \\
		 \hhline{|=====|} 
{$FB_c \mathcal{O}_2$}
 & \multirow{4}{*}{$C_1$} & 16 & 16 & \multirow{4}{*}{$\sigma'_1$ $(+)$}  \\  \cline{1-1} \cline{3-4}
{$FB_c \mathcal{O}_{9}$}  & & 20 & 14 &  \\ \cline{1-1} \cline{3-4}
$FB_c \mathcal{O}_{14}$,  $FB_c \mathcal{O}_{15}$,  $FB_c \mathcal{O}_{16}$ & & 24 & 12  &  \\ \cline{1-1} \cline{3-4}
$FB_c \mathcal{O}_{24}$ & & 28 & 10  &  \\ \hline
{$FB_c \mathcal{O}_4$}  &  \multirow{14}{*}{$C_2$} &  
\multirow{3}{*}{16}  & \multirow{3}{*}{16} & {$\sigma'_{11}$} $(-)$ \\
\cline{1-1} \cline{5-5}
$FB_c \mathcal{O}_5$ & &  & &  {$\sigma'_{16}$} $(-)$ \\
\cline{1-1} \cline{5-5}
{$FB_c \mathcal{O}_6$}  & &  &  &  {$\sigma'_{12}$} $(-)$ \\
\cline{1-1} \cline{3-5}
$FB_c \mathcal{O}_k$, $k=8,10,...,13$ & & 20 & 8 & $\sigma'_{15}$ $(+)$ \\ \cline{1-1} \cline{3-5}
{$FB_c \mathcal{O}_{17}$, $FB_c \mathcal{O}_{22}$
}  & &  \multirow{3}{*}{24} & \multirow{3}{*}{12} & {{$\sigma'_{13}$}} $(+)$ \\
\cline{1-1} \cline{5-5}
{$FB_c \mathcal{O}_{21}$
}  & &  &   & {{$\sigma'_{14}$}} $(-)$ \\
\cline{1-1} \cline{5-5}
{$FB_c \mathcal{O}_{23}$
}  & &  &   & {{$\sigma'_{10}$}} $(-)$ \\
\cline{1-1} \cline{3-5}
$FB_c \mathcal{O}_{25}$ &  & 28 & 10  & {$\sigma'_{13}$} $(+)$ \\
\cline{1-1} \cline{3-5}
{$FB_c \mathcal{O}_{19}$}  &  &  
 \multirow{3}{*}{32} &  \multirow{3}{*}{8} & {$\sigma'_{11}$} $(-)$ \\
\cline{1-1} \cline{5-5}
{$FB_c \mathcal{O}_{27}$}  & &  & &  {{$\sigma'_{12}$}} $(-)$ \\
\cline{1-1} \cline{5-5}
{$FB_c \mathcal{O}_{28}$}  & &  &  &  {{$\sigma'_{14}$}} $(-)$ \\
\cline{1-1} \cline{3-5}
{$FB_c \mathcal{O}_{26}$}  & &  36  & 6 &  {$\sigma'_{13}$} $(+)$ \\
\cline{1-1} \cline{3-5}
$FB_c \mathcal{O}_{29}$ &  & 40  & 4 & {$\sigma'_{11}$} $(-)$ \\ 
\hline
{$FB_c \mathcal{O}_{1}$}  &  \multirow{3}{*}{$C_2 \times C_2$}  &  
\multirow{2}{*}{16} & \multirow{2}{*}{16} & {$\sigma'_{5}$},  {$\sigma'_{11}$} $(+, -)$ \\
\cline{1-1} \cline{5-5}
{$FB_c \mathcal{O}_{3}$}  & &  & &  {$\sigma'_{5}$},  {$\sigma'_{12}$} $(+, -)$ \\
\cline{1-1} \cline{3-5}
$FB_c \mathcal{O}_{18}$  & & 32  & 8 & {$\sigma'_{5}$},  {$\sigma'_{11}$} $(+, -)$ \\
\hline
$FB_c \mathcal{O}_{20}$ & {$D_4$} & 32 & 8 & {$\sigma'_{3}$},  {$\sigma'_{11}$} $(+, -)$ \\
\hline
$FB_c \mathcal{O}_{7}$  & {$D_8$} & 16 & 16 & {$\sigma'_{6}$},   {$\sigma'_{11}$} $(-, -)$ \\
\hline
	\end{tabular}
\egroup
\end{center}
\vspace*{-0.6cm} 
\caption{{Data of $FB_c \mathcal{O}_k$, $k = 1,  \ldots, 29$}}
\label{combstruct4}
\end{table}}

{\begin{table}[h!]
\begin{center}
\bgroup \scriptsize 
\def\arraystretch{1.5}
    \begin{tabular}[]{ | c | c | c | c | c |} 
		\hline
		\multirow{2}{*}{f-tilings} & \multirow{2}{*}{\parbox[c]{2cm}{\centering Symmetry group}} & \multicolumn{2}{c |}{Prototiles} & \multirow{2}{*}{\parbox[c]{2cm}{\centering Generators (Determinants)}} \\
		\cline{3-4}
		 &  & $\# \triangle abc$ & $\# \triangle \bar{b}c^2$ &  \\
		 \hhline{|=====|} 
$FB_a \mathcal{O}_{2}$  & $D_4$ & 32 & 8 &  {$\sigma'_{7}$},  {$\sigma'_{13}$} $(+,+)$ \\  \hline
$FB_a \mathcal{O}_{1}$ & $D_8$ & 16 & 16 & {$\sigma'_8$},  {$\sigma'_{13}$} $(-,+)$ \\
\hline
	\end{tabular}
\egroup
\end{center}
\vspace*{-0.6cm} 
\caption{{Data of $FB_a \mathcal{O}_k$, $k = 1, 2$}}
\label{combstruct5}
\end{table}}

\end{document}